\documentclass[reqno]{amsart}

\usepackage[english]{babel}
\usepackage[utf8]{inputenc}

\usepackage{amsthm,amsfonts,amssymb,amsmath}
\numberwithin{equation}{section}
\usepackage{dsfont}
\usepackage{multirow}
\usepackage{url}
\usepackage{xcolor} 
\usepackage{graphics}
\usepackage{booktabs}
\usepackage{epsfig}
\usepackage{epstopdf}
\usepackage{psfrag}               % Modifikation von Postscriptgrafiken
\usepackage{mathrsfs}
\usepackage{mathtools}            % Fuer \shortintertext
\usepackage{subfigure}
\usepackage{longtable}	%Tabellen mit Seitenumbruch --> evtl wieder entfernen
\usepackage{nicefrac}
\usepackage{paralist}
\usepackage[inline]{enumitem}

\usepackage[pdftex,
pdftitle={Generalized Whittle--Matern fields: 
	regularity and approximation},
	pdfauthor={S.~G.~Cox and K.~Kirchner},
	bookmarksopen,
	colorlinks,
	linkcolor=black,
	urlcolor=black,
	citecolor=black
	]{hyperref}

\newcommand{\eps}{\varepsilon}

\newcommand{\bbC}{\mathbb{C}}
\newcommand{\bbE}{\mathbb{E}}
\newcommand{\E}{\mathbb{E}} 
\newcommand{\bbN}{\mathbb{N}}
\newcommand{\bbP}{\mathbb{P}}
\newcommand{\bbR}{\mathbb{R}}
\newcommand{\R}{\mathbb{R}}

\newcommand{\cB}{\mathcal{B}}
\newcommand{\cC}{\mathcal{C}}
\newcommand{\cD}{\mathcal{D}}
\newcommand{\cE}{\mathcal{E}}
\newcommand{\cF}{\mathcal{F}}
\newcommand{\cG}{\mathcal{G}}

\newcommand{\cI}{\mathcal{I}}

\newcommand{\cL}{\mathcal{L}}

\newcommand{\cN}{\mathcal{N}}

\newcommand{\cP}{\mathcal{P}}
\newcommand{\cQ}{\mathcal{Q}}

\newcommand{\cT}{\mathcal{T}}

\newcommand{\cW}{\mathcal{W}}

\newcommand{\cZ}{\mathcal{Z}}

\newcommand{\scrD}{\mathscr{D}}

\newcommand{\bvec}{\mathbf{b}}

\newcommand{\zvec}{\mathbf{z}}

\newcommand{\zerovec}{\mathbf{0}} 
\newcommand{\GPvec}{\mathbf{Z}}

\newcommand{\Cmat}{\mathbf{C}}

\newcommand{\Lmat}{\mathbf{L}}
\newcommand{\Mmat}{\mathbf{M}}
\newcommand{\Qmat}{\mathbf{Q}}

\newcommand{\T}{\top} %transpose

\newcommand{\norm}[2]{     \| #1       \|_{ #2 }}
\newcommand{\seminorm}[2]{  | #1        |_{ #2 }}

\newcommand{\scalar}[2]{     ( #1       )_{ #2 }}

\newcommand{\duality}[2]{     \langle #1       \rangle_{ #2 }}

\newcommand{\white}{\cW}

\newcommand{\tr}{\operatorname{tr}}
 
\newcommand{\rd}{\mathrm{d}}
\newcommand{\from}{\colon}
\newcommand{\GP}{\cZ}
\newcommand{\dirac}[1]{\text{\dh}_{#1}}
\newcommand{\nbeta}{n_\beta}
\newcommand{\betafrac}{\beta_{\star}}
\newcommand{\clos}[1]{\overline{ #1 }}
\newcommand{\pitilde}{\widetilde{\Pi}_h}

\newcommand{\dual}[1]{#1^*} 
\newcommand{\Hdot}[1]{\dot{H}^{#1}_L}
\newcommand{\Hdoth}[1]{\dot{H}^{#1}_{h}}

\newtheorem{lemma}{Lemma}[section]
\newtheorem{proposition}[lemma]{Proposition}
\newtheorem{theorem}[lemma]{Theorem}
\newtheorem{corollary}[lemma]{Corollary}

\theoremstyle{remark}
\newtheorem{remark}[lemma]{Remark}

\theoremstyle{definition}
\newtheorem{definition}[lemma]{Definition}
\newtheorem{assumption}[lemma]{Assumption}

\begin{document}
%------------------------------------------------------------------

%\title[Whittle--Mat\'ern fields: H\"older and covariance convergence]
%	{Galerkin approximations of Whittle--Mat\'ern fields: 
%		strong convergence in the H\"older norm and 
%		convergence of the covariance}

\title[Generalized Whittle--Mat\'ern fields: regularity and approximation]
{Regularity and convergence analysis 
	in Sobolev and H\"older spaces  
	for generalized Whittle--Mat\'ern fields}

%------------------------------------------------------------------

\author{Sonja G.\ Cox \and Kristin Kirchner}

\address[Sonja G.\ Cox]{Korteweg-de Vries Institute for Mathematics \\
	University of Amsterdam \\
	Postbus 94248 \\ 
	NL--1090 GE Amsterdam \\
	The Netherlands.}

\email{S.G.Cox@uva.nl}

\address[Kristin Kirchner]{Seminar f\"ur Angewandte Mathematik \\
	ETH Z\"urich \\
	R\"amistrasse 101 \\ 
	CH--8092 Z\"urich, Switzerland.}

\email{kristin.kirchner@sam.math.ethz.ch}

%------------------------------------------------------------------
% Acknowledgement

\thanks{Acknowledgment.
The authors thank Christoph Schwab and Mark Veraar 
for helpful and valuable comments.}

%-------------------------------------------------------------------
\begin{abstract}
We analyze several Galerkin approximations 
of a Gaussian random field~$\mathcal{Z}\colon\mathcal{D}\times\Omega\to\mathbb{R}$ 
indexed by a Euclidean domain $\mathcal{D}\subset\mathbb{R}^d$
whose covariance structure 
is determined by a negative 
fractional power $L^{-2\beta}$   
of a second-order
elliptic differential operator 
$L:= -\nabla\cdot(A\nabla) + \kappa^2$.  
Under minimal assumptions 
on the domain $\mathcal{D}$,  
the coefficients 
$A\colon\mathcal{D}\to\mathbb{R}^{d\times d}$, 
$\kappa\colon\mathcal{D}\to\mathbb{R}$, 
and the fractional exponent $\beta>0$, 
we prove convergence in $L_q(\Omega; H^\sigma(\mathcal{D}))$ 
and in $L_q(\Omega; C^\delta(\overline{\mathcal{D}}))$ 
at (essentially) optimal rates 
for (i) spectral Galerkin methods and (ii) 
finite element approximations. 
Specifically, 
our analysis  
is solely based on  
$H^{1+\alpha}(\mathcal{D})$-regularity of 
the differential operator $L$, 
where $0<\alpha\leq 1$. 
For this setting, 
we furthermore provide rigorous estimates 
for the error in the covariance function 
of these approximations 
in $L_{\infty}(\mathcal{D}\times\mathcal{D})$ 
and in the mixed Sobolev space 
$H^{\sigma,\sigma}(\mathcal{D}\times\mathcal{D})$, 
showing convergence which is 
more than twice as fast 
compared to the corresponding 
$L_q(\Omega; H^\sigma(\mathcal{D}))$-rate.  

For the well-known example of such 
Gaussian random fields, the 
original Whittle--Mat\'ern class, 
where $L=-\Delta + \kappa^2$ 
and $\kappa \equiv \operatorname{const.}$,  
we perform several numerical experiments 
which validate our theoretical results.  
\end{abstract}


 %secondary: 35R60, 60G15, 60G60, 60H15

\date{\today}

\maketitle

%=================================================================

%\setcounter{tocdepth}{1}
%\tableofcontents

%=================================================================
\section{Introduction}\label{section:intro}
%=================================================================

%=================================================================
\subsection{Motivation and background}\label{subsec:intro:background}
%=================================================================

By virtue of their practicality owing to  
the full characterization 
by their mean and covariance structure, 
Gaussian random fields (GRFs for short) 
are popular models  
for many applications 
in spatial statistics and uncertainty quantification, 
see, e.g., 
\cite{Bolin:2011,Cameletti:2013,Lindgren:2011,Penny:2005,Sain:2011}. 
As a result, several methodologies 
in these disciplines require 
the efficient simulation of GRFs 
at unstructured locations 
in various possibly non-convex 
Euclidean domains, 
and this topic has been 
intensively discussed 
in both areas, 
spatial statistics and 
computational mathematics, 
see, e.g., \cite{bkk-strong, 	
	Chan:1997,Chen:2017,Dietrich:1997,Feischl:2018,
	Graham:2018,Latz:2019,Osborn:2017}.
In particular, sampling from 
\emph{non-stationary} 
GRFs, for which methods based on 
circulant embedding are inapplicable, 
has become a central topic 
of current research, see, e.g.,  
\cite{bkk-strong,Chen:2017,Feischl:2018}. 

In order to capture both  
stationary and non-stationary GRFs, 
a new class of random fields has been introduced 
in~\cite{Lindgren:2011}, which is 
based on the following observation made 
by P.~Whittle~\cite{Whittle:1963}: 
A GRF $\GP$ on $\cD:=\bbR^d$ 
with covariance function of Mat\'ern type 
solves the fractional-order 
stochastic partial differential equation (SPDE for short) 
\begin{equation}\label{eq:intro:matern}
L^{\beta}\GP = \rd \white , 
\quad \text{in }\cD, 
\qquad 
L:= -\Delta + \kappa^2, 
\end{equation}
where $\Delta$ denotes the Laplacian, 
$\rd \white$ is white noise on $\bbR^d$, 
and $\kappa>0$, $\beta>\nicefrac{d}{4}$ 
are constants which determine the practical 
correlation length and the smoothness of the field. 
In~\cite{Lindgren:2011} this relation has been exploited 
to formulate generalizations of Mat\'ern fields, 
the \emph{generalized Whittle--Mat\'ern fields}, 
by considering the SPDE \eqref{eq:intro:matern} 
for non-stationary differential operators $L$ 
(e.g., by allowing for a spatially varying coefficient 
$\kappa\from\cD\to\bbR$)
on bounded domains $\cD\subset\bbR^d$, $d\in\{1,2,3\}$. 
Note that the covariance structure of a GRF is uniquely 
determined by its covariance operator, in this case 
given by the negative 
fractional-order differential operator~$L^{-2\beta}$. 
Furthermore, for the case $2\beta\in\bbN$, 
approximations based on a finite element discretization 
have been proposed in~\cite{Lindgren:2011}. 

Subsequently, 
a computational approach which allows for  
arbitrary fractional exponents $\beta>\nicefrac{d}{4}$ 
has been suggested 
in~\cite{bkk-strong,bkk-weak}.  
To this end, a sinc quadrature 
combined with a Galerkin 
discretization of the 
differential operator $L$
is applied to the Balakrishnan integral representation of
the fractional-order inverse $L^{-\beta}$. 

In this work, we investigate 
Sobolev and H\"older regularity of 
generalized Whittle--Mat\'ern fields 
and we perform a rigorous 
error analysis in these norms for 
several Galerkin approximations, 
including the sinc-Galerkin approximations 
of~\cite{bkk-strong,bkk-weak}.  
Specifically, we consider a 
GRF~$\GP^\beta \from\cD\times\Omega\to\bbR$, 
indexed by a Euclidean domain $\cD\subset\bbR^d$,
whose covariance operator  
is given by the negative 
fractional power $L^{-2\beta}$ 
of a second-order
elliptic differential operator 
$L\colon \scrD(L)\subseteq L_2(\cD) \rightarrow L_2(\cD)$ 
in divergence form with Dirichlet boundary conditions, 
formally given by 
\begin{equation}\label{eq:intro:L}
Lu = - \nabla \cdot ( A \nabla u) + \kappa^2 u, 
\qquad 
u \in \scrD(L)\subseteq L_2(\cD). 
\end{equation}
Here, we solely assume that 
$\cD\subset\bbR^d$ has a Lipschitz boundary, 
$\kappa\in L_{\infty}(\cD)$,  
and that 
$A\in L_\infty\bigl(\cD; \bbR^{d\times d}\bigr)$ 
is symmetric and
uniformly positive definite. 

For a sequence $\bigl( \GP_N^\beta \bigr)_{N\in\bbN}$ 
of Galerkin 
approximations for $\GP^\beta$ 
(namely, spectral Galerkin approximations  
in Section~\ref{section:matern:spectral} 
and sinc-Galerkin approximations in
Section~\ref{section:matern:galerkin})
defined with respect to family 
$(V_N)_{N\in\bbN}$ of 
subspaces $V_N\subset H^1_0(\cD)$ 
of finite dimension 
$\dim(V_N)=N<\infty$, 
we prove 
convergence at 
(essentially) optimal rates. 
More precisely, 
under minimal regularity conditions 
on the operator $L$ in \eqref{eq:intro:L} 
and for $0\leq \sigma < 2\beta -\nicefrac{d}{2}$, 
$\delta \in (0,\sigma)$, 
within a suitable parameter range we show 
that for all $\eps,q>0$ 
there exists a constant $C>0$ such that, 
for all $N\in\bbN$, 
\begin{align}
\left(\bbE\Bigl[
\bigl\| \GP^{\beta} - \GP^\beta_{N} \bigr\|_{H^{\sigma}(\cD)}^q
\Bigr]\right)^{\nicefrac{1}{q}} 
&\leq C
N^{-\nicefrac{1}{d} \, 
	\left(2\beta - \sigma 
	- \nicefrac{d}{2} -\eps\right)}, 
\label{eq:intro:conv:field:sobolev} 
\\ 
\left( \bbE 
\Bigl[ 
\bigl\| \GP^\beta - \GP^{\beta}_{N} \bigr\|_{C^{\delta}(\clos{\cD})}^q 
\Bigr] 
\right)^{\nicefrac{1}{q}} 
& \leq C 
N^{-\nicefrac{1}{d} \, 
	\left(2\beta - \sigma  
	- \nicefrac{d}{2}-\eps\right)},
\label{eq:intro:conv:field:hoelder} 
\\  
\bigl\| 
\varrho^\beta - \varrho^\beta_{N} 
\bigr\|_{H^{\sigma,\sigma}(\cD\times\cD)}
&\leq C    
N^{-\nicefrac{1}{d} \, 
	\left(4\beta - 2\sigma 
	- \nicefrac{d}{2}-\eps\right)},  
\label{eq:intro:conv:cov:sobolev} 
\\
\sup_{x,y\in\clos{\cD}} 
\bigl| \varrho^\beta(x,y) - \varrho^\beta_{N}(x,y) \bigr| 
& \leq C
N^{-\nicefrac{1}{d} \, 
	\left(4\beta - d -\eps\right)}.  
\label{eq:intro:conv:cov:Linfty} 
\end{align}
Here, $\varrho^\beta,\varrho_N^\beta\from\cD\times\cD\to\bbR$ 
denote the covariance functions 
of the Whittle--Mat\'ern 
field $\GP^\beta$ and of 
the Galerkin approximation $\GP_N^\beta$, 
respectively. 
For details, see 
Corollaries~\ref{cor:spectral:conv:hdot}--\ref{cor:spectral:conv:hoelder} 
for spectral Galerkin approximations, 
and 
Theorems~\ref{thm:fem:smooth:sobolev},~\ref{thm:fem:non-smooth} 
for the sinc-Galerkin approach. 
``Suitable parameter range'' refers to 
the observations that  
\begin{enumerate*} 
	\item if a finite element method 
	of polynomial degree $p\in\bbN$ 
	is used to define the sinc-Galerkin approximation 
	or 
	\item  if $L$ in \eqref{eq:intro:L} 
	is \emph{$H^{1+\alpha}(\cD )$-regular}
	for $0<\alpha\leq 1$ maximal 
	(see Definition~\ref{def:H1alpha-regular}), 
\end{enumerate*} 
then the convergence rates  
of the sinc-Galerkin approximation cannot exceed 
$p+1-\sigma$ 
or 
$\min\{ 1+\alpha - \sigma , 2\alpha \}$, 
where $0\leq\sigma\leq 1$.  

We point out that due to the low regularity 
of white noise, 
$\rd\white\in H^{-\nicefrac{d}{2}-\eps}(\cD)$, 
which holds $\bbP$-almost surely and in $L_q(\Omega)$  
(cf.~\cite[Prop.~2.3]{bkk-strong})
the convergence results 
\eqref{eq:intro:conv:field:sobolev}--\eqref{eq:intro:conv:cov:Linfty}
are (essentially, up to $\eps>0$) optimal 
and they are also reflected in our 
numerical experiments, 
see Section~\ref{section:numexp}    
and the discussion 
in Section~\ref{section:conclusion}. 
Note furthermore that the convergence rates 
in \eqref{eq:intro:conv:field:hoelder}, \eqref{eq:intro:conv:cov:Linfty} 
of the field with respect to $L_q(\Omega;C^\delta(\clos{\cD}))$ 
and of the covariance function in the 
$C(\clos{\cD\times\cD})$-norm, 
which we obtain via a Kolmogorov--Chentsov argument,  
are by $\nicefrac{d}{2}$ better than 
applying the results 
\eqref{eq:intro:conv:field:sobolev}, \eqref{eq:intro:conv:cov:sobolev} 
combined with the Sobolev embeddings 
$H^{\delta+\nicefrac{d}{2}}(\cD)\hookrightarrow 
C^\delta(\clos{\cD})$ 
and 
$H^{\eps+\nicefrac{d}{2},\,\eps+\nicefrac{d}{2}}(\cD\times\cD)
\hookrightarrow C(\clos{\cD\times\cD})$,  
respectively. 
We remark that strong convergence 
of the sinc-Galerkin approximation 
with respect to the $L_2(\Omega;L_2(\cD))$-norm, 
i.e., \eqref{eq:intro:conv:field:sobolev} 
for $\sigma=0$,  
at the rate $2\beta-\nicefrac{d}{2}$ has  
already been proven in~\cite[Thm.~2.10]{bkk-strong}. 
However, the assumptions made 
in~\cite[Ass.~2.6 and Eq.~(2.19)]{bkk-strong}
require the differential operator $L$ 
to be at least $H^2(\cD)$-regular. 
Thus, our results do not only generalize 
the analysis of~\cite{bkk-strong} 
for the strong error 
to different norms, but also 
to less regular differential operators. 
This is of relevance for several practical 
applications, since the spatial domain,  
where the GRF is simulated, may be 
non-convex or the coefficient $A$ 
may have jumps. 
For this reason, in Subsection~\ref{subsubsec:fem:non-smooth} 
we work under the 
assumption that $L$ is $H^{1+\alpha}(\cD)$-regular for some 
$0< \alpha \leq 1$ (for instance, 
$\alpha<\nicefrac{\pi}{\omega}$ 
if $\cD$ is a non-convex domain 
with largest interior angle 
$\omega > \pi$). 

As an interim result while deriving 
the error 
bounds~\eqref{eq:intro:conv:field:sobolev}--\eqref{eq:intro:conv:cov:Linfty} 
for the sinc-Galerkin approximation,   
we prove a non-trivial extension 
of one of the main results in~\cite{BonitoPasciak:2015}.
Namely, we show that 
for all 
\[
\beta > 0, 
\quad 
0\leq \sigma \leq \min\{1,2\beta\},
\quad  
-1\leq \delta \leq 1+\alpha, 
\ 
\delta\neq\nicefrac{1}{2}   
\quad 
\text{with}
\quad 
2\beta + \delta - \sigma > 0, 
\]
and for all $\eps>0$, 
there exists a constant $C > 0$ such that, 
for $N\in\bbN$ and $g\in H^\delta(\cD)$,  
\begin{align*}
\bigl\| L^{-\beta} g - \widetilde{L}_{N}^{-\beta} g
\bigr\|_{H^\sigma(\cD)}
&\leq C 
N^{-\nicefrac{1}{d}
	\min\{ 2\beta + \delta - \sigma - \varepsilon, \,
	1+\alpha-\sigma,\, 1+\alpha+\delta,\, 2\alpha\}}
\| g \|_{H^\delta(\cD)}. 
\end{align*} 
Here, $\widetilde{L}_N^{-1} \from H^{-1}(\cD) \to V_N$ denotes 
the approximation of the data-to-solution map 
$L^{-1} \from H^{-1}(\cD) \to H^1_0(\cD)$ with 
respect to the Galerkin space $V_N\subset H^1_0(\cD)$. 
For details see Theorem~\ref{thm:fem-error}, 
Remark~\ref{rem:fem-error-sobolev} 
and Lemmata~\ref{lem:hdot-sobolev-alpha}--\ref{lem:scott-zhang}. 
This error estimate
was proven 
in~\cite[Thm.~4.3 \& Rem.~4.1]{BonitoPasciak:2015} 
only for $\beta \in (0,1)$, 
$\sigma =0$, and $\delta \geq 0$, 
see also the comparison in Remark~\ref{rem:bonito-comparison}.  

%=================================================================
\subsection{Outline}\label{subsec:intro:ouline}
%=================================================================

After specifying the mathematical setting 
as well as our notation in 
Subsections~\ref{subsec:intro:setting}--\ref{subsec:intro:notation}, 
we rigorously define the 
second-order elliptic differential 
operator $L$ from \eqref{eq:intro:L} under minimal 
assumptions on the coefficients $A,\kappa$ 
and the domain $\cD\subset\bbR^d$ 
in Section~\ref{section:2odo}; 
thereby collecting several auxiliary 
results for this type of operators. 
Section~\ref{section:grfs}  
is devoted to the regularity 
analysis of a GRF \emph{colored}  
by a linear operator $T$ 
which is bounded on $L_2(\cD)$. 
These results are subsequently applied 
in Section~\ref{section:matern-regularity} 
to the class of generalized Whittle--Mat\'ern 
fields, where $T:=L^{-\beta}$  
with $L$ defined as in Section~\ref{section:2odo}  
and $\beta > \nicefrac{d}{4}$.  
In Section~\ref{section:matern:spectral} 
we derive the convergence 
results 
\eqref{eq:intro:conv:field:sobolev}--\eqref{eq:intro:conv:cov:Linfty}
for spectral Galerkin approximations 
where the finite-dimensional subspace $V_N$ 
is generated by the eigenvectors 
of the operator $L$ 
corresponding to the $N$ smallest 
eigenvalues.  
We then investigate sinc-Galerkin 
approximations in Section~\ref{section:matern:galerkin},  
where we first let $V_N$ be an  
abstract Galerkin space 
satisfying certain 
approximation properties, 
see 
Subsections~\ref{subsec:matern:sinc-galerkin}--\ref{subsec:matern:galerkin-error}. 
Subsequently, in Subsection~\ref{subsec:matern:fem} 
we show that these 
properties are indeed satisfied 
if the Galerkin spaces originate 
from a quasi-uniform family of 
finite element discretizations  
of polynomial degree $p\in\bbN$,  
and we discuss the convergence 
behavior for two cases in detail: 
\begin{enumerate*} 
	\item the coefficients $A,\kappa$ 
		and the domain $\cD$ in~\eqref{eq:intro:L} 
		are smooth, and 
	\item $A,\kappa,\cD$ are such that 
		the differential operator $L$ in \eqref{eq:intro:L} 
		is only 
		$H^{1+\alpha}(\cD)$-regular for some 
		$0<\alpha\leq 1$. 
\end{enumerate*} 
In Section~\ref{section:numexp} 
we perform several numerical 
experiments for the   
model example \eqref{eq:intro:matern}, 
$d=1$, and sinc-Galerkin discretizations 
generated with a finite element method  
of polynomial degree $p\in\{1,2\}$. 
In Section~\ref{section:conclusion} 
we reflect on our outcomes. 

%=================================================================
\subsection{Setting}\label{subsec:intro:setting}
%=================================================================

Throughout this article, 
we let $(\Omega,\cF,\bbP)$ 
be a complete probability space
with expectation operator $\bbE$, 
and $\cD$ be a bounded, connected and 
open subset of $\bbR^d$, $d\in\bbN$, 
with closure~$\clos{\cD}$. 
 
In addition, we let 
$\white\colon L_2(\cD) \rightarrow L_2(\Omega)$
be an $L_2(\cD)$-isonormal Gaussian process 
in the sense of~\cite[Def.~1.1.1]{Nualart:2006}.

%=================================================================
\subsection{Notation}\label{subsec:intro:notation}
%=================================================================

For $B \subseteq \R^d$, $\cB(B)$ denotes the 
Borel $\sigma$-algebra on $B$ (i.e., the 
$\sigma$-algebra generated 
by the sets that are relatively open in $B$). 
%and $\lambda$ is the Lebesgue measure on $\cB(B)$. 
For two $\sigma$-algebras $\cF$ and $\cG$, 
$\cF\otimes \cG$ is the 
$\sigma$-algebra generated by $\cF \times \cG$. 

If $(E,\norm{\,\cdot\,}{E})$ is a Banach space, 
then 
$(\dual{E},\norm{\,\cdot\,}{\dual{E}})$ 
denotes 
its dual, 
$\duality{\,\cdot\,, \,\cdot\,}{\dual{E}\times E}$ 
the duality pairing 
on $\dual{E}\times E$,  
$\operatorname{Id}_E$ 
the identity on $E$, 
and $\cL(E;F)$ the space of bounded linear operators 
from $(E,\norm{\,\cdot\,}{E})$ to 
another Banach space 
$(F,\norm{\,\cdot\,}{F})$. 
For $T\in \cL(E;F)$ we write  
$\dual{T} \in \cL(\dual{F}; \dual{E})$
for the adjoint of $T$.
If $V$ is a vector space such that $E,F \subseteq V$ 
and if, in addition, $\operatorname{Id}_V|_{E} \in \cL(E;F)$, 
then 
we write $(E,\norm{\,\cdot\,}{E} ) 
\hookrightarrow (F, \norm{\,\cdot\,}{F})$. 
Moreover, the notation 
$(E, \norm{\,\cdot\,}{E} ) \cong (F, \norm{\,\cdot\,}{F})$ 
indicates that  
$(E, \norm{\,\cdot\,}{E} ) \hookrightarrow 
(F, \norm{\,\cdot\,}{F}) 
\hookrightarrow (E, \norm{\,\cdot\,}{E})$. 

If not specified otherwise,
$\scalar{\,\cdot\,,\,\cdot\,}{H}$
is the inner product
on a Hilbert space~$H$ and
$\cL_2(H;U)\subseteq\cL(H;U)$ 
denotes the Hilbert space 
of Hilbert--Schmidt operators 
between two Hilbert spaces~$H$ 
and~$U$. 
The adjoint of $T\in\cL(H;U)$ 
is identified with 
$\dual{T}\in\cL(U;H)$ 
(via the Riesz maps on~$H$ and on~$U$). 
We write $\cL(E)$ and $\cL_2(H)$
whenever $E=F$ and $H=U$. 
The domain of a possibly 
unbounded operator~$L$
is denoted by $\scrD(L)$. 

For $1\leq q < \infty$,  
$L_q(\cD;E)$ is the space of 
(equivalence classes of) 
$E$-valued, 
Bochner measurable,
$q$-integrable   
functions on $\cD$ 
and $L_q(\Omega;E)$ 
denotes the space 
of (equivalence classes of) $E$-valued random variables 
with finite $q$-th moment, i.e., 
\begin{align*}
	\norm{f}{L_q(\cD;E)} 
	&:= 
	\left( 
	\int_\cD \norm{f(x)}{E}^q \, \rd x 
	\right)^{\nicefrac{1}{q}},  
	&& 
	f \in L_q(\cD;E), \\ 
	\norm{X}{L_q(\Omega;E)}  
	&:=  
	\left( 
	\bbE \bigl[ \norm{X}{E}^q \bigr] 
	\right)^{\nicefrac{1}{q}}, 
	&& 
	X \in L_q(\Omega;E). 
\end{align*} 
The space $L_{\infty}(\cD; E)$ 
consists of all equivalence classes of $E$-valued,  
Bochner measurable functions 
which are essentially bounded on $\cD$, i.e., 
\[
	\norm{f}{L_{\infty}(\cD;E)} 
	:= 
	\operatorname{ess} \sup_{x\in\cD} 
	\norm{f(x)}{E}, 
	\qquad 
	f\in L_{\infty}(\cD;E).
\]
For $\gamma \in (0,1)$, we furthermore 
define the mappings 
\[
	\seminorm{\,\cdot\,}{C^\gamma(\clos{\cD};E)}, \,
	\norm{\,\cdot\,}{C^\gamma(\clos{\cD};E)}
	\from C(\clos{\cD};E) \to [0,\infty] 
\]
on the Banach space 
\[
	\bigl( C(\clos{\cD};E), \norm{\,\cdot\,}{C(\clos{\cD};E)} \bigr), 
	\qquad 
	\norm{f}{C(\clos{\cD};E)} 
	:= 
	\sup_{x \in \clos{\cD}} \norm{f(x)}{E}, 
\]
of continuous 
functions from $\clos{\cD}$ to 
$(E,\norm{\,\cdot\,}{E})$
via 
\begin{align}
	\seminorm{f}{C^\gamma(\clos{\cD};E)}
	&:=  
	\sup_{\substack{x,y\in \clos{\cD} \\ x\neq y}} \frac{\norm{f(x)-f(y)}{E}}{|x-y|^{\gamma}}, 
	\label{eq:def:hoelderseminorm} \\
	\norm{f}{C^\gamma(\clos{\cD};E)}
	&:=  
	\sup_{x \in \clos{\cD}} \norm{f(x)}{E} 
	+ 
	\seminorm{f}{C^\gamma(\clos{\cD};E)} . 
	\label{eq:def:hoeldernorm} 
\end{align}
Note that the norm 
$\norm{\,\cdot\,}{C^\gamma(\clos{\cD};E)}$ 
renders the subspace 
\begin{align}\label{eq:def:Cgamma} 
	C^{\gamma}(\clos{\cD};E) 
	&= 
	\left\{ f \in C(\clos{\cD};E) : 
	\norm{f}{C^{\gamma}(\clos{\cD};E)} < \infty \right\}
	\subset 
	C(\clos{\cD};E) 
\end{align}
of $\gamma$-H\"older continuous 
functions a Banach space. 
Whenever the functions or
random variables are real-valued, 
we omit the image space and write 
$C(\clos{\cD})$, 
$C^\gamma(\clos{\cD})$, $L_q(\cD)$,  
and $L_q(\Omega)$, respectively. 
For $\sigma > 0$, the  
(integer- or fractional-order)
Sobolev space is denoted by 
$H^\sigma(\cD)$ (see~\cite[Sec.~2]{NezzaEtAl:2012},  
see also~\cite[Sec.~1.11.4/5]{Yagi:2010}), 
and $H^1_0(\cD) \subset H^1(\cD)$  
is the closure of 
the space $C^{\infty}_{\operatorname{c}}(\cD)$ 
of compactly supported smooth functions 
in $\bigl( H^1(\cD), \norm{\,\cdot\,}{H^1(\cD)} \bigr)$. 

We mark
equations which hold 
almost everywhere or 
$\bbP$-almost surely 
with 
a.e.\ and 
$\bbP$-a.s., respectively. 
For two random variables 
$X, Y$, we write 
$X \overset{d}{=} Y$ 
whenever $X$ and $Y$ have 
the same probability distribution. 
The Dirac measure 
at $x\in\clos{\cD}$ 
is denoted 
by $\dirac{x}$.  
Given a parameter set $\cP$ 
and mappings $A,B\colon \cP \to \R$, 
we let $A(p) \lesssim B(p)$ 
denote the relation that 
there exists a constant $C>0$, 
independent of $p\in\cP$, 
such that 
$A(p) \leq C B(p)$ 
for all $p\in\cP$.  
For a further parameter set 
$\cQ$ and mappings 
$A,B\colon \cP\times \cQ \rightarrow \bbR$, 
we write $A(p,q) \lesssim_{q} B(p,q)$ 
if, for all $q\in\cQ$,  
there exists a constant 
$C_q > 0$, 
independent of $p\in\cP$, 
such that  
$A(p,q) \leq C_q B(p,q)$
for all $p\in\cP$ and $q\in\cQ$.
Finally, $A(p)\eqsim B(p)$ 
indicates that both relations, $A(p) \lesssim B(p)$ and $B(p) \lesssim A(p)$, 
hold simultaneously; 
and similarly for $A(p,q) \eqsim_q B(p,q)$.

%=======================================================================
\section{Auxiliary results on second-order elliptic differential operators}\label{section:2odo}
%=======================================================================

As outlined in Subsection~\ref{subsec:intro:background}, 
the overall objective of this article is to 
study (generalized) Whittle--Mat\'ern fields 
and Galerkin approximations for them. 
Here, we call a Gaussian random field 
a \emph{generalized Whittle--Mat\'ern field} 
if its covariance operator 
is given by a negative fractional power 
of a second-order elliptic differential operator.
The purpose of this section is to present 
preliminary results on second-order   
differential operators which will be of importance 
for the regularity and error analysis of these fields.

Firstly, we specify 
the class of differential operators 
that we consider. 
We start by formulating 
assumptions on the 
coefficients  
of the operator. 

\begin{assumption}[on the coefficients $A$ and $\kappa$]\label{ass:coeff} 
	Throughout this article we assume: % the following: 
	\begin{enumerate}[label=\Roman*.]
		\item\label{ass:coeff:A} 
			$A\in L_\infty\bigl(\cD; \bbR^{d\times d}\bigr)$ is symmetric and
			uniformly positive definite, i.e.,
			\begin{equation}\label{eq:A:elliptic}
				\exists a_0 > 0 :
				\quad
				\operatorname{ess} \inf_{x\in\cD} \xi^\T A(x) \xi \geq a_0 \seminorm{\xi}{}^2
				\quad
				\forall \xi \in \bbR^d;  
			\end{equation}	 
		\item\label{ass:coeff:c} 
			$\kappa\in L_{\infty}(\cD)$. 
    \end{enumerate} 
    Where explicitly specified, we require in addition:
    \begin{enumerate}[label=\Roman*.] \setcounter{enumi}{2} 
		\item\label{ass:coeff:ALip} 
			$A\from\clos{\cD}\to\bbR^{d\times d}$ 
			is Lipschitz continuous on the closure $\clos{\cD}$, i.e.,
			\[
				\exists a_{\operatorname{Lip}} > 0 : 
				\quad 
				\seminorm{A_{ij}(x) - A_{ij}(y)}{}  
				\leq 
				a_{\operatorname{Lip}} \seminorm{x-y}{}  
				\quad 
				\forall x,y\in\clos{\cD}, 
			\]
			for all $i,j\in\{1,\ldots,d\}$. 
	\end{enumerate}
\end{assumption} 

Under Assumptions~\ref{ass:coeff}.I--II 
we let $L \colon \scrD(L) \subset L_2(\cD) \rightarrow L_2(\cD)$ 
denote the maximal accretive operator 
on $L_2(\cD)$ associated with $A$ and $\kappa^2$ 
with largest domain $\scrD(L) \subset H^1_0(\cD)$. 
By this we mean that $\scrD(L)$ 
consists of precisely those $u\in H^1_0(\cD)$
for which there exists a constant $C \geq 0$ such that 
\[
    \left| \int_{\cD} 
        \bigl[ 
        \scalar{ A(x) \nabla u(x), \nabla v(x)}{\R^d} +
        \kappa^2(x) u(x) v(x) 
        \bigr] 
        \,\rd x \right|
    \leq 
    C \norm{v}{L_2(\cD)} 
    \quad 
    \forall v\in H^1_0(\cD),
\]
and, for $u\in \scrD(L)$, 
$Lu$ is the unique element of $L_2(\cD)$ 
which, for all $v\in H^1_0(\cD)$, satisfies
\begin{align}\label{eq:def:L}  
	\int_{\cD} 
	\bigl[ 
	\scalar{ A(x) \nabla u(x), \nabla v(x)}{\R^d} +
	\kappa^2(x)u(x)v(x) 
	\bigr] 
	\,\rd x
	= \scalar{Lu,v}{L_2(\cD)}. 
\end{align}
It is well-known that the operator 
$L\from\scrD(L) \to L_2(\cD)$ 
defined via \eqref{eq:def:L} 
is densely defined and self-adjoint 
(e.g., \cite[Prop.~1.22 and Prop.~1.24]{Ouhabaz:2005}). 
Furthermore, by the Lax--Milgram lemma, 
its inverse exists and extends to a bounded linear operator 
$L^{-1} \from \dual{H^1_0(\cD)} \to H^1_0(\cD)$ 
(e.g., \cite[Lem.~1.3]{Ouhabaz:2005}). 
By the  
Kondrachov compactness theorem 
$L^{-1} \from L_2(\cD) \to L_2(\cD)$ is compact 
(e.g.,~\cite[Thm.~7.22]{GilbargTrudinger:2001}).  

For this reason, the spectrum of $L$ 
consists of a system of only positive 
eigenvalues~$(\lambda_j)_{j\in \bbN}$ 
with no accumulation point, 
whence we can assume them to be in 
nondecreasing order. 
The following 
asymptotic spectral behavior, 
known as \emph{Weyl's law}   
(see, e.g., \cite[Thm.~6.3.1]{Davies:1995}), 
will be exploited 
several times
in our analysis. 

\begin{lemma}\label{lem:spectral-behav}
	Let $L$ be the second-order 
	differential operator in~\eqref{eq:def:L}, 
	defined with respect to 
	the bounded open  
	domain $\cD\subset\bbR^d$, and 
	with coefficients $A$ and $\kappa$ 
	fulfilling Assumptions~\ref{ass:coeff}.I--II. 
	Then, the eigenvalues 
	of $L$ (in nondecreasing order) 
	satisfy 
	\begin{equation}\label{eq:lem:spectral-behav} 
		\lambda_j \eqsim_{(A,\kappa,\cD)} j^{\nicefrac{2}{d}} , 
		\qquad 
		j\in\bbN. 
	\end{equation}
\end{lemma} 

%\begin{proof}
%	If the operator 
%	$L = - a \Delta + \kappa^2\operatorname{Id}_{L_2(\cD)}$ 
%	has constant coefficients 
%	$a > 0$, $\kappa \in \bbR$, and the 
%	spatial domain $\cD\subset\bbR^d$ is an open rectangle,
%	i.e.,
%	%
%	\begin{equation}\label{eq:cube} 
%		\cD=(r_1, R_1)\times\ldots\times(r_d,R_d), 
%		\qquad 
%		r_i,R_i\in\bbR, 
%		\quad 
%		r_i < R_i, 
%		\quad 
%		1\leq i \leq d,  
%	\end{equation}
%	%
%	the behavior~\eqref{eq:lem:spectral-behav}
%	of the eigenvalues of $L$ 
%	is well-known (see, e.g.,~\cite[Ch.~VI.4]{Courant:1962}). 
%	 
%	The general case now follows by the min-max principle,
%	see, e.g.,~\cite[Thm.~XIII.1]{ReedSimon:1978} 
%	or~\cite[Thm.~12.1 and Cor.~12.3]{Schmudgen:2012}. More 
%	precisely, we define 
%	\[
%		\kk{a_1 := 
%		%d^2 \max_{1\leq i,j\leq d} 
%		%\operatorname{ess} \sup_{x\in\cD} |A_{ij}(x)|, 
%		\sup_{\xi\in\bbR^d, \, |\xi|=1} 
%		\operatorname{ess} \sup_{x\in\cD} 
%		\xi^\T  A(x) \xi},
%		\qquad 
%		\kappa_1 := \operatorname{ess} \sup_{x\in\cD} \kappa(x), 
%		\qquad 
%		\kappa_0 := 0, 
%	\]
%	and recall the constant $a_0>0$ 
%	from~\eqref{eq:A:elliptic}. 
%	In addition, we let $\cD_0,\cD_1\subset \R^d$ 
%	be open rectangles such that 
%	$\cD_1 \subseteq \cD \subseteq \cD_0$,
%	and we define 
%	$L_{i} \colon H^2(\cD_i)\cap H^1_0(\cD_i) \rightarrow L_2(\cD_i)$
%	by $L_i := - a_i \Delta + \kappa_i^2 \operatorname{Id}_{L_2(\cD)}$, 
%	$i\in \{0,1\}$. The fact 
%	that~\eqref{eq:lem:spectral-behav} 
%	holds for $L=L_i$, $i\in \{0,1\}$, 
%	and~\cite[Cor.~12.3]{Schmudgen:2012} 
%	imply~\eqref{eq:lem:spectral-behav}.
%\end{proof} 
 
We let $\cE := \{e_j\}_{j\in\bbN}$ denote 
a system of eigenvectors 
of the operator $L$ in~\eqref{eq:def:L} 
which corresponds to the 
eigenvalues~$(\lambda_j)_{j\in\bbN}$ 
and which is orthonormal in $L_2(\cD)$.
Note that, for $\sigma > 0$, 
the fractional power operator 
$L^{\sigma}\colon 
\scrD(L^{\sigma}) \subset L_2(\cD) \rightarrow L_2(\cD)$ 
is well-defined. 
Indeed, on the domain 
\[
	\scrD(L^{\sigma})
	:=
	\Biggl\{
	\psi\in L_2(\cD) :
	\sum_{j\in\bbN}
	\lambda_j^{2\sigma} \scalar{\psi,e_j}{L_2(\cD)}^2
	< \infty
	\Biggr\} 
\]
the action of $L^\sigma$  
is given via  
the spectral representation 
\[
	L^{\sigma} \psi := \sum_{j\in\bbN}
	\lambda_j^{\sigma} \scalar{\psi,e_j}{L_2(\cD)} e_j, 
	\qquad 
	\psi\in\scrD(L^{\sigma}). 
\]
The subspace 
\begin{align}\label{eq:def:Hdot}
	\bigl( \Hdot{\sigma}, \scalar{\,\cdot\,, \,\cdot\,}{\sigma} \bigr), 
	\qquad 
	\Hdot{\sigma} := \scrD\bigl(L^{\nicefrac{\sigma}{2}}\bigr) 
	\subset L_2(\cD), 
\end{align}
is itself a Hilbert space 
with respect to the inner product 
\[
	\scalar{\phi,\psi}{\sigma} 
	:= 
	\bigl( L^{\nicefrac{\sigma}{2}} \phi, 
	L^{\nicefrac{\sigma}{2}} \psi \bigr)_{L_2(\cD)}
	= \sum_{j\in\bbN} \lambda_j^{\sigma}  
	\scalar{\phi,e_j}{L_2(\cD)} 
	\scalar{\psi,e_j}{L_2(\cD)}, 
\]
and the corresponding 
induced norm $\norm{\,\cdot\,}{\sigma}$. 
In what follows, we let 
$\Hdot{0} := L_2(\cD)$ 
and, for $\sigma>0$, 
$\Hdot{-\sigma}$ denotes  
the dual space $\dual{(\Hdot{\sigma})}$ 
after identification via the 
inner product 
$\scalar{\,\cdot\,, \,\cdot\,}{L_2(\cD)}$  
on $L_2(\cD)$ 
which is continuously extended 
to a duality pairing. 

In order to derive regularity and convergence 
results with respect to the  
Sobolev space $H^\sigma(\cD)$ 
and the space $C^{\gamma}(\clos{\cD})$ 
of $\gamma$-H\"older continuous 
functions in~\eqref{eq:def:Cgamma}, 
we wish to
relate the various norms involved 
by well-known results from interpolation theory and 
by Sobolev embeddings. 
To this end, we need to consider various assumptions on 
the spatial domain~$\cD$, specified below.

\begin{assumption}[on the domain $\cD$]\label{ass:dom} 
Throughout this article, we assume that 
	\begin{enumerate}[label=\Roman*.]
		\item\label{ass:dom:Lipschitz} $\cD$ has a Lipschitz continuous boundary $\partial\cD$. 
    \end{enumerate}
Where explicitly specified, we additionally suppose 
one or both of the following:
    \begin{enumerate}[label=\Roman*.]\setcounter{enumi}{1}
		\item\label{ass:dom:convex} $\cD$ is convex; 
		\item\label{ass:dom:polytope} $\cD$ is a polytope. 
	\end{enumerate}
\end{assumption} 

Note that~\ref{ass:dom:convex} implies~\ref{ass:dom:Lipschitz} 
(see, e.g.,~\cite[Cor.~1.2.2.3]{Grisvard:2011}). 

In the following lemma 
we specify the relationship  
between the spaces~$\Hdot{\sigma}$ 
in~\eqref{eq:def:Hdot}
and the Sobolev space~$H^{\sigma}(\cD)$,   
under two sets of 
assumptions 
on the spatial domain~$\cD$ 
and on the coefficients~$A$,~$\kappa$ 
of the differential operator $L$ 
in~\eqref{eq:def:L}. 
We recall that $[E,F]_{\sigma}$ 
denotes the complex interpolation space 
between $(E,\norm{\,\cdot\,}{E})$ 
and $(F,\norm{\,\cdot\,}{F})$ 
with parameter $\sigma \in [0,1]$,
see, e.g., \cite[Ch.~2]{Lunardi:2018}. 

\begin{lemma}\label{lem:hdot-sobolev}  
	Let 
	Assumptions~\ref{ass:coeff}.I--II 
	and~\ref{ass:dom}.I 
	be satisfied.  
	Then    
	\begin{equation}\label{eq:hdot-inter}
		\bigl(
	 		\Hdot{\sigma}, \norm{\,\cdot\,}{\sigma}
	 	\bigr) 
		\cong 
	 	\bigl( 
	 		\left[
        		L_2(\cD), H^1_0(\cD)
        	\right]_{\sigma}, 
        	\norm{\,\cdot\,}{
            	[L_2(\cD), H^1_0(\cD)]_{\sigma} }
		\bigr),  
		\qquad 
		0\leq \sigma \leq 1, 
	\end{equation}
	holds for the space 
	$\bigl(\Hdot{\sigma}, \norm{\,\cdot\,}{\sigma} 
	\bigr)$ from~\eqref{eq:def:Hdot}. 
	Furthermore, 
	\begin{equation}\label{eq:hdot-sobolev1}
		\bigl(
			\Hdot{\sigma}, 
			\norm{\,\cdot\,}{\sigma} 
		\bigr) 
        \hookrightarrow
        \bigl(
            H^{\sigma}(\cD), 
			\norm{\,\cdot\,}{H^{\sigma}(\cD)}
		\bigr), 
		\qquad 
		0\leq \sigma \leq 1,  
	\end{equation}
	and the norms 
	$\norm{\,\cdot\,}{\sigma}$, 
	$\norm{\,\cdot\,}{H^{\sigma}(\cD)}$ 
	are equivalent on $\Hdot{\sigma}$ 
	for $0\leq\sigma\leq1$ 
	and $\sigma\neq \nicefrac{1}{2}$. 
	If, in addition, 
	Assumptions~\ref{ass:coeff}.III 
	and~\ref{ass:dom}.II  
	hold, then    
	\begin{equation}\label{eq:hdot-sobolev2}
	 	\bigl(
	 		\Hdot{\sigma}, 
	 		\norm{\,\cdot\,}{\sigma}
	 	\bigr) 
	 	\cong 
	 	\bigl(
	 		H^{\sigma}(\cD)\cap H^1_0(\cD), 
	 		\norm{\,\cdot\,}{H^{\sigma}(\cD)} 
	 	\bigr), 
		\qquad 
		1\leq \sigma \leq 2. 
	\end{equation}
\end{lemma}

\begin{proof} 
	First, note that \cite[Cor.~2.4]{Yagi:2010} 
	implies~\eqref{eq:hdot-inter}.
	If $(E,\norm{\,\cdot\,}{E})$, 
	$(F, \norm{\,\cdot\,}{F})$, 
	and $(G, \norm{\,\cdot\,}{G})$ are Banach spaces 
	such that 
	$(F, \norm{\,\cdot\,}{F}) 
	\hookrightarrow 
	(G, \norm{\,\cdot\,}{G})$, 
	then by definition of complex 
	interpolation we have 
	$\bigl([E,F]_{\sigma}, \norm{\,\cdot\,}{[E,F]_{\sigma}}\bigr) 
	\hookrightarrow 
	 \bigl([E,G]_{\sigma}, \norm{\,\cdot\,}{[E,G]_{\sigma}}\bigr)$.
	This observation in connection   
	with~\cite[Thm.~1.35]{Yagi:2010} 
	(which collects several results from~\cite{Triebel:1978})
	shows~\eqref{eq:hdot-sobolev1}. 
	Equivalence of 
	$\norm{\,\cdot\,}{\sigma}$, 
	$\norm{\,\cdot\,}{H^{\sigma}(\cD)}$ 
	on $\Hdot{\sigma}$ for $0\leq\sigma\leq 1$, 
	$\sigma \neq \nicefrac{1}{2}$,  
	is proven in \cite[Thm.~8.1]{Grisvard:1967}. 
	By combining~\eqref{eq:hdot-inter} 
	for $\sigma=1$,~\cite[Thm.~4.36]{Lunardi:2018} 
	and~\cite[Lem.~A2]{Guermond:2009} 
	(recalling Assumption~\ref{ass:dom}.II)
	we find that~\eqref{eq:hdot-sobolev2} 
	for $\sigma \in (1,2)$
	follows once~\eqref{eq:hdot-sobolev2} 
	is established for the case $\sigma =2$. 
	
	It thus remains to 
	prove~\eqref{eq:hdot-sobolev2} 
	for $\sigma=2$. 
	To this end, 
	we first observe that, 
	for a vanishing coefficient $\kappa\equiv 0$ 
	of the operator $L$ in~\eqref{eq:def:L}, 
	we have, e.g., by~\cite[Thm.~3.2.1.2]{Grisvard:2011} 
	the regularity result 
	\begin{equation}\label{eq:regularity-smooth} 
	    f \in L_2(\cD) 
	    \quad 
	    \Rightarrow 
	    \quad 
	    u:=L^{-1}f \in H^2(\cD)\cap H^1_0(\cD).
	\end{equation}  
	If $\kappa\not\equiv 0$, then 
	$u \in H^1_0(\cD)$  
	satisfies the equality 
	$-\nabla\cdot(A\nabla u) = f - \kappa^2 u$ 
	in the weak sense 
	so that~\cite[Thm.~3.2.1.2]{Grisvard:2011}  
	applied to 
	$\widetilde{f} := f - \kappa^2 u \in L_2(\cD)$
	again yields~\eqref{eq:regularity-smooth}. 
	By the
	closed graph theorem, 
	$\bigl(\Hdot{2}, \norm{\,\cdot\,}{2}\bigr) 
	\hookrightarrow 
	\bigl(H^2(\cD)\cap H^1_0(\cD), \norm{\,\cdot\,}{H^2(\cD)}\bigr)$.
	 
	We now establish the reverse embedding. 
	By Assumption~\ref{ass:coeff}.III  
	and, e.g.,~\cite[Thm.~4 in Ch.~5.8]{Evans:2010} 
	(note that the assumptions on 
	the boundary posed therein 
	can be circumvented by exploiting an extension argument as, 
	e.g., in~\cite[Sec.~VI.2.3 Thm.~3]{Stein:1970}, 
	see also the remark below~\cite[Thm.~4 in Ch.~5.8]{Evans:2010}),
	$A_{ij}$ is differentiable a.e.~in~$\cD$ 
	with essentially bounded weak derivatives 
	$\partial_{x_k} A_{ij} \in L_{\infty}(\cD)$, 
	$1\leq i,j,k\leq d$. 
	Thus (by first approximating 
	$A_{ij}$ in $H^1(\cD)$ with a sequence in 
	$C^{\infty}(\cD)$ to obtain 
	that $A_{ij} \partial_{x_j} u$
	is weakly differentiable with  
	$\partial_{x_k}(A_{ij} \partial_{x_j} u)
	= \partial_{x_k} A_{ij} \partial_{x_j} u 
	+ A \partial_{x_k x_j} u$), 
	we conclude that 
	$A \nabla u \in H^1(\cD)^d$  
	whenever $u\in H^2(\cD)\cap H^1_0(\cD)$.
	Finally, by the closed graph theorem  
	$\bigl(H^2(\cD)\cap H^1_0(\cD), \norm{\,\cdot\,}{H^2(\cD)} \bigr) 
	\hookrightarrow 
	\bigl( \Hdot{2}, \norm{\,\cdot\,}{2} \bigr)$ 
	follows. 
\end{proof}

%=======================================================================
\section{General results on Gaussian Random Fields (GRFs)}\label{section:grfs}
%=======================================================================

In this section we address different notions 
of regularity (H\"older and Sobolev) for 
Gaussian random fields (GRFs) 
and their covariance functions. 
We first recall the definition 
of a GRF and specify 
then what we mean by a \emph{colored} GRF.  
As usually, we work in the setting 
formulated in Subsection~\ref{subsec:intro:setting}. 

\begin{definition}\label{def:GRF}
	Let $B\subseteq \bbR^d$. 
	A family of $\cF$-measurable 
	$\bbR$-valued random variables $(\GP(x))_{x\in B}$
	is called a \emph{random field (indexed by $B$)}. 
	It is called \emph{Gaussian} if the random 
	vector~$(\GP(x_1),\ldots,\GP(x_n))^\T$ is Gaussian 
	for all finite sets $\{ x_1,\ldots,x_n \} \subset B$.  
	It is called \emph{continuous} 
	if the mapping~$x\mapsto \GP(x)(\omega)$ 
	is continuous for all $\omega\in \Omega$.
\end{definition}
\begin{definition}\label{def:colored} 
	Let $T\in\cL(L_2(\cD))$. 
	We call~$\GP\colon \cD\times \Omega \rightarrow \bbR$ 
	a \emph{Gaussian random field (GRF) 
	colored by $T$}
	if it is a GRF, 
	a $\cB(\cD)\otimes \cF$-measurable 
	mapping, and 
	\begin{equation}\label{eq:def:colored}
	\scalar{\GP,\psi}{L_2(\cD)} 
	= 
	\white(\dual{T}\psi)  
	\quad 
	\bbP\text{-a.s.} 
	\quad 
	\forall\psi\in L_2(\cD). 
	\end{equation}
\end{definition}
\begin{remark}
	It is well-known and easily verified 
	(see also Proposition~\ref{prop:regularity:hdot}) 
	that there exists a square-integrable GRF~$\GP$ 
	colored by~$T$ 
	if and only if~$T\in \cL_2(L_2(\cD))$, and 
	in this case 
	$\bbE \bigl[\| \GP \|_{L_2(\cD)}^2 \bigr] 
		= \tr(T\dual{T})$, 
	where $\tr(\,\cdot\,)$ is the trace 
	on $L_2(\cD)$.  
\end{remark}

%=======================================================================
\subsection{H\"older regularity of GRFs}\label{subsec:grfs:hoelder}
%=======================================================================

We now provide an abstract result on 
the construction and 
H\"older regularity of a GRF
assuming that the color and, thus, 
the covariance structure 
of the field is given. 
\begin{proposition}\label{prop:regularity:hoelder}
	Assume that $T\in \cL(L_2(\cD); C^\gamma(\clos{\cD}))$ 
	for some $\gamma\in(0,1)$. 
	Then there exists a continuous GRF $\GP$ 
	colored by $T$ such that
	\begin{equation}\label{eq:def:reg:hoelder}
		\GP(x) = \white(\dual{T}\dirac{x})  
		\quad 
		\bbP\text{-a.s.}   
		\quad 
		\forall x\in\clos{\cD}. 
	\end{equation}
	Furthermore, 
	for $q \in (0,\infty)$ and 
	$\delta \in (0,\gamma)$, 
	we have 
	\begin{equation}\label{eq:rf_reg}
		\left( 
		\bbE \Bigl[ \norm{\GP}{C^{\delta}(\clos{\cD})}^q \Bigr] 
		\right)^{\nicefrac{1}{q}} 
		\lesssim_{(q,\gamma,\delta,\cD)} 
		\norm{T}{\cL(L_2(\cD); C^{\gamma}(\clos{\cD}))}.
	\end{equation}
\end{proposition} 

\begin{proof}
	We first define the random field 
	$\GP_0\from \clos{\cD} \times \Omega \to \bbR$
	by $\GP_0(x) := 
	\white(\dual{T} \dirac{x})$ 
	for all $x\in\clos{\cD}$. 
	By the properties of 
	an isonormal Gaussian process 
	we find, for $x,y\in\clos{\cD}$, 
	\begin{align}
		\left( \bbE\bigl[ |\GP_0(x) - \GP_0(y)|^2\bigr]
			\right)^{\nicefrac{1}{2}}
		&= \left( \bbE\bigl[ 
			|\white( \dual{T}
			(\dirac{x} - \dirac{y})
			 )|^2\bigr] \right)^{\nicefrac{1}{2}} 
		 = \norm{\dual{T} 
		 	(\dirac{x} - \dirac{y})
	 		}{\dual{L_2(\cD)}} \notag \\
		&\leq \norm{\dual{T}}{
			\cL(\dual{C^{\gamma}(\clos{\cD})};
				\dual{L_2(\cD)})}
			\norm{
				\dirac{x} - \dirac{y}
			}{\dual{C^{\gamma}(\clos{\cD})}} 
			\notag \\
		&= 
			\norm{T}{\cL(L_2(\cD);
				C^{\gamma}(\clos{\cD}))} 
			|x-y|^{\gamma}.
		\label{eq:Hoelder_estimate} 
	\end{align} 
	Since 
	$\GP_0(x)-\GP_0(y) 
	= \white (\dual{T}
	(\dirac{x} - \dirac{y}))$ 
	is a 
	real-valued 
	Gaussian random variable, 
	we can apply  
	the Khintchine inequalities 
	(see, e.g., \cite[Thm.~4.7 and p.~103]{LedouxTalagrand:2011}) 
	and conclude 
	with~\eqref{eq:Hoelder_estimate} 
	that, for all $q\in (0,\infty)$,  
	the estimate 
	\begin{align}
		\seminorm{\GP_0}{C^{\gamma}(\clos{\cD};L_q(\Omega))}
		&\leq C_q 
			\sup\limits_{\substack{x,y\in \clos{\cD} \\ x\neq y}} 
			\left( \E \Biggl[ \left| 
				\frac{\GP_0(x)-\GP_0(y)}{|x-y|^{\gamma}}
			\right|^{2} \Biggr] \right)^{\nicefrac{1}{2}} \notag \\ 
		&\leq C_q \norm{T}{\cL(L_2(\cD); C^{\gamma}(\clos{\cD}))} 
		\label{eq:Hoelderseminormest}
	\end{align}
	holds, with 	
	a constant $C_q > 0$  
	depending only on $q$.  
		
	Thus, by the Kolmogorov--Chentsov 
	continuity theorem 
	(e.g.,~\cite[Thm.~I.2.1]{RevuzYor3:1999},
	combined with an extension argument as discussed in 
	the proof of 
	\cite[Thm.~2.1]{MittmannSteinwart:2003}, see also~\cite[Thm.~VI.2.3]{Stein:1970}), 
	there exists a continuous random field 
	$\GP\from \clos{\cD} \times \Omega \to \bbR$
	such that $\GP(x)=\GP_0(x)$ $\bbP$-a.s.\ 
	for all $x\in \clos{\cD}$, and furthermore,
	for every $\delta\in (0,\gamma)$ and 
	every finite $q > (\gamma-\delta)^{-1}$, 
	we can find a constant $C_{q,\gamma,\delta,\cD} > 0$,  
	depending only on $q$, $\gamma$, $\delta$, 
	as well as the dimension 
	and the diameter of $\cD\subset\bbR^d$, such that 
	\begin{align}\label{eq:Hoelderseminormbound}
	\left(\bbE \Bigl[ \seminorm{\GP}{C^{\delta}(\clos{\cD})}^q 
		\Bigr] \right)^{\nicefrac{1}{q}} 
	\leq C_{q,\gamma,\delta,\cD} 
		\seminorm{\GP_0}{C^{\gamma}(\clos{\cD}; L_q(\Omega))}.
	\end{align}
	Next, again by the Khintchine inequalities, 
	we have, for all $x\in\clos{\cD}$ and all $q\in (0,\infty)$,  
	\begin{align}
		\left(\bbE\left[ | \GP(x) |^q \right] \right)^{\nicefrac{1}{q}} 
		&=
		\left(\bbE\left[ | \GP_0(x) |^q \right] \right)^{\nicefrac{1}{q}} 
		\notag \\ 
		&\leq C_q 
		\left(\bbE\left[ |\white(\dual{T} \dirac{x})|^2 
			\right] \right)^{\nicefrac{1}{2}} 
		\leq C_q  
		\norm{T}{\cL(L_2(\cD); C^{\gamma}(\clos{\cD}))}.
		\label{eq:onepointest}
	\end{align}
	From~\eqref{eq:def:hoelderseminorm}--\eqref{eq:def:hoeldernorm} 
	we deduce, for every~$\delta\in(0,1)$ 
	and all~$f\in C^\delta(\clos{\cD})$, 
	the relation 
	\begin{align*}
		\norm{f}{C^{\delta}(\clos{\cD})} 
		\leq 
		|f(x)|
		+ 
		\bigl( 1 + \operatorname{diam}(\cD)^{\delta} \bigr) 
		\seminorm{f}{C^{\delta}(\clos{\cD})}  
		%\leq 
		%\bigl( 1 + \operatorname{diam}(\cD)^{\delta} \bigr)
		%\norm{f}{C^{\delta}(\cD;E)}, 
		\quad 
		\forall x \in \clos{\cD}. 
	\end{align*}
	We combine this observation  
	with~\eqref{eq:Hoelderseminormest},~\eqref{eq:Hoelderseminormbound}, 
	and~\eqref{eq:onepointest} to derive,  
	for all $\delta\in (0,\gamma)$ and all  
	finite $q > (\gamma-\delta)^{-1}$, 
	the bound  
	\begin{align}
		\left(\bbE\Bigl[ \norm{\GP}{C^{\delta}(\clos{\cD})}^q
			\Bigr]\right)^{\nicefrac{1}{q}} 
		&\leq C_q 
			\norm{T}{\cL(L_2(\cD); C^{\gamma}(\clos{\cD}))}
			+ \bigl( 1 + \operatorname{diam}(\cD)^{\delta} \bigr) 
			\left(\bbE\Bigl[ 
			    \seminorm{\GP}{C^{\delta}(\clos{\cD})}^q 
				\Bigr] \right)^{\nicefrac{1}{q}} \notag \\ 
		&\leq C_q 
			\bigl( 1 + C_{q,\gamma,\delta,\cD} 
				\bigl( 1+\operatorname{diam}(\cD)^{\delta}\bigr) \bigr) 
			\norm{T}{\cL(L_2(\cD); C^{\gamma}(\clos{\cD}))}.
		\label{eq:Hoeldernormbound}
	\end{align}
	Note that H\"older's inequality and~\eqref{eq:Hoeldernormbound} 
	ensure that~\eqref{eq:rf_reg} holds 
	for every $\delta\in (0,\gamma)$ and 
	every $q \in (0,\infty)$. 
	Furthermore, for every $\psi\in L_2(\cD)$, one readily verifies  
	the identity  
	$\bbE\bigl[|\scalar{\GP, \psi}{L_2(\cD)} - \white(\dual{T}\psi)|^2\bigr]=0$, 
	i.e., 
	$\GP$ is colored by $T$. 	
%	
%	It remains to verify that $\GP$ is colored by $T$. Note that the 
%	fact that $\GP$ is continuous implies 
%	that it is $\cB(\clos{\cD})\otimes \cF$-measurable.
%	Moreover, it follows from~\eqref{eq:def:reg:hoelder}, 
%	Fubini's theorem,
%	and the Wiener chaos decomposition 
%	(e.g.,~\cite[Thm.~1.1.1]{Nualart:2006}) that 
%	%
%	\begin{align*}
%		\bbE \Bigl[ \scalar{ \GP,
%		&\, \psi }{L_2(\cD)}  e^{\white(\phi)} \Bigr]
%	    	= 
%	    	\bbE\left[
%		  	\int_{\cD} \white(\dual{T}\dirac{x}) \psi(x) \, \rd x
%		  	\, e^{\white(\phi)}
%		 	\right] \\
%		&=
%	    \int_{\cD} 
%	        \psi(x) \, 
%	        \bbE\left[
%	        	\white(\dual{T}\dirac{x})e^{\white(\phi)} 
%	        \right] 
%	        \rd x 
%	    = \scalar{ \psi, T \phi }{L_2(\cD)} 
%	    = \bbE\left[ \white(\dual{T} \psi) e^{\white(\phi)} \right],  
%	\end{align*}
%	%
%	for all $\psi, \phi\in L_2(\cD)$. 
%	As $\bigl\{ e^{\white(\phi)} :  
%	\phi \in L_2(\cD) \bigr\}$ separates points in 
%    $L_2(\Omega,\sigma(\white),\bbP)$ 
%    (see, e.g., \cite[Lem.~1.1.2]{Nualart:2006}), 
%    this proves that~\eqref{eq:def:colored} holds, i.e., 
%    $\GP$ is colored by $T$. 
\end{proof}

If Assumption~\ref{ass:dom}.I 
is fulfilled,  
the Sobolev embedding theorem 
(see, e.g.,~\cite[Thm.~5.4 and Thm.~8.2]{NezzaEtAl:2012}) 
is applicable and 
we obtain $\gamma$-H\"older 
continuity~\eqref{eq:def:Cgamma} 
for elements in the 
fractional-order Sobolev space 
$H^{\gamma + \nicefrac{d}{2}}(\cD)$ 
for every
$\gamma\in(0,1)$. 
This continuous embedding, 
$H^{\gamma + \nicefrac{d}{2}}(\cD) 
\hookrightarrow C^{\gamma}(\clos{\cD})$,  
combined 
with Proposition~\ref{prop:regularity:hoelder} 
leads to the following result.
\begin{corollary}\label{cor:regularity:hoelder}
	Let 
	Assumption~\ref{ass:dom}.I, 
	$\gamma \in (0,1)$, and 
	$T\in \cL\bigl(L_2(\cD); H^{\gamma + \nicefrac{d}{2}}(\cD)\bigr)$
	be satisfied. 
	Then there exists a continuous GRF  
	$\GP\colon \clos{\cD} \times \Omega \rightarrow \bbR$ 
	colored by~$T$, cf.~\eqref{eq:def:colored}, \linebreak   
	such that 
	$\GP(x) = \white(\dual{T}\dirac{x})$ 
	$\bbP$-a.s.\ for all $x\in\clos{\cD}$. 
	Moreover, the stability estimate 
	\begin{align}\label{eq:cor:regularity:hoelder} 
		\left( \bbE 
			\Bigl[ \norm{\GP}{C^{\delta}(\clos{\cD})}^q \Bigr]
		\right)^{\nicefrac{1}{q}} 
		\lesssim_{(q,\gamma,\delta,\cD)} 
			\norm{T}{
				\cL\left(L_2(\cD);  
				H^{\gamma + \nicefrac{d}{2}}(\cD)\right)} 
	\end{align} 
	for the $q$-th moment of $\GP$ 
	with respect to the $\delta$-H\"older 
	norm~\eqref{eq:def:hoeldernorm} 
	holds for every $\delta\in(0,\gamma)$ 
	and $q\in(0,\infty)$. 
\end{corollary}

We close the subsection with a brief discussion on
\begin{enumerate*} 
	\item the continuity of covariance functions 
		of colored GRFs, and 
	\item the $L_\infty(\cD\times\cD)$-distance between two     
		covariance functions of GRFs 
		colored by different operators.  
\end{enumerate*} 

By definition, the covariance function 
$\varrho \in L_2(\cD\times\cD)$ 
of a square-integrable random field 
$\GP \in L_2(\cD \times \Omega)$ 
satisfies 
\begin{equation}\label{eq:def:covfct} 
	\varrho(x,y) 
	= 
	\bbE[(\GP(x) - \bbE[\GP(x)]) 
	(\GP(y) - \bbE[\GP(y)])]
	\quad 
	\text{a.e.\ in }\cD\times\cD.  
\end{equation}
We obtain the one-to-one correspondence 
\[
	\int_{\cD} 
	\varrho(x,y) \psi(y)  
	\, \rd y 
	= 
	(\cC \psi)(x)   
	\quad 
	\text{a.e.\ in }\cD,   
	\quad 
	\forall \psi \in L_2(\cD), 
\]
with the covariance operator 
$\cC \from L_2(\cD) \to L_2(\cD)$ 
of the field $\GP$, 
which is defined via 
\[
	\scalar{\cC \phi, \psi}{L_2(\cD)} 
	= 
	\bbE\bigl[
		\scalar{\GP - \bbE[\GP],\phi}{L_2(\cD)}
		\scalar{\GP - \bbE[\GP],\psi}{L_2(\cD)}
	\bigr] 
	\quad 
	\forall \phi,\psi \in L_2(\cD). 
\]
From this definition it is evident 
that a GRF $\GP$ colored by~$T$ 
(note that $\bbE[\GP] = 0$ by construction, 
see Definition~\ref{def:colored}) 
has the covariance operator 
$\cC = T\dual{T}$. 
In the next lemma, this relation 
is exploited to characterize 
continuity of the covariance function~$\varrho$ 
in terms of the color $T$ of the GRF~$\GP$.  

\begin{proposition}\label{prop:regularity:Linf} 
	Let 
	$\GP, \widetilde{\GP}$ be GRFs  
	colored by $T$ and 
	$\widetilde{T}$, 
	respectively, see~\eqref{eq:def:colored},
	with covariance functions 
	denoted by 
	$\varrho$ and $\widetilde{\varrho}$, 
	cf.~\eqref{eq:def:covfct}. 
	Then, 
	\begin{enumerate}[label=\normalfont{(\roman*)}]
		\item\label{prop:regularity:Linf-i} 
			$\varrho$ 
			has a continuous representative 
			on $\clos{\cD\times\cD} $ 
			(again denoted by 
			$\varrho$)
			if and only if 
			$T\in\cL(L_2(\cD); C(\clos{\cD}))$. 
			In this case, 
			\begin{equation}\label{eq:prop:regularity:Linf-i} 
				\sup_{x,y\in\clos{\cD}} | \varrho(x,y) | 
				\leq 
				\norm{T\dual{T}}{\cL(\dual{C(\clos{\cD})}; 
						 C(\clos{\cD}))} \,; 
			\end{equation}
		\item\label{prop:regularity:Linf-ii}  
			if $T,\widetilde{T}\in \cL(L_2(\cD); C(\clos{\cD}))$, 
			then~$\varrho,\widetilde{\varrho} 
			\in C(\clos{\cD\times\cD})$ satisfy  
			\begin{equation}\label{eq:prop:regularity:Linf-ii} 
				\sup_{x,y\in \clos{\cD}}  
				\bigl| 
					\varrho(x,y) - \widetilde{\varrho}(x,y) 
				\bigr|  
				\leq 
				\bigl\| 
					T\dual{T} - \widetilde{T}\dual{\widetilde{T}} 
				\bigr\|_{ 
					\cL(\dual{C(\clos{\cD})}; C(\clos{\cD}))} \,.
			\end{equation} 
	\end{enumerate} 
\end{proposition}

\begin{proof} 
	By~\eqref{eq:def:covfct}, 
	the covariance function~$\varrho$  
	of a GRF $\GP$ colored by $T$ 
	is given by 
	\begin{equation}\label{eq:covfct-xy}
		\varrho(x,y) 
		= 
		\scalar{\dual{T}\dirac{x}, \dual{T}\dirac{y}}{\dual{L_2(\cD)}}
		\quad 
		\text{a.e.\ in }\clos{\cD\times\cD}. 
	\end{equation}
	First, let $T\in\cL(L_2(\cD); C(\clos{\cD}))$. 
	Then, we have 
	$\dual{T}\in\cL(\dual{C(\clos{\cD})}; \dual{L_2(\cD)})$
	and continuity 
	of $\varrho \from\clos{\cD\times\cD} \to \bbR$ 
	follows from~\eqref{eq:covfct-xy}. 
	Assume now that  
	$\varrho \in C(\clos{\cD\times\cD})$. 
	Then, again by~\eqref{eq:covfct-xy}, 
	we obtain  
	$\norm{\dual{T}\dirac{x}}{\dual{L_2(\cD)}}^2 
	= \varrho(x,x)	< \infty$ 
	for all $x\in\clos{\cD}$   
	and   
	\[
		\norm{T \phi}{C(\clos{\cD})} 
		= 
		\sup_{x\in\clos{\cD}} 
		\duality{\dirac{x}, T \phi}{
				\dual{C(\clos{\cD})}\times C(\clos{\cD})}
%		= 
%		\sup_{x\in\cD} 
%		\duality{ \dual{T}\dirac{x}, \phi}{
%			\dual{L_2(\cD)} \times L_2(\cD)}  
		\leq 
		\sup_{x\in\clos{\cD}} 
		\norm{\dual{T}\dirac{x}}{\dual{L_2(\cD)}}
		<\infty 
	\]
	holds for all $\phi\in L_2(\cD)$ 
	with $\norm{\phi}{L_2(\cD)} \leq 1$.  
	Thus, $T\in\cL(L_2(\cD); C(\clos{\cD}))$ 
	if $\varrho$ 
	is continuous. 
	Furthermore, by identifying 
	$\dual{L_2(\cD)} \cong L_2(\cD)$ 
	via the Riesz map, 
	the covariance operator~$\cC$ of~$\GP$ 
	satisfies  
	$\cC = T\dual{T} \in \cL(\dual{C(\clos{\cD})}; C(\clos{\cD}))$,  
	and we can deduce~\eqref{eq:prop:regularity:Linf-i}
	from~\eqref{eq:covfct-xy} since,  
	for all $x,y \in \clos{\cD}$,  
	\[ 
		|\varrho(x,y)| 
		= 
		|\duality{\dirac{x}, T\dual{T}\dirac{y}}{
			\dual{C(\clos{\cD})}\times C(\clos{\cD})}| 
		\leq \norm{T\dual{T}\dirac{y}}{C(\clos{\cD})} 
		\leq \norm{T\dual{T}}{\cL(\dual{C(\clos{\cD})}; C(\clos{\cD}))}. 
	\]
	Finally, the estimate~\eqref{eq:prop:regularity:Linf-ii} 
	can be shown similarly since, 
	for all $x,y\in\clos{\cD}$, 
	\begin{align*}
		|\varrho(x,y)
			-\widetilde{\varrho}(x,y)| 
		&= 
		\bigl| \bigl\langle \dirac{x}, 
			\bigl( T\dual{T} - \widetilde{T}\dual{\widetilde{T}}\bigr)
				\dirac{y}\bigr\rangle_{\dual{C(\clos{\cD})}\times C(\clos{\cD})} \bigr| .
		\qedhere
	\end{align*} 
\end{proof}

%=======================================================================
\subsection{Sobolev regularity of GRFs and their covariances}\label{subsec:grfs:cov}
%=======================================================================

After having characterized 
\begin{enumerate}[label=(\roman*)] 
	\item the H\"older 
		regularity (in $L_q(\Omega)$-sense) 
		of a GRF $\GP$, and  
	\item continuity of the covariance function 
		$\varrho$ in~\eqref{eq:def:covfct}, 
\end{enumerate}
in terms of the color of $\GP$,  
we now proceed with this discussion 
for Sobolev spaces. 
Specifically, 
we investigate the regularity 
of $\GP$ in $L_q(\Omega;H^\sigma(\cD))$
and  
of the covariance 
function~$\varrho$ 
with respect to the norm 
on the mixed Sobolev space 
\begin{equation}\label{eq:def:Hmixed} 
	H^{\sigma,\sigma}(\cD\times\cD)
	:= 
	H^{\sigma}(\cD) 
	\mathop{\hat{\otimes}} 
	H^{\sigma}(\cD) . 
\end{equation}  
Here, $\mathop{\hat{\otimes}}$ denotes 
the tensor product of Hilbert spaces. 
Thus, the inner product on 
$H^{\sigma,\sigma}(\cD\times\cD)$ 
inducing the norm 
$\norm{\,\cdot\,}{H^{\sigma,\sigma}(\cD\times\cD)}$ 
is uniquely defined via   
\[
	\scalar{\phi \otimes \chi,    
	\psi \otimes \vartheta }{
	H^{\sigma,\sigma}(\cD\times\cD)} 
	:= 
	\scalar{\phi, \psi}{H^\sigma(\cD)} 
	\scalar{\chi, \vartheta}{H^\sigma(\cD)} 
	\quad 
	\forall \phi,\psi,\chi,\vartheta\in H^\sigma(\cD). 
\]
To this end, 
in the following proposition 
we first quantify the $\Hdot{\sigma}$-regularity 
(in $L_q(\Omega)$-sense)  
of a colored GRF 
in terms of its color,  
cf.~\eqref{eq:def:Hdot} and Definition~\ref{def:colored}. 
In addition, we specify the regularity 
of the covariance function~\eqref{eq:def:covfct}
in the Hilbert tensor product space 
\begin{equation}\label{eq:def:Hdotmixed} 
	\bigl( \Hdot{\sigma,\sigma}, 
	\norm{\,\cdot\,}{\sigma,\sigma} \bigr), 
	\qquad 
	\Hdot{\sigma,\sigma}  
	:= 
	\Hdot{\sigma} \mathop{\hat{\otimes}} \Hdot{\sigma},  
\end{equation}
cf.~\eqref{eq:def:Hmixed}. 
Finally, we characterize 
the distance between two GRFs 
which are colored by different operators 
with respect to these norms.  

\begin{proposition}
	\label{prop:regularity:hdot}
	Let $\GP\from\cD\times\Omega\to\R$ 
	be a GRF colored by 
	$T\in\cL(L_2(\cD))$, 
	cf.~\eqref{eq:def:colored}.  
	Then~$\GP$ is square-integrable,  
	i.e., $\GP\in L_2(\cD\times\Omega)$, 
	if and only if 	
	its covariance operator   
	$\cC = T\dual{T}$ 
	has a finite trace 
	on $L_2(\cD)$. 
	More generally, for 
	$\sigma\geq 0$ and 
	$q\in(0,\infty)$, 
	we have 
	\begin{align}
		\bbE \bigl[
		\norm{\GP}{\sigma}^2 \bigr] 
		&= 
		\operatorname{tr}( 
		T\dual{T} L^\sigma) 
		= 
		\norm{T}{\cL_2^{0;\sigma}}^2,
		\label{eq:field-reg-2}
		\\ 
		\left(\bbE\bigl[
		\norm{\GP}{\sigma}^q
		\bigr]
		\right)^{\nicefrac{1}{q}}
		&\eqsim_{q}  
		\sqrt{\tr( 
		T\dual{T} L^\sigma)} 
		= 
		\norm{T}{\cL_2^{0;\sigma}},
		\label{eq:field-reg-p}
		\\ 
		\norm{\varrho}{
			\sigma,\sigma} 
		&= 
		\norm{\cC}{\cL_2^{-\sigma;\sigma}}
		=
		\norm{T\dual{T}}{\cL_2^{-\sigma;\sigma}}. 
		\label{eq:cov-fct-reg}
	\end{align} 
	Here, 
	$\tr(\,\cdot\,)$ 
	is the trace on $L_2(\cD)$,  
	$L$ is the differential operator 
	in~\eqref{eq:def:L} 
	with coefficients $A,\kappa$ satisfying  
	Assumptions~\ref{ass:coeff}.I--II, 
	$\varrho$ is the 
	covariance 
	function of~$\GP$, 
	see~\eqref{eq:def:covfct}, 
	and~$\cL_2^{\theta;\sigma}$ 
	is a short notation for the 
	Hilbert--Schmidt space  
	$\cL_2\bigl( \Hdot{\theta}; 
	\Hdot{\sigma} \bigr)$.
	
	If 
	$\widetilde{\GP}\in L_2(\cD\times\Omega)$ 
	is another GRF colored by 
	$\widetilde{T} \in \cL(L_2(\cD))$, 
	with covariance function $\widetilde{\varrho}$ 
	and covariance operator 
	$\widetilde{\cC} = \widetilde{T}\dual{\widetilde{T}}$, 
	we have, 
	for $\sigma \geq 0$ and $q\in(0,\infty)$, 
	\begin{align}
	\left( \bbE \Bigl[ 
			\bigl\| \GP - \widetilde{\GP} \bigr\|_{\sigma}^q 
		\Bigr] \right)^{\nicefrac{1}{q}}
	&\eqsim_q  
			\bigl\| T - \widetilde{T} \bigr\|_{\cL_2^{0;\sigma}},  
	\label{eq:field-err-p} \\
	\norm{\varrho - \widetilde{\varrho}}{
		\sigma,\sigma} 
	&= 
		\bigl\| \cC - \widetilde{\cC} 
		\bigr\|_{\cL_2^{-\sigma;\sigma}} 
	=
	\bigl\| 
		T\dual{T} - \widetilde{T} \dual{\widetilde{T}} 
	\bigr\|_{\cL_2^{-\sigma;\sigma}}. 
	\label{eq:cov-fct-err}  
	\end{align}	
\end{proposition}

\begin{proof} 
	Assume first that 
	$\GP\in L_2(\cD\times\Omega)$. 
	Since $\GP$ has mean zero 
	and since it is colored by~$T\in\cL(L_2(\cD))$, 
	we obtain $\cC = T\dual{T}$, i.e., 
	\begin{align*} 
		\bbE\bigl[
		\scalar{\GP,\phi}{L_2(\cD)} 
		\scalar{\GP,\psi}{L_2(\cD)}
		\bigr]
		=
		\scalar{T\dual{T} \phi, \psi}{L_2(\cD)} 
		\quad 
		\forall 
		\phi,\psi\in L_2(\cD). 
	\end{align*}	
	By choosing 
	$\phi=\psi:=\lambda_j^{\nicefrac{\sigma}{2}} e_j$, 
	summing these equalities over $j\in\bbN$, 
	and exchanging the order 
	of summation and expectation, 
	we obtain 
	the identity 
	\begin{align*} 
		\bbE\bigl[\norm{\GP}{\sigma}^2\bigr]
		&= 
		\sum_{j\in\bbN} 
		\lambda_j^\sigma 
		\scalar{T\dual{T}e_j,e_j}{L_2(\cD)}
		=
		\operatorname{tr}( 
		T\dual{T} L^\sigma) 
		= 
		\bigl\| L^{\nicefrac{\sigma}{2}} T \bigr\|_{\cL_2^{0;0}} 
		= 
		\norm{T}{\cL_2^{0;\sigma}}, 
	\end{align*}
	and the first part of the proposition 
	as well as~\eqref{eq:field-reg-2} 
	are proven. 
	The estimate~\eqref{eq:field-reg-p} 
	follows from~\eqref{eq:field-reg-2} 
	by the Kahane--Khintchine inequalities
	(see, e.g.,~\cite[Thm.~4.7 and p.~103]{LedouxTalagrand:2011}), 
	since $\GP$ is an $\Hdot{\sigma}$-valued 
	zero-mean Gaussian random variable. 
	
	Assume now that 
	$\widetilde{\GP} \in L_2(\cD\times\Omega)$ 
	is another GRF colored by $\widetilde{T}\in\cL(L_2(\cD))$. 
	Then we obtain~\eqref{eq:field-err-p} 
	from~\eqref{eq:field-reg-p}, since 
	$\GP - \widetilde{\GP}$ is again 
	a GRF, colored by $T - \widetilde{T}$, 
	see~\eqref{eq:def:colored} and Definition~\ref{def:colored}. 
	Furthermore, 
	we find
	\begin{align*} 
		\norm{\varrho 
		- \widetilde{\varrho}}{\sigma,\sigma}^2 
%		&= 
%		\sum\limits_{i\in\bbN}
%		\sum\limits_{j\in\bbN} 
%		\lambda_i^\sigma 
%		\lambda_j^\sigma 
%		\bigl(\scalar{\cC e_i, e_j}{L_2(\cD)} 
%		- 
%		\scalar{\widetilde{\cC} e_i, e_j}{L_2(\cD)}
%		\bigr)^2 \\
		&=
		\sum\limits_{i\in\bbN}
		\sum\limits_{j\in\bbN} 
		\lambda_i^\sigma 
		\lambda_j^\sigma 
		\bigl( 
			\bigl(T\dual{T} 
			- \widetilde{T}\dual{\widetilde{T}} \bigr)
			e_i, 
			e_j \bigr)_{L_2(\cD)}^2 \\ 	
		&= 	
		\sum\limits_{i\in\bbN}
		\bigl\| \bigl(T\dual{T} 
			- \widetilde{T}\dual{\widetilde{T}} \bigr) 
			L^{\nicefrac{\sigma}{2}}e_i \bigr\|_{
			\sigma}^2  
		= \bigl\| T\dual{T} 
			- \widetilde{T}\dual{\widetilde{T}} 
			\bigr\|_{\cL_2^{-\sigma;\sigma}}^2.
	\end{align*}
	This proves~\eqref{eq:cov-fct-err} 
	and~\eqref{eq:cov-fct-reg}
	follows from this result   
	for   
	$\widetilde{\GP} \equiv 0$. 
\end{proof}

\begin{remark} 
	\label{rem:regularity:sobolev}
	Note that if 
	Assumptions~\ref{ass:coeff}.I--II,~\ref{ass:dom}.I,  
	and 
	$0\leq \sigma\leq 1$
	(or 
	Assumptions~\ref{ass:coeff}.I--III,~\ref{ass:dom}.II,  
	and 
	$0\leq \sigma\leq 2$)  
	are satisfied and $\sigma\neq\nicefrac{1}{2}$, 
	it follows from Lemma~\ref{lem:hdot-sobolev}
	that 
	all assertions 
	of 
	Proposition~\ref{prop:regularity:hdot}
	%and Lemma~\ref{lem:matern:hdot}
	remain true 
	if we replace 
	the equalities 
	with equivalences and 
	the norms 
	$\norm{\,\cdot\,}{\sigma}$,  
	$\norm{\,\cdot\,}{\sigma,\sigma}$ 
	(cf.~the spaces in~\eqref{eq:def:Hdot}, \eqref{eq:def:Hdotmixed})
	with the Sobolev norm 
	$\norm{\,\cdot\,}{H^\sigma(\cD)}$ 
	and with the 
	norm $\norm{\,\cdot\,}{H^{\sigma,\sigma}(\cD\times\cD)}$ 
	on the mixed 
	Sobolev space~\eqref{eq:def:Hmixed},  
	respectively. 
	Furthermore, 
	by~\eqref{eq:hdot-sobolev1} 
	Proposition~\ref{prop:regularity:hdot} 
	provides upper bounds for these 
	quantities if $\sigma=\nicefrac{1}{2}$.
\end{remark}

%=======================================================================
\section{Regularity of Whittle--Mat\'ern fields}
\label{section:matern-regularity}
%=======================================================================

In this section 
we focus on the 
regularity of 
(generalized)
Whittle--Mat\'ern fields, 
i.e., of GRFs 
colored (cf.\ Definition~\ref{def:colored}) 
by a negative fractional power 
of the differential 
operator $L$ 
as provided in~\eqref{eq:def:L}. 
Specifically, we consider 
\begin{equation}\label{eq:gpbeta} 
	\GP^\beta \from \cD\times\Omega \to \R, 
	\qquad 
	\bigl( \GP^\beta, \psi \bigr)_{L_2(\cD)}
	= \white\bigl(L^{-\beta}\psi\bigr)  
	\quad 
	\bbP\text{-a.s.}  
	\quad 
	\forall\psi\in L_2(\cD), 
\end{equation}
for 
\begin{equation}\label{eq:def:beta} 
	\beta := \nbeta + \betafrac,  
	\qquad 
	\nbeta\in\bbN_0, 
	\qquad 
	0\leq \betafrac < 1.
\end{equation}
We emphasize the dependence 
of the covariance structure of $\GP^\beta$ 
on the fractional exponent $\beta>0$ 
by the index and   
write $\varrho^\beta$ 
for the covariance function~\eqref{eq:def:covfct}
of~$\GP^\beta$. 

The first aim of this section is to apply 
Proposition~\ref{prop:regularity:hdot}
for specifying the regularity 
of $\GP^\beta$ in~\eqref{eq:gpbeta} 
and of its covariance function $\varrho^\beta$ 
with respect to the spaces 
$\Hdot{\sigma}$ 
and $\Hdot{\sigma,\sigma}$ 
in~\eqref{eq:def:Hdot}, \eqref{eq:def:Hdotmixed}.  
%see Lemma~\ref{lem:matern:hdot}. 
As already pointed out 
in Remark~\ref{rem:regularity:sobolev}, 
provided that 
the assumptions of Lemma~\ref{lem:hdot-sobolev}
are satisfied, 
this implies  
regularity in the Sobolev space $H^{\sigma}(\cD)$  
and in the mixed Sobolev space $H^{\sigma,\sigma}(\cD\times\cD)$ 
in~\eqref{eq:def:Hmixed}, respectively.

Besides this regularity result 
with respect to the 
spaces 
$\Hdot{\sigma}$ 
and 
$H^{\sigma}(\cD)$, 
we obtain a stability  
estimate with respect to the 
H\"older norm 
from Corollary~\ref{cor:regularity:hoelder} 
and continuity of the 
covariance function 
from Proposition~\ref{prop:regularity:Linf}. 
%in Lemmata~\ref{lem:matern:hoelder}--\ref{lem:matern:Linf}.
Although we believe that, 
at least in some specific cases, 
these results are well-known, 
for the sake of completeness, 
we derive them here 
in our general framework. 

\begin{lemma} 
	\label{lem:matern:hdot}
	Let Assumptions~\ref{ass:coeff}.I--II 
	be fulfilled, 
	$\beta,q\in(0,\infty)$, 
	$\sigma\geq 0$, 
	and $\GP^\beta$ be the 
	Whittle--Mat\'ern field in~\eqref{eq:gpbeta}, 
	with covariance function $\varrho^\beta$. 
	Then, 
	\begin{enumerate}[label=\normalfont{(\roman*)}] 
		\item\label{lem:matern:hdot:field}  
			$\bbE\Bigl[\bigl\| \GP^\beta \bigr\|_{\sigma}^q\Bigr] < \infty$ 
			if and only if 
			$2\beta > \sigma + \nicefrac{d}{2}$, and 
		\item\label{lem:matern:hdot:cov}     
			$\bigl\| \varrho^\beta \bigr\|_{\sigma,\sigma}
			< \infty$  
			if and only if 
			$2\beta > \sigma + \nicefrac{d}{4}$. 
	\end{enumerate}	
	If, in addition,  
	Assumption~\ref{ass:dom}.I 
	and 
	$0\leq \sigma\leq 1$ 
	(or   
	Assumptions~\ref{ass:coeff}.I--III,~\ref{ass:dom}.II,  
	and 
	$0\leq \sigma\leq 2$) 
	hold, 
	then the assertions 
	\ref{lem:matern:hdot:field}--\ref{lem:matern:hdot:cov} 
	remain true 
	if we formulate them 
	with respect to the Sobolev norms   
	$\norm{\,\cdot\,}{H^\sigma(\cD)}$,  
	$\norm{\,\cdot\,}{H^{\sigma,\sigma}(\cD\times\cD)}$.
\end{lemma} 

\begin{proof}  
	By Proposition~\ref{prop:regularity:hdot} 
	we have, for any 
	$\beta,q\in(0,\infty)$ and $\sigma\geq 0$,  
	\begin{align}
		\left(\bbE\Bigl[
		\bigl\| \GP^\beta \bigr\|_{\sigma}^q \Bigr]
		\right)^{\nicefrac{2}{q}} 
		\eqsim_{q} 
		\tr\bigl( L^{-2\beta + \sigma} \bigr) 
		&= 
		\sum_{j\in\bbN} \lambda_j^{-(2\beta-\sigma)}, 
		\label{eq:lem:matern:hdot:field}\\
		\bigl\| \varrho^\beta \bigr\|_{\sigma,
			\sigma}^2
		= 
		\bigl\| L^{-2\beta} \bigr\|_{
			\cL_2\left(\Hdot{-\sigma};  
			\Hdot{\sigma}\right)}^2 
		&= 
		\sum_{j\in\bbN} 
		\lambda_j^{-2(2\beta - \sigma)}. 
		\label{eq:lem:matern:hdot:cov}
	\end{align}	
	
	Combining the spectral 
	behavior~\eqref{eq:lem:spectral-behav} 
	of~$L$ from Lemma~\ref{lem:spectral-behav}
	with~\eqref{eq:lem:matern:hdot:field}/\eqref{eq:lem:matern:hdot:cov}    
	proves~\ref{lem:matern:hdot:field}/\ref{lem:matern:hdot:cov}
	for $\norm{\,\cdot\,}{\sigma}$,  
	$\norm{\,\cdot\,}{\sigma,\sigma}$. 
	If the assumptions stated in the 
	second part of the lemma are satisfied, 
	then applying 
	Lemma~\ref{lem:hdot-sobolev} 
	completes the proof. 
\end{proof} 

\pagebreak

\begin{lemma}\label{lem:matern:hoelder} 
	Suppose that 
	\begin{enumerate}[label=\normalfont{(\roman*)}]  
		\item\label{lem:matern:hoelder:ass-i}
			Assumptions~\ref{ass:coeff}.I--II 
			%and~\ref{ass:dom}.I  
			are satisfied, $0< 2\gamma \leq 1$, and 
			$d=1$, or 
		\item\label{lem:matern:hoelder:ass-ii} 
			Assumptions~\ref{ass:coeff}.I--III 
			and~\ref{ass:dom}.II 
			are fulfilled, 
			$d\in\{1,2,3\}$ and 
			$\gamma\in(0,1)$ are such that 
			$\gamma \leq 2 - \nicefrac{d}{2}$. 	
	\end{enumerate}
	In either of these cases and  
	if $2\beta \geq \gamma + \nicefrac{d}{2}$, 
	there exists a continuous Whittle--Mat\'ern
	field 
	$\GP^\beta\colon \clos{\cD} \times \Omega \rightarrow \bbR$ 
	satisfying \eqref{eq:gpbeta} 
	such that 
	$\GP^\beta(x) = \white(L^{-\beta}\dirac{x})$ 
	$\bbP$-a.s.\ for all~$x\in\clos{\cD}$, 
	and, for every $\delta\in(0,\gamma)$ 
	and $q\in(0,\infty)$, the bound 
	\begin{equation}\label{eq:matern:hoelder}
		\left( \bbE \Bigl[ \bigl\| \GP^\beta \bigr\|_{C^{\delta}(\clos{\cD})}^q \Bigr]
		\right)^{\nicefrac{1}{q}} 
		\lesssim_{(q,\gamma,\delta,\cD)}
		\bigl\| L^{-\beta} \bigr\|_{\cL\left( 
			\Hdot{0}; 
			\Hdot{\gamma+\nicefrac{d}{2}}\right)} 
		< \infty, 
	\end{equation}
	for the 
	$q$-th moment of $\GP^\beta$ 
	with respect to the 
	$\delta$-H\"older norm, cf.~\eqref{eq:def:hoeldernorm},  
	holds. 
\end{lemma}

\begin{proof}
	Note that by definition of $\Hdot{\sigma}$, 
	see~\eqref{eq:def:Hdot}, 
	for any $\beta > 0$, the operator 
	\[
		L^{-\beta}\from L_2(\cD) 
		= \Hdot{0} \to \Hdot{2\beta} 
	\] 
	is an isometric isomorphism. 
	For this reason,    
	$L^{-\beta}\from L_2(\cD) 
	\to \Hdot{\gamma + \nicefrac{d}{2}}$ 
	is bounded provided that  
	$2\beta \geq \gamma + \nicefrac{d}{2}$. 
	For $d$ and $\gamma$ 
	as specified 
	in~\ref{lem:matern:hoelder:ass-i}/\ref{lem:matern:hoelder:ass-ii}
	above, 
	we have  
	$\bigl(\Hdot{\gamma + \nicefrac{d}{2}}, 
	\norm{\,\cdot\,}{\gamma + \nicefrac{d}{2}} \bigr) 
	\hookrightarrow 
	\bigl( H^{\gamma + \nicefrac{d}{2}}(\cD), 
	\norm{\,\cdot\,}{H^{\gamma + \nicefrac{d}{2}}(\cD)} \bigr)$ 
	by the 
	relations~\eqref{eq:hdot-sobolev1}--\eqref{eq:hdot-sobolev2} 
	from Lemma~\ref{lem:hdot-sobolev}  
	and we conclude that  
	$L^{-\beta} \in \cL\bigl(L_2(\cD);  
	H^{\gamma + \nicefrac{d}{2}}(\cD) \bigr)$. 
	The proof is then completed by applying 
	Corollary~\ref{cor:regularity:hoelder}
	in both 
	cases~\ref{lem:matern:hoelder:ass-i}/\ref{lem:matern:hoelder:ass-ii}.  
\end{proof}

\begin{lemma}\label{lem:matern:Linf}  
	Let Assumptions~\ref{ass:coeff}.I--II 
	be satisfied and $\beta > \nicefrac{d}{4}$.  
	Suppose furthermore that 
	a system of $L_2(\cD)$-orthonormal 
	eigenvectors $\cE = \{e_j\}_{j\in\bbN}$ 
	corresponding to the eigenvalues 
	$0<\lambda_1 \leq \lambda_2 \leq \ldots$ 
	of $L$ in~\eqref{eq:def:L}  
	is uniformly bounded in $C(\clos{\cD})$, i.e., 
	\begin{equation}\label{eq:ass:ej-unibound} 	
		\exists C_\cE > 0 : 
		\quad 
		\sup_{j\in\bbN} \, \sup_{x\in\clos{\cD}} |e_j(x)| \leq C_\cE.  
	\end{equation}
	Then the covariance 
	function, cf.~\eqref{eq:def:covfct}, 
	of the Whittle--Mat\'ern field $\GP^\beta$ 
	in~\eqref{eq:gpbeta}
	has a continuous representative 
	$\varrho^\beta\from\clos{\cD\times\cD} \to \bbR$
	and 
	\[
		\sup_{x,y\in\clos{\cD}} \bigl| \varrho^\beta(x,y) \bigr| 
		\leq 
		C_\cE^2 \, \tr\left( L^{-2\beta} \right),  
	\]	
	where $\tr(\,\cdot\,)$ denotes the trace 
	on $L_2(\cD)$. 
\end{lemma} 

\begin{proof} 
	By Proposition~\ref{prop:regularity:Linf}\ref{prop:regularity:Linf-i}
	we have to show boundedness of  
	$L^{-\beta} \from L_2(\cD) \to C(\clos{\cD})$ 	
	to infer that 
	$\varrho^\beta  \in C(\clos{\cD\times\cD})$, 
	with  
	\begin{equation}\label{eq:lem:matern:Linf:proof-1}
		\sup_{x,y\in\clos{\cD}} \bigl| \varrho^\beta(x,y) \bigr| \leq 
			\bigl \| L^{-2\beta} 
			\bigr\|_{\cL(\dual{C(\clos{\cD})}; C(\clos{\cD}))}.
	\end{equation}  
	For $\psi\in L_2(\cD)$, 
	the spectral representation 
	$L^{-\beta} \psi 
		= \sum_{j\in\bbN} \lambda_j^{-\beta} 
			\scalar{\psi, e_j}{L_2(\cD)} \, e_j$	
	shows that, 
	for all $x\in\clos{\cD}$, 
	\begin{align*} 
		\bigl| \bigl(L^{-\beta} \psi\bigr)(x) \bigr| 
		 \leq C_\cE \sum_{j\in\bbN} \lambda_j^{-\beta} \,  
			\left|\scalar{\psi, e_j}{L_2(\cD)}\right| 
		 \lesssim_{(A,\kappa,\cD)} 
		 	C_\cE 
			\Biggl( \sum_{j\in\bbN} j^{-\nicefrac{4\beta}{d}} \Biggr)^{\nicefrac{1}{2}} 
			\norm{\psi}{L_2(\cD)} 
	\end{align*} 
	is finite, 
	provided that $\beta > \nicefrac{d}{4}$.  
	Here, we have used 
	the Cauchy--Schwarz inequality 
	and 
	the spectral behavior~\eqref{eq:lem:spectral-behav} 
	from Lemma~\ref{lem:spectral-behav}
	in the last estimate. 
	Similarly, 
	\begin{align}\label{eq:lem:matern:Linf:proof-2}
		\bigl| \bigl(L^{-2\beta} \varphi\bigr)(x) \bigr| 
		\leq C_\cE \sum_{j\in\bbN} \lambda_j^{-2\beta} \,  
		\bigl|
		\duality{\varphi, e_j}{\dual{C(\clos{\cD})}\times C(\clos{\cD})} 
		\bigr| 
		\leq C_\cE^2 \, 
		\tr\left( L^{-2\beta} \right)  
		\norm{\varphi}{\dual{C(\clos{\cD})}},  
	\end{align} 
	for all $\varphi\in \dual{C(\clos{\cD})}$. 
	Combining~\eqref{eq:lem:matern:Linf:proof-1} 
	and~\eqref{eq:lem:matern:Linf:proof-2} 
	completes the proof. 
\end{proof} 

\begin{remark} 
	Note that if 
	$\gamma\in(0,1)$ and $d\in\{1,2,3\}$ 
	are such that 	
	Assumption~\ref{lem:matern:hoelder:ass-i} 
	or~\ref{lem:matern:hoelder:ass-ii} 
	of Lemma~\ref{lem:matern:hoelder} is satisfied, 
	then the Sobolev embedding 
	$H^{\delta + \nicefrac{d}{2}}(\cD) 
	\hookrightarrow 
	C^{\delta}(\clos{\cD})$ 
	and 
	Lemma~\ref{lem:hdot-sobolev} 
	are applicable for any 
	$0 < \delta \leq \gamma$. 
	Thus, 
	if~$2\beta \geq \delta + \nicefrac{d}{2}$,  
	we find 
	\[
		L^{-\beta} 
		\from 
		L_2(\cD) 
		= 
		\Hdot{0} 
		\to 
		\Hdot{2\beta} 
		\hookrightarrow 
		\Hdot{\delta + \nicefrac{d}{2}} 
		\cong 
		H^{\delta + \nicefrac{d}{2}}(\cD) 
		\hookrightarrow 
		C^{\delta}(\clos{\cD}) 
		\hookrightarrow 
		C(\clos{\cD}), 
	\]
	i.e., $L^{-\beta} \from L_2(\cD) \to C(\clos{\cD})$ 
	is bounded.  
	Thus, by Proposition~\ref{prop:regularity:Linf}\ref{prop:regularity:Linf-i} 
	the covariance 
	function~$\varrho^\beta \from \clos{\cD\times\cD} \to \bbR$ 
	of the Whittle--Mat\'ern field 
	$\GP^{\beta}$ in~\eqref{eq:gpbeta} 
	is a continuous kernel and  
	$\Hdot{2\beta}$ is the corresponding 
	reproducing kernel Hilbert space, cf.~\cite{Steinwart2012}.  
\end{remark}

%=======================================================================
\section{Spectral Galerkin approximations}\label{section:matern:spectral}
%=======================================================================

In this section we investigate convergence of  
spectral Galerkin approximations 
for the Whittle--Mat\'ern 
field~$\GP^\beta$ in~\eqref{eq:gpbeta}.  
Recall that the covariance structure 
of the GRF~$\GP^\beta$ is uniquely determined 
via its color~\eqref{eq:def:colored}
given by the negative fractional 
power~$L^{-\beta}$ of 
the second-order differential 
operator~$L$ in~\eqref{eq:def:L}
which is defined with respect to the bounded spatial 
domain $\cD\subset\bbR^d$.    

For $N\in\bbN$, 
the spectral Galerkin approximation 
$\GP^\beta_{N}$ of $\GP^\beta$ 
is ($\bbP$-a.s.)\ defined
by 
\begin{equation}\label{eq:def:gpbetaN}
	\bigl( \GP^\beta_N,\psi \bigr)_{L_2(\cD)} 
	= 
	\white\bigl( L_N^{-\beta} \psi \bigr)  
	\quad 
	\bbP\text{-a.s.} 
	\quad 
	\forall\psi\in L_2(\cD),  
\end{equation}
i.e., it is a GRF colored by the 
finite-rank operator 
\begin{equation}\label{eq:def:LN} 
	L_N^{-\beta} \from L_2(\cD) \to V_N \subset L_2(\cD), 
	\qquad 
	L_N^{-\beta} \psi
	:= 
	\sum_{j=1}^N \lambda_j^{-\beta} \scalar{\psi,e_j}{L_2(\cD)} e_j, 
\end{equation} 
mapping to the finite-dimensional subspace 
$V_N:=\operatorname{span}\{e_1,\ldots,e_N\}$ 
generated by the first $N$ eigenvectors 
of~$L$ 
corresponding to the eigenvalues 
$0<\lambda_1 \leq \ldots \leq \lambda_N$. 

The following three corollaries,  
which provide explicit  
convergence rates of these approximations 
and their covariance functions  
with respect to the truncation parameter~$N$, 
are consequences of the Propositions~\ref{prop:regularity:hoelder}, 
\ref{prop:regularity:Linf} and \ref{prop:regularity:hdot}. 
We first formulate the results 
in the Sobolev norms.    
\begin{corollary}
\label{cor:spectral:conv:hdot} 
	Suppose Assumptions~\ref{ass:coeff}.I--II  
	and that 
	$d\in\bbN$, $\sigma\geq0$,  
	$\beta,q\in(0,\infty)$. 
	Let $\GP^\beta$ be the Whittle--Mat\'ern 
	field in~\eqref{eq:gpbeta}
	and, for $N\in\bbN$, let $\GP^\beta_N$ be the 
	spectral Galerkin approximation 
	in~\eqref{eq:def:gpbetaN}.
	If $2\beta-\sigma > \nicefrac{d}{2}$,  
	the following bounds hold,  
	\begin{align}
		\left(\bbE\Bigl[
		\bigl\| \GP^{\beta} - \GP^\beta_{N} \bigr\|_{\sigma}^q
		\Bigr]\right)^{\nicefrac{1}{q}} 
		&\lesssim_{(q,\sigma,\beta,A,\kappa,\cD)} 
		N^{-\nicefrac{1}{d} \, 
				\left(2\beta - \sigma 
				- \nicefrac{d}{2}\right)}, 
		\label{eq:spectral:conv:hdot} 
		\\
		\bigl\| \varrho^\beta - \varrho^\beta_{N} \bigr\|_{\sigma,
			\sigma}
		&\lesssim_{(\sigma,\beta,A,\kappa,\cD)}    
			N^{-\nicefrac{1}{d} \, 
			\left(4\beta - 2\sigma 
			- \nicefrac{d}{2}\right)},  
		\label{eq:spectral:conv:cov} 
	\end{align}
	where $\varrho^\beta, \varrho^\beta_{N}$ denote  
	the covariance functions 
	of $\GP^\beta$ and $\GP^\beta_{N}$, 
	respectively, 
	cf.~\eqref{eq:def:covfct}. 
	
	If, in addition,   		 
	Assumption~\ref{ass:dom}.I  
	and 
	$0\leq \sigma\leq 1$ 
	(or 
	Assumptions~\ref{ass:coeff}.I--III, \ref{ass:dom}.II, 
	and 
	$0\leq \sigma\leq 2$) 
	are satisfied, 
	then the assertions 
	\eqref{eq:spectral:conv:hdot}--\eqref{eq:spectral:conv:cov} 
	remain true 
	if we formulate them 
	with respect to the Sobolev norms   
	$\norm{\,\cdot\,}{H^\sigma(\cD)}$,   
	$\norm{\,\cdot\,}{H^{\sigma,\sigma}(\cD\times\cD)}$.
\end{corollary}

\begin{proof} 
	The 
	estimates~\eqref{eq:spectral:conv:hdot}/\eqref{eq:spectral:conv:cov} 
	follow from~\eqref{eq:field-err-p}/\eqref{eq:cov-fct-err} 
	of Proposition~\ref{prop:regularity:hdot}
	with $\GP := \GP^\beta$,  
	$\widetilde{\GP} := \GP^\beta_{N}$, 
	$T := L^{-\beta}$, and 
	$\widetilde{T} := L_N^{-\beta}$ 
	by exploiting the 
	spectral behavior~\eqref{eq:lem:spectral-behav} 
	from Lemma~\ref{lem:spectral-behav}. 
	Finally, 
	applying Lemma~\ref{lem:hdot-sobolev} 
	proves 
	the last claim of this proposition.  
\end{proof}

By Proposition~\ref{prop:regularity:Linf} 
we furthermore obtain 
the following convergence result 
as $N\to\infty$  
for the covariance 
function~$\varrho^\beta_{N}$ 
in the $L_{\infty}(\cD\times\cD)$-norm. 

\begin{corollary}\label{cor:spectral:conv:Linf}
	Suppose Assumptions~\ref{ass:coeff}.I--II  
	and that the system $\cE = \{e_j\}_{j\in\bbN}$ 
	of $L_2(\cD)$-orthonormal eigenvectors  
	of the operator $L$ in~\eqref{eq:def:L} 
	is uniformly bounded 
	in~$C(\clos{\cD})$ as 
	in~\eqref{eq:ass:ej-unibound}.  
	Then, for $\beta> \nicefrac{d}{4}$, 
	the covariance 
	functions
	of~$\GP^\beta$ in~\eqref{eq:gpbeta} 
	and of $\GP^{\beta}_{N}$ in~\eqref{eq:def:gpbetaN}  
	have continuous 
	representatives~$\varrho^\beta, \varrho^\beta_{N} 
	\from\clos{\cD\times\cD} \to \bbR$, and 
	\begin{equation}\label{eq:spectral:conv:Linf} 
		\sup_{x,y\in\clos{\cD}} 
			\bigl| \varrho^\beta(x,y) - \varrho^\beta_{N}(x,y) \bigr| 
		\lesssim_{(C_\cE,\beta,A,\kappa,\cD)} 
		N^{-\nicefrac{1}{d} \, 
			\left(4\beta - d \right)}. 
	\end{equation}
\end{corollary}

\begin{proof}
	By Lemma~\ref{lem:matern:Linf}, 
	$\varrho^\beta$ and $\varrho^{\beta}_{N}$ 
	have continuous representatives. 
	In addition, the 
	estimate~\eqref{eq:prop:regularity:Linf-ii}
	from Proposition~\ref{prop:regularity:Linf} 
	proves~\eqref{eq:spectral:conv:Linf} 
	since, for all   
	$x\in\clos{\cD}$, $\varphi\in \dual{C(\clos{\cD})}$, 
	\begin{align*}
		\bigl\langle \dirac{x}, 
		\bigl(L^{-2\beta} - L_N^{-2\beta}\bigr)\varphi \bigr\rangle_{ 
			\dual{C(\clos{\cD})} \times  C(\clos{\cD})} 
		&= 
		\sum_{j>N} 
		\lambda_j^{-2\beta} 
		\duality{\varphi, e_j}{\dual{C(\clos{\cD})}\times C(\clos{\cD})}
		\, e_j(x) \\
		&\leq  
		C_\cE^2 \, \norm{\varphi}{\dual{C(\clos{\cD})}} 
		\sum_{j>N} 
		\lambda_j^{-2\beta}. 	
	\end{align*} 
	Finally, for $\beta > \nicefrac{d}{4}$,  
	the spectral behavior~\eqref{eq:lem:spectral-behav} 
	of~$L$ from Lemma~\ref{lem:spectral-behav} 
	yields  
	\[
		\bigl\| L^{-2\beta} - L_N^{-2\beta} 
			\bigr\|_{\cL(\dual{C(\clos{\cD})}; C(\clos{\cD}))} 
		\lesssim_{(C_\cE,\beta,A,\kappa,\cD)} 
		N^{-\nicefrac{1}{d} \, 
			\left(4\beta - d \right)}. 
		\qedhere 
	\]
\end{proof}  

If Assumption~\ref{lem:matern:hoelder:ass-i}  
or~\ref{lem:matern:hoelder:ass-ii}  
of Lemma~\ref{lem:matern:hoelder} 
is satisfied, we obtain not only 
Sobolev regularity 
of the GRF $\GP^\beta$ in ($L_q(\Omega)$-sense), 
but also H\"older continuity. 
The next proposition 
shows that in this case  
the sequence of spectral Galerkin 
approximations~$\bigl( \GP^{\beta}_{N} \bigr)_{N\in\bbN}$ 
converges also 
with respect to these norms. 

\begin{corollary}\label{cor:spectral:conv:hoelder}
	Suppose that 
	$d\in\{1,2,3\}$, $\gamma\in(0,1)$ satisfy 
	Assumption~\ref{lem:matern:hoelder:ass-i} 
	or~\ref{lem:matern:hoelder:ass-ii} of 
	Lemma~\ref{lem:matern:hoelder},   
	and let 
	$L, L_N^{-\beta}$ be the operators 
	in~\eqref{eq:def:L} and \eqref{eq:def:LN}. 
	Then, for every $N\in\bbN$ and 
	$2\beta \geq \gamma + \nicefrac{d}{2}$, 
	there exist continuous random fields  
	$\GP^\beta, \GP^{\beta}_{N} \colon 
	 \clos{\cD}\times\Omega \rightarrow \bbR$,  
	colored by $L^{-\beta}$ and $L_N^{-\beta}$, respectively, 
	such that  
	\begin{align}\label{eq:spectral:conv:hoelder}
		\left( \bbE 
		\Bigl[ 
		\bigl\| \GP^\beta - \GP^{\beta}_{N} \bigr\|_{C^{\delta}(\clos{\cD})}^q 
		\Bigr] 
		\right)^{\nicefrac{1}{q}} 
		\lesssim_{(q,\gamma,\delta,\beta,A,\kappa,\cD)} 
		N^{-\nicefrac{1}{d} \, 
			\left(2\beta - \gamma  
			- \nicefrac{d}{2}\right)},
	\end{align} 
	for every $\delta\in(0,\gamma)$ 
	and $q\in(0,\infty)$.  
\end{corollary} 

\begin{proof}
	By Lemma~\ref{lem:matern:hoelder}
	there exist continuous random 
	fields 
	$\GP^\beta, \GP^{\beta}_{N} 
		\colon \clos{\cD} \times \Omega \rightarrow \bbR$  
	colored by $L^{-\beta}$ and $L_N^{-\beta}$, respectively. 
	Their difference 
	$\GP^\beta - \GP^{\beta}_{N}$
	is then a continuous random 
	field colored by 
	$T_N := L^{-\beta} -L_N^{-\beta}
		= (L-L_N)^{-\beta}$ 
	and we 
	obtain the convergence result  
	in~\eqref{eq:spectral:conv:hoelder} 
	from the stability 
	estimate~\eqref{eq:matern:hoelder} 
	of Lemma~\ref{lem:matern:hoelder}
	applied to $\GP^\beta - \GP^{\beta}_{N}$, 
	since, for every $\psi\in L_2(\cD)$, 
	\begin{align*} 
		\norm{T_N\psi}{\Hdot{\gamma+\nicefrac{d}{2}}}^2 
		&= 
		\sum_{j>N} 
		\lambda_j^{-2\beta + \gamma + \nicefrac{d}{2}} 
		\scalar{\psi, e_j}{L_2(\cD)}^2 
		\leq \lambda_N^{-2\beta + \gamma + \nicefrac{d}{2}} 
		\sum_{j>N} 
		\scalar{\psi, e_j}{L_2(\cD)}^2 \\
		&\lesssim_{(\gamma,\beta,A,\kappa,\cD)}  
		N^{-\nicefrac{2}{d} \, 
			\left(2\beta - \gamma  
			- \nicefrac{d}{2}\right)} 
		\norm{\psi}{L_2(\cD)}^2.   
	\end{align*}
	Here, we have used the spectral 
	behavior~\eqref{eq:lem:spectral-behav} from 
	Lemma~\ref{lem:spectral-behav} 
	for~$\lambda_N$. 
\end{proof}

%=======================================================================
\section{General Galerkin approximations}\label{section:matern:galerkin}
%=======================================================================

After having derived error estimates 
for spectral Galerkin approximations 
in the previous subsection, we now consider 
a family of general 
Galerkin approximations 
for the Whittle--Mat\'ern field~$\GP^\beta$ 
in~\eqref{eq:gpbeta} 
which, for the case $\beta\in(0,1)$, has been proposed 
in~\cite{bkk-strong, bkk-weak}. 
Recall that the random field~$\GP^\beta$ 
is indexed by   
the bounded spatial domain $\cD\subset\bbR^d$.

%=======================================================================
\subsection{Sinc-Galerkin approximations}
\label{subsec:matern:sinc-galerkin}
%=======================================================================

The approximations proposed in~\cite{bkk-strong, bkk-weak} 
are based on a Galerkin method for the spatial discretization 
$L_h$ of $L$  
and a sinc quadrature for an integral representation 
of the resulting discrete fractional inverse $L_h^{-\beta}$. 
We recall that approach in this subsection, and formulate all 
assumptions and auxiliary results 
which are needed  
for the subsequent error analysis 
in Subsection~\ref{subsec:matern:galerkin-error}. 

%=======================================================================
\subsubsection{Galerkin discretization}
%=======================================================================

We assume that we are given a family 
$(V_h)_{h > 0}$ of 
subspaces 
of $H^1_0(\cD)$, 
with dimension $N_h := \dim(V_h) <\infty$.
We let $\Pi_h \from L_2(\cD) \to V_h$  
denote the $L_2(\cD)$-orthogonal projection onto $V_h$. 
Since $V_h\subset H^1_0(\cD) = \Hdot{1}$, 
$\Pi_h$ can be uniquely extended to a bounded 
linear operator $\Pi_h\from \Hdot{-1} \to V_h$. 
In addition, we let 
$L_h \from V_h \to V_h$ be the Galerkin discretization 
of the differential operator $L$ in~\eqref{eq:def:L}
with respect to $V_h$, i.e., 
\begin{equation}\label{eq:def:Lh} 
	\scalar{L_h \phi_h, \psi_h}{L_2(\cD)} 
	= 
	\duality{L \phi_h, \psi_h}{\Hdot{-1} \times \Hdot{1}}
	\quad 
	\forall 
	\phi_h, \psi_h \in V_h. 
\end{equation}
We arrange  
the eigenvalues of~$L_h$
in nondecreasing order, 
\[
	0
	<
	\lambda_{1,h} 
	\leq 
	\lambda_{2,h} 
	\leq 
	\ldots 
	\leq 
	\lambda_{N_h,h}, 
\]
and let  
$\{e_{j,h}\}_{j=1}^{N_h}$ be a set of corresponding 
eigenvectors which are orthonormal 
in~$L_2(\cD)$. 
The operator $R_h \from H^1_0(\cD) = \Hdot{1} \to V_h$ 
is the Rayleigh--Ritz projection, i.e., 
$R_h := L_h^{-1} \Pi_h L$ and, 
for all $\psi\in\Hdot{1}$, 
\begin{equation}\label{eq:def:ritz}
	\scalar{R_h \psi, \phi_h}{1}
	= 
	\scalar{\psi , \phi_h}{1} 
	\quad 
	\forall \phi_h \in V_h. 
\end{equation} 

All further assumptions on the 
finite-dimensional subspaces $(V_h)_{h > 0}$
are 
summarized below and explicitly 
referred to, when needed 
in our error analysis.  

\begin{assumption}[on the Galerkin discretization]
	\label{ass:galerkin} 
	$\quad$ 
	
	\begin{enumerate}[label=\Roman*.] 
	\item\label{ass:galerkin-i}
		There exist $\theta_1 > \theta_0 > 0$   
		and a linear operator 
		$\cI_h \from H^{\theta_1}(\cD) \to V_h$ 
		such that,
		for all $\theta_0 < \theta \leq \theta_1$, 
		$\cI_h \from H^{\theta}(\cD) \to V_h$ 
		is a continuous extension, and 
		\begin{align}
			\label{eq:err-interpol}
			\norm{v - \cI_h v}{H^\sigma(\cD)} 
			\lesssim_{(\sigma,\theta,\cD)} 
			h^{\theta - \sigma} 
			\norm{v}{H^\theta(\cD)} 
			\quad 
			\forall 
			v \in H^{\theta}(\cD), 
		\end{align} 
		holds for $0\leq \sigma \leq \min\{1, \theta\}$ and 
		sufficiently small $h>0$.  	
	\item\label{ass:galerkin-ii}  
		For all $h>0$ sufficiently small and all  
		$0\leq\sigma\leq 1$ 
		the following inverse inequality 
		holds:  
		\begin{equation}
			\norm{\phi_h}{H^{\sigma}(\cD)} 
			\lesssim_{(\sigma,\cD)} 
			h^{-\sigma} \norm{\phi_h}{L_2(\cD)} 
			\quad 
			\forall 
			\phi_h \in V_h. 
			\label{eq:ass:galerkin:inverse}  
		\end{equation} 
	\item\label{ass:galerkin-iii} 
		$\dim(V_h) = N_h \eqsim_{\cD} h^{-d}$ 
		for sufficiently small $h>0$. 
	\item\label{ass:galerkin-iv} 
		There exist 
		$r,s_0,t,C_{0}, C_{\lambda}>0 $ 
		such that for all $h>0$ sufficiently small 
		and for all 
		$j\in\{1,\ldots,N_h\}$
		the following error estimates hold: 
		\begin{gather} 
			\lambda_j 
			\leq 
			\lambda_{j,h}
			\leq 
			\lambda_j + C_\lambda h^r \lambda_j^t, 
			\label{eq:ass:galerkin:lambda}
			\\
			\norm{e_j - e_{j,h}}{L_2(\cD)}^2 
			\leq
			C_0 h^{2s_0} \lambda_j^t,  
			\label{eq:ass:galerkin:e}
		\end{gather} 
		where $\{(\lambda_j,e_j)\}_{j\in\bbN}$ 
		are the 
		eigenpairs of the operator~$L$ 
		in~\eqref{eq:def:L}. 
	\end{enumerate}
\end{assumption} 

We refer to Subsection~\ref{subsec:matern:fem} 
for explicit examples 
of finite element spaces $(V_h)_{h>0}$, 
which satisfy these assumptions. 
\begin{remark}\label{rem:min-max} 
	It is a consequence of 
	the min-max principle that the first inequality 
	in~\eqref{eq:ass:galerkin:lambda}, 
	$\lambda_j\leq\lambda_{j,h}$, 
	is satisfied for all conforming Galerkin spaces 
	$V_h\subset \Hdot{1}$. 
\end{remark} 

In Theorem~\ref{thm:fem-error} below, 
we bound the \emph{deterministic  
Galerkin error} in the fractional case, 
i.e., we consider 
the error between $L^{-\beta}g$ 
and $L_h^{-\beta}\Pi_h g$. 
This theorem is one 
of our main results  
and it will be a crucial ingredient 
when analyzing general Galerkin 
approximations of the Whittle--Mat\'ern 
field $\GP^\beta$ from~\eqref{eq:gpbeta} 
in Subsection~\ref{subsec:matern:galerkin-error}. 
For its derivation, we need the 
following two lemmata. 

\begin{lemma}\label{lem:fem-error-integer}
	Let $L$ be as in~\eqref{eq:def:L}  
	and, for $h>0$, 
	let $L_h, R_h$  be as in~\eqref{eq:def:Lh}
	and~\eqref{eq:def:ritz}.  
	Suppose Assumptions~\ref{ass:coeff}.I--II 
	and~\ref{ass:dom}.I.  
	Let $0 < \alpha \leq 1$ 
	be such that 
	\begin{equation}\label{eq:ass:1+alpha}
		\bigl(
		\Hdot{1+\delta}, 
		\norm{\,\cdot\,}{1+\delta}
		\bigr) 
		\cong 
		\bigl(
		H^{1+\delta}(\cD)\cap H^1_0(\cD), 
		\norm{\,\cdot\,}{H^{1+\delta}(\cD)} 
		\bigr), 
		\qquad 
		0\leq \delta \leq \alpha, 
	\end{equation}
	where $\Hdot{1+\delta}$ is defined as in 
	in~\eqref{eq:def:Hdot}.  
	Furthermore, let Assumption~\ref{ass:galerkin}.I  
	be satisfied with parameters   
	$\theta_0 \in (0,1)$ and $\theta_1\geq 1+\alpha$. 
	Then, for  
	every   
	$0 \leq \eta \leq \vartheta \leq \alpha$, 
	\begin{align}
		\| u - R_h u \|_{ 
			1-\eta} 
			&\lesssim_{(\eta,\vartheta,A,\kappa,\cD)} 
			h^{\vartheta+\eta}
			\norm{u}{1+\vartheta}, 
			&& 
			u \in \Hdot{1+\vartheta}, 
		\label{eq:ritz-error} \\
		\bigl\| L^{-1} g - L_h^{-1} \Pi_h g \bigr\|_{ 
				1-\eta} 
			&\lesssim_{(\eta,\vartheta,A,\kappa,\cD)} 
			h^{\vartheta+\eta}
			\norm{g}{\vartheta-1}, 
			&& 
			g \in \Hdot{\vartheta-1}, 
		\label{eq:fem-error-integer} 
	\end{align} 
	for sufficiently small $h>0$. 
\end{lemma}  

\begin{proof}
	Since $R_h u \in V_h$  
	is the best approximation 
	of $u \in \Hdot{1}$ 
	with respect to 
	$\norm{\,\cdot\,}{1}$,  
	we find 
	by~Assumption~\ref{ass:galerkin}.I  
	and the assumed equivalence \eqref{eq:ass:1+alpha} 
	that, for $e:=u-R_h u$ 
	and any $0\leq\vartheta\leq\alpha$, 
	\begin{align*} 
		\norm{e}{1}  
		\lesssim_{(A,\kappa,\cD)}
		\norm{u - \cI_h u }{H^1(\cD)}
		\lesssim_{(\vartheta,A,\kappa,\cD)}
		h^\vartheta \norm{u}{H^{1+\vartheta}(\cD)}
		\lesssim_{(\vartheta,A,\kappa,\cD)}
		h^\vartheta \norm{u}{1+\vartheta},  
	\end{align*} 
	i.e.,~\eqref{eq:ritz-error} for $\eta=0$ follows. 
	Furthermore, if we let   
	$\psi := L^{-\vartheta} e \in \Hdot{1+2\vartheta}$, 
	the estimate above and the orthogonality of $e$
	to $V_h$ in $\Hdot{1}$, combined with \eqref{eq:hdot-sobolev1}, 
	Assumption~\ref{ass:galerkin}.I 
	and~\eqref{eq:ass:1+alpha} yield  
	\begin{align*} 
		\norm{e}{1-\vartheta}^2
		= 
		\scalar{\psi, e}{1} 
		=
		\scalar{\psi - \cI_h \psi, e}{1} 
		\leq 
		\norm{\psi - \cI_h \psi}{1}
		\norm{e}{1} 
		\lesssim_{(\vartheta,A,\kappa,\cD)}
		h^{2\vartheta} 
		\norm{u}{1+\vartheta} 
		\norm{\psi}{1+\vartheta},    
	\end{align*} 
	which proves~\eqref{eq:ritz-error} for 
	$\eta = \vartheta$ 
	since $\norm{\psi}{1+\vartheta} = \norm{e}{1-\vartheta}$. 
	For $\eta\in(0,\vartheta)$, the 
	result~\eqref{eq:ritz-error}
	holds by interpolation. 
	Now let $g\in\Hdot{\vartheta-1}$ be given. 
	Then, \eqref{eq:fem-error-integer} 
	follows from \eqref{eq:ritz-error} 
	for $u := L^{-1}g \in\Hdot{1+\vartheta}$, 
	since 
	$\norm{u}{1+\vartheta}=\norm{g}{\vartheta-1}$. 
\end{proof}

\begin{remark} 
	Note that if Assumptions~\ref{ass:coeff}.I--III 
	and \ref{ass:dom}.II 
	are satisfied, 
	i.e., if the coefficient $A$ 
	of the operator $L$ in~\eqref{eq:def:L}
	is Lipschitz continuous 
	and the domain $\cD$ is convex, 
	then the equivalence~\eqref{eq:ass:1+alpha}  
	with $\alpha=1$ is part of 
	Lemma~\ref{lem:hdot-sobolev}, 
	see \eqref{eq:hdot-sobolev2}. 
\end{remark}

\begin{lemma}\label{lem:Lh-L}  
	Suppose Assumptions~\ref{ass:coeff}.I--II 
	and~\ref{ass:dom}.I.  
	Let $L$ be as in~\eqref{eq:def:L}  
	and, for $h>0$,
	let $L_h$ be as in~\eqref{eq:def:Lh}.  
	Then, for each $0\leq \gamma \leq \nicefrac{1}{2}$, 
	we have 
	\begin{equation}\label{eq:L+Lh}
		\bigl\| L^\gamma L_h^{-\gamma} \Pi_h \bigr\|_{\cL(L_2(\cD))}
		\lesssim_\gamma 1. 
	\end{equation}
	Furthermore, if 
	the $L_2(\cD)$-orthogonal projection 
	$\Pi_h$ is $H^1(\cD)$-stable, i.e., if there exists 
	a constant $C_\Pi > 0$ such that
	\begin{equation}\label{eq:PihH1stable}
		\norm{\Pi_h}{\cL(H^1(\cD))} \leq C_\Pi, 
	\end{equation}
	for all sufficiently small $h>0$, 
	then, for such $h>0$ 
	and all $0\leq \gamma \leq \nicefrac{1}{2}$,
	\begin{equation}\label{eq:Lh+L}
		\bigl\| L_h^{\gamma} \Pi_h L^{-\gamma} \bigr\|_{\cL(L_2(\cD))} 
		\lesssim_{(\gamma,A,\kappa,\cD)} 
		1.  
	\end{equation} 
	If additionally Assumption~\ref{ass:galerkin}.II 
	is satisfied and if 
	$0<\alpha\leq 1$ is as in~\eqref{eq:ass:1+alpha}, 
	then \eqref{eq:Lh+L} 
	holds for $0\leq \gamma \leq \nicefrac{(1+\alpha)}{2}$. 
\end{lemma}

\begin{proof}
	For $g\in L_2(\cD)=\Hdot{0}$, we find 
	by the definition~\eqref{eq:def:Lh} 
	of $L_h$ that 
	\begin{align*}
		\bigl\| L^{\nicefrac{1}{2}} L_h^{-\nicefrac{1}{2}} \Pi_h g 
			\bigr\|_{0}^2 
		&= 
		\bigl\langle L L_h^{-\nicefrac{1}{2}} \Pi_h g, 
			L_h^{-\nicefrac{1}{2}} \Pi_h g 
			\bigr\rangle_{\Hdot{-1}\times \Hdot{1}} 
		= \norm{\Pi_h g}{0}^2 \leq  \norm{g}{0}^2. 
	\end{align*} 
	Thus, \eqref{eq:L+Lh} holds 
	for $\gamma\in\{0,\nicefrac{1}{2}\}$. 
	In other words,  
	the canonical embedding 
	$I_h$ of $V_h$ into $L_2(\cD)$ 
	is a continuous mapping 
	from $\Hdoth{2\gamma}$
	to $\Hdot{2\gamma}$,  
	for $\gamma\in\{0,\nicefrac{1}{2}\}$, 
	where $\Hdoth{2\gamma}$
	denotes the space~$V_h$ 
	equipped with the norm 
	$\| \,\cdot\, \|_{\Hdoth{2\gamma}} 
	:=
	\|L_h^\gamma \,\cdot\, \|_{L_2(\cD)}$. 
	Thus, 
	\[
		\bigl\| L^\gamma L_h^{-\gamma}\Pi_h \bigr\|_{\cL(L_2(\cD))}
		= 
		\| I_h \|_{\cL\left(\Hdoth{2\gamma}; \Hdot{2\gamma}\right)}
		\lesssim_\gamma  
		\| I_h \|_{\cL\left(\Hdoth{0}; \Hdot{0}\right)}^{1-2\gamma}
		\| I_h \|_{\cL\left(\Hdoth{1}; \Hdot{1}\right)}^{2\gamma}
		\leq 1,
	\]
	follows 
	by interpolation 
	for all $0\leq\gamma\leq\nicefrac{1}{2}$, 
	which completes the proof of~\eqref{eq:L+Lh}. 
	
	If $\Pi_h$ is $H^1(\cD)$-stable, 
	by Lemma~\ref{lem:hdot-sobolev}
	we have 
	$\|\Pi_h\|_{\cL\left(\Hdot{1}\right)} 
	\lesssim_{(A,\kappa,\cD)} C_\Pi$,  
	and
	\begin{align*} 
		\bigl\| L_h^{\nicefrac{1}{2}} \Pi_h L^{-\nicefrac{1}{2}} g 
			\bigr\|_{0}^2 
		&=
		\bigl(L_h \Pi_h L^{-\nicefrac{1}{2}} g , 
			\Pi_h L^{-\nicefrac{1}{2}} g \bigr)_{0}
		=
		\bigl\langle L \Pi_h L^{-\nicefrac{1}{2}} g, 
			\Pi_h L^{-\nicefrac{1}{2}}  g 
			\bigr\rangle_{\Hdot{-1}\times\Hdot{1}} 
		\\ 
		&= 
		\bigl\| \Pi_h L^{-\nicefrac{1}{2}} g \bigr\|_{1}^2 
		\lesssim_{(A,\kappa,\cD)}
		C_\Pi^2 
		\bigl\| L^{-\nicefrac{1}{2}} g \bigr\|_{1}^2
		= 
		C_\Pi^2 
		\norm{g}{0}^2 
	\end{align*} 
	follows, i.e.,   
	\eqref{eq:Lh+L} holds for 
	$\gamma\in\{0,\nicefrac{1}{2}\}$. 
	By interpreting this result 
	as continuity of $\Pi_h$ 
	as a mapping from $\Hdot{2\gamma}$ 
	to $\Hdoth{2\gamma}$,  
	again by interpolation, 
	we obtain \eqref{eq:Lh+L} for all 
	$0\leq \gamma \leq \nicefrac{1}{2}$.   
	Finally, if 
	$\gamma = \nicefrac{(1+\vartheta)}{2}$ 
	for some $0 < \vartheta \leq \alpha$, 
	we use the identity 
	\[
		L_h^{\nicefrac{(1+\vartheta)}{2}} \Pi_h 
			L^{-\nicefrac{(1+\vartheta)}{2}} 
		= 
		L_h^{-\nicefrac{(1-\vartheta)}{2}} 
		\Pi_h L^{\nicefrac{(1-\vartheta)}{2}} 
		+ 
		L_h^{\nicefrac{(1+\vartheta)}{2}} \Pi_h 
			\bigl(\operatorname{Id}_{\Hdot{1+\vartheta}} - R_h\bigr) 
			L^{-\nicefrac{(1+\vartheta)}{2}}, 
	\]
	where $R_h = L_h^{-1} \Pi_h L$ is 
	the Rayleigh--Ritz projection \eqref{eq:def:ritz}. 
	Since $0<\vartheta\leq\alpha\leq 1$, we
	obtain for the first term by~\eqref{eq:L+Lh} that 
	\[
		\bigl\| L_h^{-\nicefrac{(1-\vartheta)}{2}} 
		\Pi_h L^{\nicefrac{(1-\vartheta)}{2}} 
		\bigr\|_{\cL(L_2(\cD))}
		= 
		\bigl\| 	L^{\nicefrac{(1-\vartheta)}{2}} 
		L_h^{-\nicefrac{(1-\vartheta)}{2}} 
		\Pi_h 
		\bigr\|_{\cL(L_2(\cD))}
		\lesssim_\gamma 1. 
	\]
	To estimate the second term, 
	we write 
	$E_h^R := \operatorname{Id}_{\Hdot{1+\vartheta}} - R_h$. 
	Then, 
	\begin{align*}
		\bigl\| L_h^{\nicefrac{(1+\vartheta)}{2}} 
		&\Pi_h 
		E_h^R
		L^{-\nicefrac{(1+\vartheta)}{2}} 
		\bigr\|_{\cL(L_2(\cD))} \\
		&\leq 
		\bigl\| L_h^{\nicefrac{\vartheta}{2}} \Pi_h 
		L^{-\nicefrac{\vartheta}{2}} 
		\bigr\|_{\cL(L_2(\cD))} 
		\bigl\| L^{\nicefrac{\vartheta}{2}} 
		L_h^{\nicefrac{1}{2}} \Pi_h 
		E_h^R  
		L^{-\nicefrac{(1+\vartheta)}{2}} 
		\bigr\|_{\cL(L_2(\cD))}. 
	\end{align*}
	Here, $\| L_h^{\nicefrac{\vartheta}{2}} \Pi_h 
	L^{-\nicefrac{\vartheta}{2}} 
	\|_{\cL(L_2(\cD))} \lesssim_{(\gamma,A,\kappa,\cD)} 
	1$, 
	since $0<\vartheta=2\gamma-1\leq 1$, and we can use  
	Assumption~\ref{ass:galerkin}.II, 
	\eqref{eq:PihH1stable}, 
	and~\eqref{eq:ritz-error} 
	to conclude 
	for $\vartheta\neq\nicefrac{1}{2}$ 
	as follows,  
	\begin{align*} 
		\bigl\| L_h^{\nicefrac{1}{2}} \Pi_h 
		E_h^R  
		\bigr\|_{\cL\left(\Hdot{1+\vartheta}; \Hdot{\vartheta}\right)} 
		&\lesssim_{(\gamma,A,\kappa,\cD)} 
		h^{-\vartheta}
		\bigl\| L_h^{\nicefrac{1}{2}} \Pi_h 
		E_h^R  
		\bigr\|_{\cL\left(\Hdot{1+\vartheta}; \Hdot{0}\right)} \\
		&\lesssim_{(\gamma,A,\kappa,\cD)}  
		h^{-\vartheta}
		\bigl\| \Pi_h 
		E_h^R   
		\bigr\|_{\cL\left(\Hdot{1+\vartheta};\Hdot{1}\right)} \\
		&\lesssim_{(\gamma,A,\kappa,\cD)}  
		C_\Pi h^{-\vartheta}
		\bigl\| E_h^R   
		\bigr\|_{\cL\left(\Hdot{1+\vartheta};\Hdot{1}\right)} 
		\lesssim_{(\gamma,A,\kappa,\cD)}
		1. 
	\end{align*}
	If $\gamma=\nicefrac{3}{4}$ 
	and, thus, 
	$\vartheta=\nicefrac{1}{2}$,  
	a slight modification 
	completes the proof 
	of~\eqref{eq:Lh+L} for all 
	$\nicefrac{1}{2}<\gamma\leq\nicefrac{(1+\alpha)}{2}$. 
\end{proof}

\begin{theorem}\label{thm:fem-error}
	Let $L$ be as in~\eqref{eq:def:L}  
	and, for $h>0$, 
	let $L_h$ be as in~\eqref{eq:def:Lh}.  
	Suppose Assumptions~\ref{ass:coeff}.I--II,~\ref{ass:dom}.I,~\ref{ass:galerkin}.II  
	and that  
	$\Pi_h$ is $H^1(\cD)$-stable, 
	see~\eqref{eq:PihH1stable}. 
	Let Assumption~\ref{ass:galerkin}.I  
	be satisfied with parameters   
	$\theta_0 \in (0,1)$ and $\theta_1\geq 1+\alpha$, 
	where $0 < \alpha \leq 1$ 
	is as in~\eqref{eq:ass:1+alpha}. 
	Assume further that 
	$\beta > 0$,    
	$0\leq\sigma\leq 1$, 
	and $-1\leq \delta\leq 1+\alpha$   
	are such that 
	$2\beta + \delta - \sigma > 0$ 
	and  
	$2\beta-\sigma\geq 0$. 
	Then, for all $g\in\Hdot{\delta}$, we have 
	\begin{align}
		\hspace{-0.05cm} 
		\bigl\| L^{-\beta} g - L_h^{-\beta} \Pi_h g \bigr\|_{ 
			\sigma} 
		&\lesssim_{(\varepsilon,\delta,\sigma,\alpha,\beta,A,\kappa,\cD)} 
		h^{\min\{ 2\beta + \delta - \sigma - \varepsilon,\,
			1+\alpha-\sigma,\,
			1+\alpha+\delta, \, 
			2\alpha\}}
		\norm{g}{\delta} , 
		\label{eq:fem-error-fractional}
	\end{align} 
	for arbitrary $\eps>0$ and all $h>0$ sufficiently small. 
\end{theorem}

\begin{remark}[Sobolev bounds]\label{rem:fem-error-sobolev}
	By~\eqref{eq:hdot-sobolev1} 
	of Lemma~\ref{lem:hdot-sobolev} 
	and under the assumption given   
	by~\eqref{eq:ass:1+alpha}, 
	the result 
	\eqref{eq:fem-error-fractional} 
	implies an error bound with 
	respect to the Sobolev norms, 
	for all $0\leq\sigma\leq 1$ 
	and $-1\leq\delta\leq 1+\alpha$, 
	$\delta\neq\nicefrac{1}{2}$. 
	Namely, for all $g\in H^\delta(\cD)$, 
	\begin{align*} 
	\bigl\| L^{-\beta} g - L_h^{-\beta} \Pi_h g \bigr\|_{ 
		H^\sigma(\cD)} \hspace{-0.05cm} 
	&\lesssim_{(\varepsilon,\delta,\sigma,\alpha,\beta,A,\kappa,\cD)} 
	\hspace{-0.05cm} 
	h^{\min\{ 2\beta + \delta - \sigma - \varepsilon,\,
		1+\alpha-\sigma,\, 1+\alpha+\delta,\, 2\alpha\}}
	\norm{g}{H^\delta(\cD)}  
	\end{align*}
	for any $\eps>0$ and 
	all $h>0$ sufficiently small.  
\end{remark}

\begin{remark}[Comparison with \cite{BonitoPasciak:2015}]\label{rem:bonito-comparison}
	For the specific case  
	$\beta\in(0,1)$, $\sigma=0$, and $\delta \geq 0$ 
	the error in~\eqref{eq:fem-error-fractional} 
	has already been investigated 
	in~\cite{BonitoPasciak:2015}, 
	where $(V_h)_{h>0}$ are chosen as 
	finite element spaces with  
	continuous piecewise linear basis 
	functions, defined with respect to 
	a quasi-uniform family 
	of triangulations $(\cT_h)_{h>0}$ 
	of $\cD$. 
	If $g\in \Hdot{\delta}$, $\delta \geq 0$ and $\alpha<\beta$, 
	the results  
	of~\cite[Thm.~4.3]{BonitoPasciak:2015} 
	show convergence at the rate $2\alpha$, 
	in accordance 
	with~\eqref{eq:fem-error-fractional}. 
	For $\alpha\geq\beta$ and $g\in \Hdot{\delta}$, 
	by~\cite[Thm.~4.3 \& Rem.~4.1]{BonitoPasciak:2015} 
	\[
		\bigl\| L^{-\beta} g - L_h^{-\beta} \Pi_h g \bigr\|_{L_2(\cD)} 
		\leq 
		\begin{cases} 
		C \ln(1/h) h^{2\beta + \delta} \norm{g}{\delta} 
		& 
		\text{if } 
		0\leq \delta \leq 2(\alpha-\beta),
		\\ 
		C h^{2\alpha} \norm{g}{\delta} 
		& 
		\text{if } 
		\delta > 2(\alpha-\beta),
		\end{cases}
	\]
	i.e., compared to~\eqref{eq:fem-error-fractional}, 
	one obtains a log-term $\ln(1/h)$
	instead of $h^{-\eps}$ in the first case. 
	We point out that the purpose of  
	Theorem~\ref{thm:fem-error} 
	was to allow 
	for all $\beta > 0$ 
	and, in addition, for the wider range 
	of parameters: 
	$0\leq\sigma\leq 1$ 
	and 
	$-1\leq\delta\leq 1+\alpha$.   
\end{remark}

\begin{remark}[$p$-FEM]\label{rem:p-conv}
	Due to the term $2\alpha$ 
	and $0<\alpha\leq 1$, 
	\eqref{eq:fem-error-fractional} 
	will be sharp 
	for finite elements of first order, 
	but not for finite elements of polynomial 
	degree $p \geq 2$ when  
	$\beta > 1$ and the problem is 
	``smooth'' such that \eqref{eq:ass:1+alpha} 
	holds with $\alpha > 1$. 
\end{remark} 

\begin{proof}[Proof of Theorem \ref{thm:fem-error}]  
	We first prove~\eqref{eq:fem-error-fractional} 
	for $0\leq \delta\leq 1+\alpha$. 
	To this end,  
	let $\beta > 0$ and 
	$0\leq \sigma\leq \min\{2\beta, 1\}$  
	satisfying 
	$2\beta +\delta > \sigma$ 
	be given. 
	Without loss of generality we 
	may assume that $\eps\in\left(0,2\beta+\delta-\sigma-\alpha  
	\mathds{1}_{\{2\beta+\delta-\sigma-\alpha > 0\}}\right)$.  
	We write $I := \operatorname{Id}_{L_2(\cD)}$ and split
	\begin{align*}
		\bigl\| 
		L^{-\beta} &- L_h^{-\beta}\Pi_h  
		\bigr\|_{\cL\left(\Hdot{\delta}; \Hdot{\sigma}\right)} 
		=
		\bigl\| L^{\nicefrac{\sigma}{2}} 
		\bigl(L^{-\beta} - L_h^{-\beta}\Pi_h \bigr) 
		L^{-\nicefrac{\delta}{2}} 
		\bigr\|_{\cL(L_2(\cD))} 
		\\
		&\leq  
		\bigl\| L^{\nicefrac{\sigma}{2} - \beta}
		\bigl( I - \Pi_h \bigr) 
		L^{-\nicefrac{\delta}{2}} 
		\bigr\|_{\cL(L_2(\cD))} 
		+ 
		\bigl\| L^{\nicefrac{\sigma}{2}} 
		\bigl(L^{-\beta} - L_h^{-\beta} \bigr) \Pi_h 
		L^{-\nicefrac{\delta}{2}} 
		\bigr\|_{\cL(L_2(\cD))} \\
		&=: 
		\text{(A)} + \text{(B)}. 
	\end{align*} 

	In order to estimate term (A), 
	we first note that by 
	Assumption~\ref{ass:galerkin}.I,
	with $\theta=1+\alpha$, 
	and by~\eqref{eq:ass:1+alpha}
	we have, for $h>0$ sufficiently small,  
	\begin{align*}
		\norm{I-\Pi_h}{\cL\left(\Hdot{1+\alpha}; \Hdot{0}\right)} 
		\lesssim_{(\alpha,A,\kappa,\cD)}
		\norm{I-\Pi_h}{\cL\left( H^{1+\alpha}(\cD); L_2(\cD) \right)} 
		\lesssim_{(\alpha,A,\kappa,\cD)}
		 h^{1+\alpha}, 
	\end{align*}
	since $\Pi_h g \in V_h$ is the 
	$L_2(\cD)$-best approximation of $g\in H^\theta(\cD)$. 
	Furthermore,  
	$\norm{I-\Pi_h}{\cL(L_2(\cD))} = 1$, 
	and   
	by interpolation 
	\begin{equation*}%\label{eq:I-Pih}
		\norm{I-\Pi_h}{\cL\left(\Hdot{\theta}; \Hdot{0} \right)} 
		\lesssim_{(\theta,\alpha,A,\kappa,\cD)}
		h^{\theta}, 
		\qquad 
		0 \leq \theta\leq 1 + \alpha. 
	\end{equation*}
	By exploiting the identity 
	\[
		\bigl( L^{\nicefrac{\sigma}{2} - \beta}
		\bigl( I - \Pi_h \bigr) 
		L^{-\nicefrac{\delta}{2}} \phi, \psi \bigr)_0
		= 
		\bigl(\bigl( I - \Pi_h \bigr) 
		L^{-\nicefrac{\delta}{2}} \phi,
		\bigl( I - \Pi_h \bigr)  
		L^{\nicefrac{\sigma}{2} - \beta} \psi \bigr)_0, 
	\]
	which holds for all $\phi,\psi\in L_2(\cD)$, 
	we thus obtain,
	for all $h>0$ sufficiently small, 
	\begin{align*} 
		\text{(A)} 
		&= 
		\sup_{ \phi\in L_2(\cD)\setminus\{0\}}  
		\,
		\sup_{ \psi\in L_2(\cD)\setminus\{0\}}  
		\tfrac{1}{\norm{\phi}{0}\norm{\psi}{0}}
		\bigl( L^{\nicefrac{\sigma}{2} - \beta}
		\bigl( I - \Pi_h \bigr) 
		L^{-\nicefrac{\delta}{2}} \phi, \psi \bigr)_0 		
		\\
		&\leq 
		\bigl\|  
		I - \Pi_h 
		\bigr\|_{
			\cL\left(\Hdot{\delta}; 
			\Hdot{0} \right)} 
		\, 
		\bigl\|  
		I - \Pi_h 
		\bigr\|_{
			\cL\left(\Hdot{\theta}; 
			\Hdot{0} \right)}
		\lesssim_{(\delta,\sigma,\alpha,\beta,A,\kappa,\cD)} 
		h^{\min\{2\beta+\delta-\sigma,\, 1+\alpha+\delta\}},  
	\end{align*} 
	where we set $\theta:=\min\{2\beta-\sigma, 1+\alpha\}$ 
	and, hence, $0\leq \theta,\delta \leq 1+\alpha$.
	
	For deriving a bound for (B),  
	we first note that by~\eqref{eq:Lh+L} 
	of Lemma~\ref{lem:Lh-L} 
	\[
		\text{(B)}
		\lesssim_{(\delta,A,\kappa,\cD)}
		\bigl\| L^{\nicefrac{\sigma}{2}} 
		\bigl(L^{-\beta} - L_h^{-\beta} \bigr) 
		L_h^{-\nicefrac{\delta}{2}} \Pi_h 
		\bigr\|_{\cL(L_2(\cD))}. 
	\]
	Next, we define the contour 
	\[
		C
		:= 
		\left\{ t e^{-i\omega} : r \leq t < \infty \right\} 
		\cup 
		\left\{ r e^{i\theta} : \theta \in (-\omega, \omega) \right\}
		\cup 
		\left\{ t e^{i\omega} : r \leq t < \infty \right\}, 
	\] 
	where $\omega\in(0,\pi)$ and 
	$r := \nicefrac{\lambda_1}{2}$. 
	By, e.g., 
	\cite[Ch.~2.6, Eq.~(6.3)]{Pazy:1983}  
	we have 
	\begin{align*}
		&L^{-\beta} 
		= 
		\frac{-1}{2\pi i} 
		e^{-i\omega(1-\beta)} 
		\int_{r}^{\infty}  
		t^{-\beta} \left( L - e^{-i\omega}t I \right)^{-1} \, \rd t \\
		&\quad  
		+ 
		\frac{r^{1-\beta}}{2\pi} 		 
		\int_{-\omega}^{\omega}   
		e^{i(1-\beta)\theta} 
		\left( L-re^{i\theta} I \right)^{-1} \, \rd \theta 
		+ 
		\frac{1}{2\pi i} 
		e^{i\omega(1-\beta)} 
		\int_{r}^{\infty}  
		t^{-\beta} \left( L - e^{i\omega}t I \right)^{-1} \, \rd t. 
	\end{align*} 
	From the limit $\omega \to \pi$,  
	we then 
	obtain the representation   
	\begin{align*}
		L^{-\beta} 
		&= 
		\frac{\sin(\pi\beta)}{\pi} 
		\int_{r}^{\infty}  
		t^{-\beta} \left( t I + L \right)^{-1} \, \rd t  
		+ 
		\frac{r^{1-\beta}}{2\pi} 	 
		\int_{-\pi}^{\pi}   
		e^{i(1-\beta)\theta} 
		\left( L - r e^{i\theta} I \right)^{-1} 
		\, \rd \theta, 
	\end{align*} 
	which 
	applied to $L^{-\beta}$ 
	and $L_h^{-\beta}$ 
	(recall that $\lambda_{1,h}\geq \lambda_1$, see Remark~\ref{rem:min-max})
	implies that  
	\begin{align*}
		\bigl(L^{-\beta} - L_h^{-\beta} \bigr) \Pi_h 
		&= 
		\frac{\sin(\pi\beta)}{\pi}
		\int_r^\infty 
		t^{-\beta} 
		\bigl( (tI + L)^{-1} - (tI + L_h)^{-1} \bigr) \Pi_h  
		\, \rd t \\
		&\quad 
		+ 
		\frac{r^{1-\beta}}{2\pi} 	 
		\int_{-\pi}^{\pi}   
		e^{i(1-\beta)\theta} 
		\left( \left( L - r e^{i\theta}I \right)^{-1} 
		- \left( L_h - r e^{i\theta}I \right)^{-1} 
		\right) \Pi_h 
		\, \rd \theta. 
	\end{align*} 
	We exploit this integral representation 
	as well as 
	the identity   
	\begin{align*} 
		\left( (L-zI)^{-1} - (L_h - zI)^{-1} \right) \Pi_h
		&= 
		(L - zI)^{-1} 
		L 
		\left( L^{-1} - L_h^{-1} \Pi_h \right) 
		L_h
		(L_h - zI)^{-1} \Pi_h, 
	\end{align*} 
	which holds 
	for any $z \in \bbC$, 
	and bound 
	term (B) as follows 
	\begin{align}
		&\text{(B)}
		\lesssim_{(\delta,A,\kappa,\cD)} 
		\left(\tfrac{\sin(\pi\beta)}{\pi} + \tfrac{r^{1-\beta}}{2\pi} \right) 
		\bigl\| L^{\nicefrac{(1-\eta)}{2}} 
		(L^{-1} - L_h^{-1}) 
		L_h^{\nicefrac{(1-\vartheta)}{2}} \Pi_h \bigr\|_{\cL(L_2(\cD))} 
		\notag \\
		&\,\,\, \times \Biggl( 
		\int_r^\infty 
		t^{-\beta} 
		\bigl\| (tI+L)^{-1} L^{\mu} \bigr\|_{\cL(L_2(\cD))}  
		\bigl\| L_h^{\nu} (tI+L_h)^{-1} \Pi_h \bigr\|_{\cL(L_2(\cD))} 
		\, \rd t \notag \\
		&\quad\,    
		+ 
		\int_{-\pi}^{\pi}   
		\Bigl\| \left( L - r e^{i\theta}I \right)^{-1} L^\mu 
		\Bigr\|_{\cL(L_2(\cD))} 
		\Bigl\| L_h^\nu \left( L_h - r e^{i\theta}I \right)^{-1} \Pi_h 
		\Bigr\|_{\cL(L_2(\cD))}   
		\, \rd \theta \Biggr) ,
		\label{eq:proof:fem-error}
	\end{align} 
	where $\mu := \nicefrac{(1+\eta+\sigma)}{2}$, 
	$\nu := \nicefrac{(1+\vartheta-\delta)}{2}$   
	and $0\leq\eta\leq\vartheta\leq\alpha$ are chosen 
	as follows 
	\begin{align*} 
		\eta &:= 0, &
		\vartheta &:= 2\beta+\delta-\sigma-\varepsilon, 
		&&\text{if}\quad  
		0<2\beta+\delta-\sigma\leq\alpha, \\ 
		\eta &:= \min\{2\beta+\delta-\alpha-\varepsilon,1\}-\sigma, &
		\vartheta &:= \alpha, 
		&&\text{if}\quad  
		\alpha < 2\beta+\delta-\sigma \leq 2\alpha, \\
		\eta &:= \min\{\alpha,1-\sigma\}, &
		\vartheta &:= \alpha,  
		&&\text{if}\quad   
		2\beta+\delta-\sigma > 2\alpha.   
	\end{align*} 
	By~\eqref{eq:Lh+L} and~\eqref{eq:fem-error-integer},
	we find for  
	the term outside of the integral,    
	\begin{align*} 
		\bigl\| L^{\nicefrac{(1-\eta)}{2}} 
		(L^{-1} &- L_h^{-1}) 
		L_h^{\nicefrac{(1-\vartheta)}{2}} \Pi_h \bigr\|_{\cL(L_2(\cD))} \\
		&\leq 
		\bigl\| 
		L^{-1} - L_h^{-1} \Pi_h \bigr\|_{\cL\left( \Hdot{\vartheta-1};\Hdot{1-\eta} \right)}  
		\bigl\| L^{-\nicefrac{(1-\vartheta)}{2}}
		L_h^{\nicefrac{(1-\vartheta)}{2}} \Pi_h 
		\bigr\|_{\cL(L_2(\cD))} \\
		&\lesssim_{(\eps,\delta,\sigma,\alpha,\beta,A,\kappa,\cD)} 
		\begin{cases}
			h^{2\beta+\delta-\sigma-\varepsilon} 
			& \text{if }  0 < 2\beta+\delta-\sigma \leq \alpha, \\
			h^{\min\{ 2\beta+\delta-\sigma-\varepsilon, \, 1+\alpha-\sigma\}} 
			& \text{if } \alpha < 2\beta+\delta-\sigma\leq 2\alpha, \\
			h^{\min\{ 2 \alpha, \, 1+\alpha-\sigma\}} 
			& \text{if } 2\beta+\delta-\sigma > 2\alpha, 
		\end{cases} 
	\end{align*} 
	for $h>0$ sufficiently small, 
	where these three cases can 
	be summarized as in~\eqref{eq:fem-error-fractional}, 
	since $2\beta+\delta-\sigma-\eps < \alpha \leq 1+\alpha-\sigma$ 
	for all $0\leq\sigma\leq 1$ if $2\beta+\delta-\sigma\leq \alpha$ 
	and $2\alpha < 2\beta+\delta-\sigma-\eps$ 
	for $\eps>0$ sufficiently small 
	if $2\beta+\delta-\sigma > 2\alpha$. 
	It remains to show that the two integrals  
	in~\eqref{eq:proof:fem-error}  
	converge, uniformly in $h$.  
	To this end, 
	we first note that 
	$0\leq \mu \leq 1$ and, thus, 
	for any $t>0$, 
	\begin{align*} 
		\bigl\| 
		(tI + L)^{-1} L^{\mu} \bigr\|_{\cL(L_2(\cD))} 
		&\leq 
		\sup_{x\in[\lambda_1,\infty)}
		\frac{ x^{\mu} }{
			t + x} 
		\leq  
		\sup_{x\in[\lambda_1,\infty)}
		(t + x)^{\mu-1}
		\leq 
		t^{\mu-1}. 
	\end{align*} 
	By the same argument
	we find that 
	%
	%\begin{align*} 
		$\bigl\| L_h^{\nu} 
		(tI + L_h)^{-1} \Pi_h \bigr\|_{\cL(L_2(\cD))} 
		\leq 
		t^{\nu -1}$,  
	%\end{align*}
	for $t>0$, 
	since also $0\leq\nu\leq 1$. 
	Thus, we can bound the first integral 
	arising in~\eqref{eq:proof:fem-error} by 
	\begin{align*}
		%\int_r^\infty 
		%t^{-\beta} 
		%\bigl\| (tI+L)^{-1} L^{\mu} \bigr\|_{\cL(L_2(\cD))}  
		%&\bigl\| L_h^{\nu} (tI+L_h)^{-1} \Pi_h \bigr\|_{\cL(L_2(\cD))} 
		%\, \rd t \\
		%&<  
		\int_{\lambda_1/2}^\infty 
		t^{\mu+\nu-2-\beta} 
		\,\rd t 
		= 
		\tfrac{2^{1+\beta-\mu-\nu}}{
			(1+\beta-\mu-\nu)\lambda_1^{1+\beta-\mu-\nu}} .
	\end{align*} 
	Here, we have used that  
	$r = \nicefrac{\lambda_1}{2}$,   
	$\mu+\nu-2-\beta 
	= 
	-1 + \nicefrac{(\eta+\vartheta+\sigma-\delta-2\beta)}{2}
	\leq -1 -\nicefrac{\varepsilon}{2} 
	< -1$ if 
	$2\beta+\delta-\sigma\leq 2\alpha$,  
	and $\mu+\nu-2-\beta 
	\leq -1 - (\beta+\nicefrac{\delta}{2}-\nicefrac{\sigma}{2}-\alpha)<-1$ if 
	$2\beta+\delta-\sigma>2\alpha$. 
	To estimate the second integral 
	in~\eqref{eq:proof:fem-error}, 
	we note that, for any $z\in\bbC$ 
	with $|z| = \nicefrac{\lambda_1}{2}$, 
	\begin{align*} 
		\bigl\| 
		(L - zI)^{-1} L^{\mu} \bigr\|_{\cL(L_2(\cD))} 
		&\leq 
	%	\sup_{x\in[\lambda_1,\infty)}
	%	\frac{ x^{\mu} }{
	%		|x - z|} 
	%	\leq 
		\sup_{x\in[\lambda_1,\infty)}
		\frac{ x^{\mu} }{
			x - |z|} 
		\leq 
		\sup_{x\in[\lambda_1,\infty)}
		\frac{(x - |z|)^{\mu} + |z|^\mu}{x - |z|} 
		\leq 
		\tfrac{2^{2-\mu}}{\lambda_1^{1-\mu}}, 
	\end{align*} 
	since $(x+y)^\mu \leq x^\mu + y^\mu$ 
	if $0\leq\mu\leq 1$ and $x,y\geq 0$. 
	Similarly, for $0\leq\nu\leq 1$, 
	\begin{align*} 
		\bigl\| L_h^{\nu} 
		(L_h - zI)^{-1} \Pi_h \bigr\|_{\cL(L_2(\cD))} 
		&\leq 
		\sup_{x\in[\lambda_{1,h},\infty)}
		\frac{ x^{\nu} }{
			x - |z|}
		\leq 
		\sup_{x\in[\lambda_{1},\infty)}
		\frac{ x^{\nu} }{
			x - |z|}
		\leq 
		\tfrac{2^{2-\nu}}{\lambda_1^{1-\nu}}.  
	\end{align*}  
	With these observations, 
	we finally can bound  
	the second integral 
	in~\eqref{eq:proof:fem-error}, 
	\[
		\int_{-\pi}^{\pi}   
		\Bigl\| \left( L - r e^{i\theta}I \right)^{-1} L^\mu 
		\Bigr\|_{\cL(L_2(\cD))} 
		\Bigl\| L_h^\nu \left( L_h - r e^{i\theta}I \right)^{-1} \Pi_h 
		\Bigr\|_{\cL(L_2(\cD))}   
		\, \rd \theta  
		\leq 
		\tfrac{\pi \, 2^{5-\mu-\nu}}{\lambda_1^{2-\mu-\nu}}, 
	\]
	which completes the proof 
	of~\eqref{eq:fem-error-fractional} 
	for the case that $0\leq\delta\leq 1+\alpha$.
	
	Assume now that $\delta = -\widetilde{\sigma}$ 
	for some $0< \widetilde{\sigma}\leq 1$. 
	Then, for $g\in\Hdot{\delta}$, 
	\begin{align*}
		\bigl\| 
		L^{-\beta} g - L_h^{-\beta} \Pi_h g 
		\bigr\|_{\sigma} 
		\leq 
		\bigl\| 
		L^{\nicefrac{\sigma}{2}} 
		(L^{-\beta} - L_h^{-\beta} \Pi_h) 
		L^{\nicefrac{\widetilde{\sigma}}{2}} 
		\bigr\|_{\cL(L_2(\cD))} 
		\norm{g}{\delta}. 
	\end{align*}
	After rewriting, 
	\begin{align*} 
		L^{\nicefrac{\sigma}{2}} 
		(L^{-\beta} - L_h^{-\beta} \Pi_h) 
		L^{\nicefrac{\widetilde{\sigma}}{2}} 
		&=
		L^{\nicefrac{\sigma}{2}} 
		\bigl(L^{-(\beta-\nicefrac{\widetilde{\sigma}}{2})} 
		- L_h^{-(\beta-\nicefrac{\widetilde{\sigma}}{2})} \Pi_h \bigr) 
		L_h^{-\nicefrac{\widetilde{\sigma}}{2}} \Pi_h 
		L^{\nicefrac{\widetilde{\sigma}}{2}} \\
		&\quad + 
		L^{-\nicefrac{(2\beta - 
			\widetilde{\sigma} -\sigma )}{2}} 
		\bigl( L^{-\nicefrac{\widetilde{\sigma}}{2}} 
		- 
		L_h^{-\nicefrac{\widetilde{\sigma}}{2}} \Pi_h \bigr) 
		L^{\nicefrac{\widetilde{\sigma}}{2}}, 
	\end{align*}  
	we may exploit~\eqref{eq:fem-error-fractional}, 
	which has already been proven   
	for $0\leq\delta\leq 1+\alpha$, 
	as follows,
	\begin{align*} 
		\bigl\| 
		L^{-(\beta-\nicefrac{\widetilde{\sigma}}{2})} 
		- L_h^{-(\beta-\nicefrac{\widetilde{\sigma}}{2})} \Pi_h 
		\bigr\|_{\cL\left(\Hdot{0}; \Hdot{\sigma}\right)}
		&\lesssim_{(\eps,\widetilde{\sigma},\sigma,\alpha,\beta,A,\kappa,\cD)} 
		h^{\min\left\{2\beta-\widetilde{\sigma}-\sigma-\eps, \, 
		1+\alpha-\sigma, \, 2\alpha \right\} },  
		\\
		\bigl\| 
		L^{-\nicefrac{\widetilde{\sigma}}{2}} 
		- 
		L_h^{-\nicefrac{\widetilde{\sigma}}{2}} \Pi_h
		\bigr\|_{\cL\left(\Hdot{\widetilde{\delta}}; \Hdot{\widetilde{\sigma}}\right)} 
		&\lesssim_{(\eps,\widetilde{\sigma},\sigma,\alpha,\beta,A,\kappa,\cD)} 
		h^{\min\left\{2\beta-\widetilde{\sigma}-\sigma-\eps, \, 
			1+\alpha-\widetilde{\sigma}, \, 2\alpha \right\} }, 
	\end{align*} 
	since $\widetilde{\delta} 
	:= 2\beta-\widetilde{\sigma}-\sigma 
	= 2\beta+\delta-\sigma > 0$ by assumption. 
	Furthermore, by~\eqref{eq:L+Lh} 
	of Lemma~\ref{lem:Lh-L} 
	we have 
	$\bigl\| L_h^{-\nicefrac{\widetilde{\sigma}}{2}} 
	\Pi_h L^{\nicefrac{\widetilde{\sigma}}{2}} \bigr\|_{\cL(L_2(\cD))} 
	\lesssim_{\widetilde{\sigma}} 1$. 
	We conclude that 
	\begin{align*} 
		\bigl\| 
		L^{-\beta} - L_h^{-\beta} \Pi_h 
		\bigr\|_{\cL\left(\Hdot{\delta};\Hdot{\sigma}\right)}
		\lesssim_{(\eps,\delta,\sigma,\alpha,\beta,A,\kappa,\cD)}
		h^{\min\left\{2\beta+\delta-\sigma-\eps, \, 
			1+\alpha-\sigma, \, 1+\alpha+\delta,\, 2\alpha \right\}},
	\end{align*} 
	for the whole range of parameters 
	$\sigma,\delta$ 
	as stated in the theorem.  
\end{proof} 

%=======================================================================
\subsubsection{Sinc quadrature and fully discrete scheme}
%=======================================================================

After the Galerkin discretization (in space), 
we need a second component  
to approximate the generalized Whittle--Mat\'ern  
field~$\GP^\beta$ in~\eqref{eq:gpbeta}. 
Namely, we have to numerically realize  
a fractional inverse of the Galerkin 
operator $L_h$ in~\eqref{eq:def:Lh}. 
To this end, 
as proposed in~\cite{bkk-strong}, 
we introduce, 
for $\beta\in(0,1)$ 
and $k>0$,  
the sinc quadrature approximation 
of~$L_h^{-\beta}$ from~\cite{BonitoPasciak:2015}, 
\begin{equation}\label{eq:def:Qhk}
	Q_{h,k}^\beta \from V_h \to V_h, 
	\qquad 
	Q^\beta_{h,k} 
	:= 
	\frac{2 k \sin(\pi\beta)}{\pi} 
	\sum_{\ell=-K^{-}}^{K^{+}} 
	e^{2\beta \ell k} 
	\left( \operatorname{Id}_{V_h} + e^{2 \ell k} L_h \right)^{-1}, 
\end{equation}
where  
$K^- := \bigl\lceil \tfrac{\pi^2}{4\beta k^2} \bigr\rceil$, 
$K^+ := \bigl\lceil \frac{\pi^2}{4(1-\beta)k^2} \bigr\rceil$. 
We also formally define this operator   
for the case $\beta=0$ 
by setting~$Q^0_{h,k} := \operatorname{Id}_{V_h}$. 

For a general $\beta = \nbeta + \betafrac > 0$ 
as in~\eqref{eq:def:beta}, 
we then consider the 
approximations~$\widetilde{\GP}^\beta_{h,k}, 
\GP^\beta_{h,k}\from \cD\times\Omega \to \bbR$ 
of the Whittle--Mat\'ern field~$\GP^\beta$ 
in~\eqref{eq:gpbeta} which are ($\bbP$-a.s.)\ defined 
by  
\begin{align} 
	\left( \GP^\beta_{h,k}, \psi \right)_{L_2(\cD)}
	= \white\left( \dual{\bigl(Q^{\betafrac}_{h,k} L_h^{-\nbeta} \Pi_h \bigr)} \psi \right)  
	\quad 
	\bbP\text{-a.s.}  
	\quad 
	\forall\psi\in L_2(\cD), 
	\label{eq:def:gphk-1}
	\\ 
	\left( \widetilde{\GP}^\beta_{h,k}, \psi \right)_{L_2(\cD)}
	= \white\left( \dual{\bigl(Q^{\betafrac}_{h,k} L_h^{-\nbeta} \widetilde{\Pi}_h \bigr)} \psi \right)  
	\quad 
	\bbP\text{-a.s.}  
	\quad 
	\forall\psi\in L_2(\cD),  
	\label{eq:def:gphk-2}
\end{align}
i.e., $\GP^\beta_{h,k}$, 
$\widetilde{\GP}^\beta_{h,k}$ 
are GRFs colored by 
$Q^{\betafrac}_{h,k} L_h^{-\nbeta} \Pi_h$
and 
$Q^{\betafrac}_{h,k} L_h^{-\nbeta} \widetilde{\Pi}_h$, 
respectively,  
cf.~Definition~\ref{def:colored}. 
Here, the finite-rank 
operator $\widetilde{\Pi}_h$ 
is given by 
\begin{equation}\label{eq:def:pitilde}
	\widetilde{\Pi}_h \from L_2(\cD) \to V_h 
	\subset L_2(\cD), 
	\qquad 
	\widetilde{\Pi}_h \psi 
	:= \sum_{j=1}^{N_h} 
	\scalar{\psi, e_j}{L_2(\cD)} 
	e_{j,h}.
\end{equation} 

For $\beta\in (0,1)$, 
the construction~\eqref{eq:def:gphk-2} 
of~$\widetilde{\GP}^\beta_{h,k}$ 
gives the same approximation as 
considered in~\cite{bkk-strong, bkk-weak}. 
Note furthermore that, in contrast to $\Pi_h$, 
the operator~$\widetilde{\Pi}_h$ in~\eqref{eq:def:pitilde}
is neither a projection 
nor self-adjoint, 
and its definition depends on the 
particular choice of the eigenbases 
$\{e_j\}_{j\in\bbN}\subset L_2(\cD)$ and 
$\{e_{j,h}\}_{j=1}^{N_h} \subset V_h$. 
The reason for why we consider 
both approximations 
$\GP^\beta_{h,k}$, 
$\widetilde{\GP}^\beta_{h,k}$
will become 
apparent in the error analysis 
of Subsection~\ref{subsec:matern:galerkin-error}. 
Although, in general, they do not coincide 
in $L_q(\Omega;L_2(\cD))$-sense, i.e., 
\[
	\bbE \Bigl[ 
	\bigl\| \GP^\beta_{h,k} - \widetilde{\GP}^\beta_{h,k}  \bigr\|_{L_2(\cD)}^q 
	\Bigr] \neq 0,
\] 
they have the same Gaussian distribution 
as shown in the following lemma. 

\begin{lemma}\label{lem:galerkin:equal-approx} 
	Suppose Assumptions~\ref{ass:coeff}.I--II.   
	Let $\Pi_h$ denote the 
	$L_2(\cD)$-orthogonal projection onto~$V_h$, 
	and $\widetilde{\Pi}_h$ be the operator 
	in~\eqref{eq:def:pitilde}. 
	Then, 
	if $T_h\in\cL(V_h)$, 
	\begin{equation}\label{eq:equalCov}
		(\dual{(T_h \Pi_h)} \phi, \dual{(T_h \Pi_h)}\psi)_{L_2(\cD)} 
		= 
		\bigl( 
			\dual{\bigl( T_h \widetilde{\Pi}_h \bigr)} \phi, 
			\dual{\bigl(T_h \widetilde{\Pi}_h \bigr)} \psi 
		\bigr)_{L_2(\cD)} 
	\end{equation}
	holds for all 
	$\phi,\psi\in L_2(\cD)$. 
	In particular, 
	$\GP^\beta_{h,k}
	\overset{d}{=} 
	\widetilde{\GP}^\beta_{h,k}$ 
	 as $L_2(\cD)$-valued random variables, where 
	$\GP^\beta_{h,k}$ 
	and~$\widetilde{\GP}^\beta_{h,k}$ are as defined  
	in~\eqref{eq:def:gphk-1}--\eqref{eq:def:gphk-2}. 
\end{lemma} 

\begin{proof}
	Note that $\dual{(T_h \Pi_h)}$ 
	(resp.\ $\dual{(T_h \widetilde{\Pi}_h )}$) 
	denotes the adjoint of $T_h \Pi_h$ 
	(resp.\ of $T_h \widetilde{\Pi}_h $) 
	when interpreted as
	an operator in $\cL(L_2(\cD))$. 
	This means,
	we are identifying $T_h \Pi_h$ with 
	$I_h T_h \Pi_h$ 
	(resp.\ $T_h \widetilde{\Pi}_h$ with $I_h T_h \widetilde{\Pi}_h$),
	where $I_h$ denotes the canonical embedding of $V_h$ 
	into $L_2(\cD)$. 
	Since $\dual{I_h} = \Pi_h$ we thus find that   
	$\dual{(T_h \Pi_h)} = \dual{T_h} \Pi_h$ and 
	$\dual{(T_h  \widetilde{\Pi}_h)} 
	= \dual{\widetilde{\Pi}_h}\dual{ T_h} \Pi_h$, 
	which combined with 
	$\widetilde{\Pi}_h \dual{\widetilde{\Pi}_h} 
	= \operatorname{Id}_{V_h}$ 
	proves~\eqref{eq:equalCov}. 
	
	By definition of~$\GP_{h,k}^\beta$,~$\widetilde{\GP}_{h,k}^\beta$ 
	in~\eqref{eq:def:gphk-1}--\eqref{eq:def:gphk-2}, 
	for $M\in\bbN$  
	and $\psi_1, \ldots, \psi_M \in L_2(\cD)$, 
	the random vectors~$\zvec$,~$\widetilde{\zvec}$
	with entries 
	$z_j = \bigl( \GP_{h,k}^\beta, \psi_j \bigr)_{L_2(\cD)}$ 
	and   
	$\widetilde{z}_j = \bigl(  
		\widetilde{\GP}_{h,k}^\beta, \psi_j \bigr)_{L_2(\cD)}$, 
	$1\leq j \leq M$,   
	are multivariate Gaussian distributed. 
	Furthermore, both vanish in expectation and 
	their covariance matrices, 
	$\Cmat := \operatorname{Cov}(\zvec)$ and 
	$\widetilde{\Cmat} := \operatorname{Cov}\bigl(\widetilde{\zvec}\bigr)$, 
	coincide due to~\eqref{eq:equalCov} 
	applied to 
	$T_h := Q^{\betafrac}_{h,k} L_h^{-\nbeta}$. 
	This shows that 
	$\GP_{h,k}^\beta
	\overset{d}{=} 
	\widetilde{\GP}_{h,k}^\beta$ as 
	$L_2(\cD)$-valued random variables.
\end{proof}

\begin{remark}[Simulation in practice]\label{rem:simulation}
	To simulate samples 
	of the 
	in~\eqref{eq:def:gphk-1}--\eqref{eq:def:gphk-2} 
	abstractly defined  
	($\bbP$-a.s.)\ 
	$V_h$-valued Gaussian 
	random variables 
	$\GP_{h,k}^\beta$ and  
	$\widetilde{\GP}_{h,k}^\beta$
	in practice, 
	in both cases, 
	one first has to generate 
	a sample of a multivariate 
	Gaussian random vector~$\bvec$ 
	with mean~$\zerovec$ 
	and covariance matrix~$\Mmat$, 
	where~$\Mmat$ is the 
	Gramian with 
	respect to any fixed 
	basis~$\Phi_h = \{\phi_{j,h}\}_{j=1}^{N_h}$ of~$V_h$, i.e., 
	$M_{ij} := \scalar{\phi_{i,h}, \phi_{j,h}}{L_2(\cD)}$. 
	This follows from the identical distribution 
	of the GRFs $\GP^0_{h}$ 
	and $\widetilde{\GP}^0_{h}$
	colored by $\Pi_h$ and $\pitilde$, 
	respectively,  
	as well as from the 
	chain of equalities   
	\begin{align*}
		\bbE[ \white(\dual{\Pi_h} \phi_{i,h})
			\white(\dual{\Pi_h} \phi_{j,h}) ] 
		&= 
		\scalar{\dual{\Pi_h}\phi_{i,h}, \dual{\Pi_h}\phi_{j,h} }{L_2(\cD)}  
		=
		\scalar{\phi_{i,h}, \phi_{j,h} }{L_2(\cD)}  
		= M_{ij} \\
		&=
		\bigl( \dual{\pitilde}\phi_{i,h}, 
		\dual{\pitilde}\phi_{j,h} \bigr)_{L_2(\cD)}
		=
		\bbE\left[ \white\bigl(\dual{\pitilde} \phi_{i,h} \bigr) 
		\white\bigl(\dual{\pitilde} \phi_{j,h} \bigr) \right], 
	\end{align*}
	which we obtain 
	from Lemma~\ref{lem:galerkin:equal-approx}  
	with $T_h := \operatorname{Id}_{V_h}$. 
	Since 
	\[
		\GP_{h,k}^\beta 
		\overset{d}{=} 
		Q_{h,k}^{\betafrac} L_h^{-\nbeta} \GP^0_h
		\overset{d}{=}
		Q_{h,k}^{\betafrac} L_h^{-\nbeta} \widetilde{\GP}^0_h
		\overset{d}{=} 
		\widetilde{\GP}_{h,k}^\beta, 
	\]  
	the random vector $\GPvec_k^{\beta}$, given by 
	\begin{equation}\label{eq:Zbetavec} 
		\GPvec_k^{\beta}
		:= 
		\begin{cases}
		\Lmat^{-1} \bigl( \Mmat\Lmat^{-1} \bigr)^{\nbeta-1} \bvec, 
		& 
		\text{if } 
		\betafrac = 0, \\
		\Qmat_{k}^{\beta_\star} 
		\bigl( \Mmat\Lmat^{-1} \bigr)^{\nbeta} 
		\bvec, 
		& 
		\text{if } 
		\betafrac \in (0,1), 
		\end{cases} 
	\end{equation} 
	is then the vector of coefficients 
	when expressing the $V_h$-valued sample 
	of $\GP_{h,k}^\beta$ 
	(or of~$\widetilde{\GP}_{h,k}^\beta$)
	with respect to the basis $\Phi_h$. 
	Here, $\Lmat\in\bbR^{N_h \times N_h}$  
	represents the action of 
	the Galerkin operator $L_h$ 
	in~\eqref{eq:def:Lh}, i.e., 
	$L_{ij} := \scalar{L_h \phi_{j,h}, \phi_{i,h}}{L_2(\cD)}$, 
	and, for $\betafrac\in(0,1)$, 
	$\Qmat^{\betafrac}_{k}\in\bbR^{N_h \times N_h}$ is 
	the matrix analog  
	of the operator $Q_{h,k}^{\betafrac}$ 
	from~\eqref{eq:def:Qhk}, i.e., 
	\begin{equation}\label{eq:Qmat} 
		\Qmat^{\betafrac}_{k} 
		:= 
		\frac{2 k \sin(\pi\betafrac)}{\pi} 
		\sum_{\ell=-K^{-}}^{K^{+}} 
		e^{2\betafrac \ell k} 
		\left( \Mmat + e^{2 \ell k} \Lmat \right)^{-1}. 
	\end{equation}
\end{remark}

%======================================================================
\subsection{Error analysis}\label{subsec:matern:galerkin-error}
%======================================================================

The errors 
$\GP^\beta - \GP_{h,k}^\beta$ and 
$\GP^\beta - \widetilde{\GP}_{h,k}^\beta$
of the approximations 
in~\eqref{eq:def:gphk-1}--\eqref{eq:def:gphk-2} 
compared to the true Whittle--Mat\'ern 
field~$\GP^\beta$ from~\eqref{eq:gpbeta} 
are GRFs colored 
(see Definition~\ref{def:colored}) by 
\[ 
	E_{h,k}^\beta 
	:= 
	L^{-\beta} - Q_{h,k}^{\betafrac} L_h^{-\nbeta} \Pi_h 
	\qquad 
	\text{and}
	\qquad 
	\widetilde{E}_{h,k}^\beta  
	:= 
	L^{-\beta} - Q_{h,k}^{\betafrac} L_h^{-\nbeta} \pitilde,  
\]
respectively. 
In order to perform the error analysis  
for $\GP_{h,k}^\beta$ and $\widetilde{\GP}_{h,k}^\beta$, 
we split these operators  
as follows 
\[ 
	E_{h,k}^\beta 
	= 
	E_{V_h}^\beta + E_Q^\beta 
	\qquad 
	\text{and}
	\qquad  
	\widetilde{E}_{h,k}^\beta 
	= 
	\widetilde{E}_{N_h}^\beta 
	+ \widetilde{E}_{V_h}^\beta 
	+ \widetilde{E}_Q^\beta, 
\]
where $\widetilde{E}_{N_h}^\beta := 
L^{-\beta} - L_{N_h}^{-\beta}$ is a 
dimension truncation error 
(recall the finite-rank operator 
$L_{N_h}^{-\beta}$
from~\eqref{eq:def:LN})
which can be estimated 
with the results from 
Section~\ref{section:matern:spectral}
on spectral Galerkin approximations. 
Furthermore, we shall refer to 
\begin{align}  
	E_{V_h}^\beta 
	&:= 
	L^{-\beta} - L_{h}^{-\beta} \Pi_h,  
	&
	\widetilde{E}_{V_h}^\beta 
	&:= 
	L_{N_h}^{-\beta} - L_{h}^{-\beta} \pitilde, 
	\label{eq:def:err-galerkin}  
	\\
	E_Q^\beta 
	&:= 
	\bigl(L_{h}^{-\beta}  
	- 
	Q_{h,k}^{\betafrac} L_h^{-\nbeta} \bigr) \Pi_h, 
	&
	\widetilde{E}_Q^\beta 
	&:= 
	\bigl( L_{h}^{-\beta}   
	- 
	Q_{h,k}^{\betafrac} L_h^{-\nbeta} \bigr) \pitilde, 
	\label{eq:def:err-quad}
\end{align}
as the Galerkin errors  
and as the quadrature errors, respectively. 

In the following 
we provide 
error estimates   
for both approximations, 
$\GP_{h,k}^\beta$ 
and 
$\widetilde{\GP}_{h,k}^\beta$ 
in~\eqref{eq:def:gphk-1}--\eqref{eq:def:gphk-2},
with respect to the norm on 
$L_q(\Omega;H^\sigma(\cD))$    
as well as 
for its covariance 
functions~$\varrho_{h,k}^\beta$, 
$\widetilde{\varrho}_{h,k}^\beta$ 
in the mixed Sobolev 
norm, cf.~\eqref{eq:def:Hmixed}.   
By exploiting Theorem~\ref{thm:fem-error} 
the bounds for $\GP_{h,k}^\beta$ and $\varrho_{h,k}^\beta$ 
in Proposition~\ref{prop:galerkin:conv:sobolev-alpha} 
below will be sharp if 
a conforming finite element method with  
piecewise linear basis functions is used. 
However, to derive optimal rates for 
the case of finite elements of higher 
polynomial degree, a different approach 
will be necessary, cf.~Remark~\ref{rem:p-conv}. 
To this end, we perform an 
error analysis for  
$\widetilde{\GP}_{h,k}^\beta$ 
and $\widetilde{\varrho}_{h,k}^\beta$ 
based on spectral expansions, 
see Proposition~\ref{prop:galerkin:conv:sobolev-p}. 
Since these arguments work 
only if the differential operator $L$ 
in~\eqref{eq:def:L} is at least $H^2(\cD)$-regular, 
both approaches and results 
are needed for a complete 
discussion of smooth vs.\ $H^{1+\alpha}(\cD)$-regular 
problems in Subsection~\ref{subsec:matern:fem}.\label{p:gphk12discussion}  
Finally, in Proposition~\ref{prop:galerkin:conv:hoelder}, 
we use the approximation 
$\GP_{h,k}^\beta$ from~\eqref{eq:def:gphk-1}
to formulate a convergence result 
with respect to the H\"older 
norm~\eqref{eq:def:hoeldernorm}
in $L_q(\Omega)$-sense and 
with respect to the 
$L_\infty(\cD\times\cD)$-norm 
for its covariance 
function $\varrho^\beta_{h,k}$. 

We note that, 
at the cost of other 
assumptions (e.g., $\alpha>\nicefrac{1}{2}$) 
on the parameters involved, 
it is possible to 
circumvent the additional condition  
$\beta>1$ (instead of $\beta>\nicefrac{3}{4}$) 
needed in the following proposition for the 
$L_q(\Omega;H^\sigma(\cD))$-estimate if $d=3$. 

\begin{proposition}\label{prop:galerkin:conv:sobolev-alpha} 
	Suppose Assumptions~\ref{ass:coeff}.I--II,~\ref{ass:dom}.I, 
	\ref{ass:galerkin}.II--III, 
	and let Assumption~\ref{ass:galerkin}.I  
	be satisfied with parameters   
	$\theta_0\in(0,1)$ and $\theta_1 \geq 1+\alpha$, 
	where $0 < \alpha \leq 1$ 
	is as in~\eqref{eq:ass:1+alpha}. 
	Assume furthermore that 
	$\Pi_h$ is $H^1(\cD)$-stable, 
	see~\eqref{eq:PihH1stable},  
	and that $d\in\{1,2,3\}$, 
	$\beta>0$ and $0\leq\sigma\leq 1$ 
	are such that $2\beta-\sigma > \nicefrac{d}{2}$. 
	Let~$\GP^\beta$ be the Whittle--Mat\'ern 
	field in~\eqref{eq:gpbeta}
	and, for $h, k>0$, 
	let $\GP^\beta_{h,k}$ be the 
	sinc-Galerkin approximation 
	in~\eqref{eq:def:gphk-1}, 
	with covariance 
	functions $\varrho^\beta$ and 
	$\varrho_{h,k}^\beta$, respectively.
	Then, for 
	every $q,\eps > 0$ 
	and sufficiently small $h>0$, 
	\begin{align} 
		&\Bigl(\bbE\Bigl[
		\bigl\| \GP^{\beta} - \GP^\beta_{h,k} 
		\bigr\|_{H^\sigma(\cD)}^q
		\Bigr]\Bigr)^{\nicefrac{1}{q}} 
		\notag \\
		&\,\,\,\lesssim_{(q,\eps,\sigma,\alpha,\beta,A,\kappa,\cD)} 
		\Bigl( 
		h^{\min\left\{ 2\beta - \sigma - \nicefrac{d}{2} - \varepsilon, \, 
			1 + \alpha - \sigma, \, 
			2\alpha \right\} } 
		 + 
		e^{-\nicefrac{\pi^2}{(2k)}} 
		h^{-\sigma-\nicefrac{d}{2}\,\mathds{1}_{\{\beta < 1\}}} 
		\Bigr),  
		\label{eq:galerkin:conv:field-alpha} \\
		&\bigl\| \varrho^\beta 
		- \varrho^\beta_{h,k} 
		\bigr\|_{H^{\sigma,\sigma}(\cD\times\cD)} 
		\notag \\
		&\,\,\,\lesssim_{(\eps,\sigma,\alpha,\beta,A,\kappa,\cD)} 
		\Bigl( 
		h^{\min\left\{ 4\beta - 2\sigma - \nicefrac{d}{2} - \varepsilon, \, 
			1 + \alpha - \sigma, \, 
			2\alpha \right\} } 
		+ 
			e^{-\nicefrac{\pi^2}{(2k)}} 
			h^{-2\sigma-\nicefrac{d}{2}\,\mathds{1}_{\{\beta < 1\}}}  
		\Bigr),  
		\label{eq:galerkin:conv:cov-alpha}
	\end{align} 
	where, if $d=3$, for 
	\eqref{eq:galerkin:conv:field-alpha} to hold, 
	we also suppose 
	that $\beta>1$
	and $\alpha\geq\nicefrac{1}{2}-\sigma$. 
\end{proposition}

\begin{proof} 
	We start with splitting the errors   
	with respect to the norms   
	on $\Hdot{\sigma}$, cf.~\eqref{eq:def:Hdot}, 
	\begin{align*} 
		\left(\bbE\left[
		\bigl\| \GP^{\beta} - \GP^\beta_{h,k} 
		\bigr\|_{\sigma}^q 
		\right]\right)^{\nicefrac{1}{q}} 
		&\leq 
		\left(\bbE\left[
		\bigl\| \GP^{\beta} - \GP^{\beta}_h 
		\bigr\|_{\sigma}^q 
		\right]\right)^{\nicefrac{1}{q}} 
		+ 
		\left(\bbE\left[
		\bigl\| \GP^{\beta}_h 
		- \GP^\beta_{h,k} 
		\bigr\|_{\sigma}^q 
		\right]\right)^{\nicefrac{1}{q}}  \\
		& 
		=: \text{($\text{A}_\GP$)} + \text{($\text{B}_\GP$)},  
	\end{align*} 
	and on $\Hdot{\sigma,\sigma}$, 
	see~\eqref{eq:def:Hdotmixed},  
	respectively,  
	\begin{align} 
		\label{eq:A-rho-B-rho} 
		\bigl\| \varrho^\beta 
		- \varrho^\beta_{h,k} 
		\bigr\|_{\sigma,\sigma} 
		&\leq 
		\bigl\| \varrho^\beta 
		- \varrho^\beta_{h} 
		\bigr\|_{\sigma,\sigma} 
		+ 
		\bigl\| \varrho_h^\beta 
		- \varrho^\beta_{h,k} 
		\bigr\|_{\sigma,\sigma} 
		=: \text{($\text{A}_\varrho$)} + \text{($\text{B}_\varrho$)},  
	\end{align} 
	which by~\eqref{eq:hdot-sobolev1}
	of Lemma~\ref{lem:hdot-sobolev} 
	bound the 
	errors~\eqref{eq:galerkin:conv:field-alpha}--\eqref{eq:galerkin:conv:cov-alpha}
	in the Sobolev norms. 
	Here $\GP^\beta_{h}$
	denotes a GRF colored by 
	$L_h^{-\beta}\Pi_h$, 
	with covariance function~$\varrho^\beta_h$. 
	Furthermore, we note the 
	following: 
	For $m\geq 0$, we have 
	\begin{align*}
		\norm{L_h^{-m}\Pi_h}{\cL_2(L_2(\cD))}^2 
		%=
		%\sum_{j\in\bbN} 
		%\norm{L_h^{-m}\Pi_h e_j}{L_2(\cD)}^2 
		&= 
		%\sum_{j\in\bbN} 
		%\sum_{\ell=1}^{N_h} 
		%\scalar{L_h^{-m}\Pi_h e_j, e_{\ell,h}}{L_2(\cD)}^2 
		%= 
		\sum_{\ell=1}^{N_h} 
		\lambda_{\ell,h}^{-2m}
		\leq 
		\sum_{\ell=1}^{N_h} 
		\lambda_{\ell}^{-2m}, 
	\end{align*} 
	where the observation 
	of Remark~\ref{rem:min-max} 
	was used in the last step. 
	Thus, by the spectral asymptotics from 
	Lemma~\ref{lem:spectral-behav} 
	and by Assumption~\ref{ass:galerkin}.III 
	we have for $m\geq 0$, $m\neq\nicefrac{d}{4}$, 
	\begin{align}\label{eq:sum-Pihm} 
	\norm{L_h^{-m}\Pi_h}{\cL_2(L_2(\cD))}  
	&\lesssim_{(m,A,\kappa,\cD)}
	\max\bigl\{ h^{2m-\nicefrac{d}{2}}, 1 \bigr\}. 
	\end{align} 

	For terms ($\text{A}_\GP$) and ($\text{B}_\GP$), 
	we obtain 
	with the definitions 
	of the Galerkin and 
	quadrature errors 
	${E}_{V_h}^\beta, 
	{E}_{Q}^\beta$ 
	from~\eqref{eq:def:err-galerkin}--\eqref{eq:def:err-quad}  
	by~\eqref{eq:field-err-p} 
	of Proposition~\ref{prop:regularity:hdot} 
	that 
	\begin{align*} 
		\text{($\text{A}_\GP$)}   
		\lesssim_{q} 
		\bigl\| E_{V_h}^\beta 
		\bigr\|_{\cL_2^{0;\sigma}}
		\quad 
		\text{and} 
		\quad  
		\text{($\text{B}_\GP$)} 
		\lesssim_{q} 
		\bigl\| E_{Q}^\beta 
		\bigr\|_{\cL_2^{0;\sigma}}, 
		\qquad 
		\cL_2^{\theta;\sigma} := 
		\cL_2\bigl(\Hdot{\theta};\Hdot{\sigma}\bigr). 
	\end{align*}

	For bounding term ($\text{A}_\GP$), 
	we let $\gamma\in(0,\beta)$ and 
	rewrite $E_{V_h}^\beta$ from~\eqref{eq:def:err-galerkin} 
	as follows,  
	\begin{align}\label{eq:A-GP-gamma}
		E_{V_h}^\beta
		%= 
		%L^{-\beta} - L_h^{-\beta} \Pi_h 
		&=
		\bigl( L^{-(\beta-\gamma )} 
		- L_h^{-(\beta-\gamma )}\Pi_h 
		\bigr) 
		L_h^{-\gamma} \Pi_h 
		+ 
		L^{-(\beta - \gamma)} 
		\bigl(L^{-\gamma} 
		- L_h^{-\gamma} \Pi_h \bigr).  
	\end{align}
	We first bound ($\text{A}_\GP$) 
	for $d\in\{1,2\}$. 
	To this end, 
	let $\eps_{0}>0$ be chosen sufficiently small 
	such that 
	$2\beta-\sigma-\nicefrac{d}{2} > 4\eps_0$ 
	and choose $\gamma:=\nicefrac{d}{4}+\eps_0$ 
	in~\eqref{eq:A-GP-gamma}. 
	We obtain thus 
	$\text{($\text{A}_\GP$)} 
	\lesssim_q 
	\text{($\text{A}'_\GP$)}
	+ 
	\text{($\text{A}''_\GP$)}$, 
	where 
	\begin{align*} 
		\text{($\text{A}'_\GP$)}
		&:= 
		\bigl\| 
		L^{\nicefrac{\sigma}{2}} 
		\bigl( L^{-(\beta-\nicefrac{d}{4}-\eps_0 )} 
		- L_h^{-(\beta-\nicefrac{d}{4}-\eps_0 )}\Pi_h 
		\bigr) 
		L_h^{-(\nicefrac{d}{4}+\eps_0 )} \Pi_h 
		\bigr\|_{\cL_2^{0;0}} , \\
		\text{($\text{A}''_\GP$)} 
		&:= 
		\bigl\| 
		L^{-(\beta -\nicefrac{\sigma}{2} - \nicefrac{d}{4} - \eps_0)} 
		\bigl(L^{-(\nicefrac{d}{4} + \eps_0)} 
		- L_h^{-(\nicefrac{d}{4} + \eps_0)} \Pi_h \bigr) 
		\bigr\|_{\cL_2^{0;0}}. 
	\end{align*}

	For ($\text{A}'_\GP$), 
	we find by~\eqref{eq:fem-error-fractional}  
	of Theorem~\ref{thm:fem-error} 
	and by \eqref{eq:sum-Pihm}, 
	applied for the parameters 
	$\beta' 
	:= \beta-\nicefrac{d}{4}-\eps_0$, 
	$\sigma'  := \sigma$, 
	$\delta' := 0$, and 
	$m=\nicefrac{d}{4} + \eps_0$, 
	respectively, 
	\begin{align*} 
		\text{($\text{A}'_\GP$)} 
		&\leq 
		\bigl\| 
		L^{-(\beta-\nicefrac{d}{4}-\eps_0 )} 
		- L_h^{-(\beta-\nicefrac{d}{4}-\eps_0 )}\Pi_h 
		\bigr\|_{\cL\left(\Hdot{0}; \Hdot{\sigma}\right)} 
		\bigl\| 
		L_h^{-(\nicefrac{d}{4}+\eps_0 )} \Pi_h 
		\bigr\|_{\cL_2^{0;0}}	\\
		&\lesssim_{(\eps_0,\eps',
			\sigma,\alpha,\beta,A,\kappa,\cD)} 
		h^{\min\{2\beta - \sigma -\nicefrac{d}{2} 
			- 2\eps_0 - \eps', \, 
			1+\alpha-\sigma,\, 2\alpha\}}, 
	\end{align*} 
	for any $\eps' > 0$ 
	and sufficiently small $h>0$. 

	After rewriting term ($\text{A}''_\GP$) 
	we again apply \eqref{eq:fem-error-fractional}
	of Theorem~\ref{thm:fem-error}, 
	this time for the 
	parameters 
	$\beta'' := \nicefrac{d}{4} + \eps_0 > 0$, 
	$\sigma'' := 0$, and 
	$\delta'' := \min\{2\beta-\sigma-d-4 \eps_0, 
	1+\alpha\}$. 
	Note that, due to the choice of 
	$\eps_0>0$ and since $d\in\{1,2\}$,
	we have $\delta''>-1$ and  
	\begin{gather*} 
		2\beta'' - \sigma'' + \delta'' 
		= 
		\min\{2\beta-\sigma-\nicefrac{d}{2}-2\eps_0, 
		1+\alpha+\nicefrac{d}{2} + 2\eps_0\} 
		> 
		2\eps_0 > 0.   
	\end{gather*} 
	We thus find 
	that, for any $\eps''>0$ and 
	sufficiently small $h>0$, 
	\begin{align*} 
		\text{($\text{A}''_\GP$)} 
		&\leq 
		\bigl\| 
		\bigl(L^{-(\nicefrac{d}{4} + \eps_0)} 
		- L_h^{-(\nicefrac{d}{4} + \eps_0)} \Pi_h \bigr) 
		L^{-(\beta -\nicefrac{\sigma}{2} - \nicefrac{d}{2} - 2\eps_0)} 
		\bigr\|_{\cL(L_2(\cD))} 
		\bigl\| 
		L^{-(\nicefrac{d}{4}+\eps_0)} 
		\bigr\|_{\cL_2^{0;0}} \\
		&\lesssim_{(\eps_0,\eps'',  
			\sigma,\alpha,\beta,A,\kappa,\cD)} 
		h^{\min\{2\beta - \sigma -\nicefrac{d}{2} 
			- 2 \eps_0 - \eps'', \, 
			1+\alpha+\delta'',\, 2\alpha\}} 
		\bigl\| 
		L^{-(\nicefrac{d}{4}+\eps_0)}
		\bigr\|_{\cL_2^{0;0}}. 
	\end{align*} 
	The Hilbert--Schmidt norm 
	$\| 
	L^{-(\nicefrac{d}{4}+\eps_0)}
	\|_{\cL_2^{0;0}}$ 
	converges 
	for any $\eps_0>0$ 
	due to the spectral asymptotics~\eqref{eq:lem:spectral-behav}
	of Lemma~\ref{lem:spectral-behav}. 
	In addition, since $1+\alpha > \nicefrac{d}{2}$ for $d\in\{1,2\}$, 
	we find that 
	$1+\alpha+\delta'' 
	> 
	\min\{2\beta-\sigma-\nicefrac{d}{2}-4\eps_0, 1+\alpha\}$, 
	and we conclude that
	\begin{align}\label{eq:A-GP-final}  
		\text{($\text{A}_\GP$)} 
		&\lesssim_{(\eps,  
			\sigma,\alpha,\beta,A,\kappa,\cD)} 
		h^{\min\{2\beta - \sigma -\nicefrac{d}{2} 
			- \eps, \, 
			1+\alpha-\sigma,\, 2\alpha\}}, 
	\end{align} 
	for sufficiently small $h>0$ 
	and any $\eps > 0$ 
	(by adjusting $\eps_0,\eps',\eps''>0$). 
	
	If $d=3$, let $\eps_0>0$ be such that 
	$2\eps_0 < \min\{2\beta-\sigma-\nicefrac{3}{2}, \beta - 1\}$, 
	and choose 
	$\gamma 
	:= 
	\nicefrac{3}{4} - \nicefrac{\sigma}{2} + \eps_0 \in(0,\beta)$ 
	in~\eqref{eq:A-GP-gamma}. 
	We thus need to bound the terms 
	\begin{align*} 
		\text{($\text{A}'_\GP$)}
		&:= 
		\bigl\| 
		L^{\nicefrac{\sigma}{2}} 
		\bigl( 
		L^{-(\beta+\nicefrac{\sigma}{2}-\nicefrac{3}{4}-\eps_0 )} 
		- L_h^{-(\beta+\nicefrac{\sigma}{2}-\nicefrac{3}{4}-\eps_0 )}\Pi_h 
		\bigr) 
		L_h^{\nicefrac{\sigma}{2}-(\nicefrac{3}{4}+\eps_0 )} \Pi_h 
		\bigr\|_{\cL_2^{0;0}} , \\
		\text{($\text{A}''_\GP$)} 
		&:= 
		\bigl\| 
		L^{-(\beta - \nicefrac{3}{4} - \eps_0)} 
		\bigl(L^{-(\nicefrac{3}{4}-\nicefrac{\sigma}{2} + \eps_0)} 
		- L_h^{-(\nicefrac{3}{4}-\nicefrac{\sigma}{2} + \eps_0)} \Pi_h \bigr) 
		\bigr\|_{\cL_2^{0;0}}. 
	\end{align*}
	This can be achieved similarly 
	as for $d\in\{1,2\}$ 
	by picking the parameters 
	\begin{align*}
		\beta' &:= \beta+\nicefrac{\sigma}{2}-\nicefrac{3}{4}-\eps_0,
		& 
		\sigma' &:= \sigma, 
		& 
		\delta' &:= -\sigma, \\
		\beta'' &:= \nicefrac{3}{4}-\nicefrac{\sigma}{2} + \eps_0,
		& 
		\sigma'' &:= 0, 
		& 
		\delta'' &:= \min\{2\beta-3-4\eps_0, 1+\alpha\},  
	\end{align*}
	(recall that $\beta>1$ if $d=3$ 
	and, thus, $\delta''>-1$). 
	These choices result, 
	for sufficiently small $h>0$, 
	in the estimates 
	\begin{align*} 
		\text{($\text{A}'_\GP$)}
		&\lesssim_{(A,\kappa,\cD)} 
		\bigl\| 
		L^{-\beta'} 
		- L_h^{-\beta'}\Pi_h 
		\bigr\|_{\cL\left(\Hdot{-\sigma};\Hdot{\sigma}\right)} 
		\bigl\|
		L_h^{-(\nicefrac{3}{4}+\eps_0)} \Pi_h 
		\bigr\|_{\cL_2^{0;0}} \\
		&\lesssim_{(\eps_0,\eps',\sigma,\alpha,\beta,A,\kappa,\cD)}
		h^{\min\{ 
			2\beta - \sigma - \nicefrac{3}{2} - 2\eps_0 - \eps', \,
			1+\alpha-\sigma, \, 
			2\alpha
			\}}, 
		\\
		\text{($\text{A}''_\GP$)} 
		&:= 
		\bigl\| 
		\bigl(L^{-\beta''} 
		- L_h^{-\beta''} \Pi_h \bigr) 
		L^{-(\beta - \nicefrac{3}{2} - 2\eps_0)} 
		\bigr\|_{\cL(L_2(\cD))} 
		\bigl\| 
		L^{-(\nicefrac{3}{4} + \eps_0)} 
		\bigr\|_{\cL_2^{0;0}} \\
		&\lesssim_{(\eps_0,\eps'',\sigma,\alpha,\beta,A,\kappa,\cD)}
		h^{\min\{ 
			2\beta - \sigma - \nicefrac{3}{2} - 2\eps_0 - \eps'', \,
			1+\alpha+\delta'', \, 
			2\alpha
			\}}, 
	\end{align*}
	for all $\eps',\eps'' > 0$, 
	where we also have used~\eqref{eq:Lh+L} 
	and~\eqref{eq:sum-Pihm} 
	for ($\text{A}'_\GP$). Finally, 
	since
	$\alpha\geq\nicefrac{1}{2}-\sigma$ 
	if $d=3$, 
	we again have 
	$1+\alpha+\delta'' 
	\geq 
	\min\{2\beta-\sigma-\nicefrac{d}{2}-4\eps_0, 1+\alpha\}$.  
	Thus, \eqref{eq:A-GP-final} also holds for $d=3$. 
	
	To estimate ($\text{B}_\GP$),  
	we recall the convergence 
	result of the sinc quadrature 
	from~\cite[Lem.~3.4, Rem.~3.1 \& Thm.~3.5]{BonitoPasciak:2015}. 
	For sufficiently small $k>0$, 
	we have 
	\[
		\bigl\| E_Q^\beta \psi \bigr\|_{L_2(\cD)} 
		%= 
		%\bigl\| \bigl(L_h^{-\betafrac} - Q_{h,k}^{\betafrac} \bigr) 
		%L_h^{-\nbeta} 
		%\Pi_h e_j \bigr\|_{L_2(\cD)} 
		\lesssim_{(\beta,A,\kappa,\cD)}
		e^{-\nicefrac{\pi^2}{(2k)}} 
		\bigl\| L_h^{-\nbeta} \Pi_h \psi \bigr\|_{L_2(\cD)}  
		\quad 
		\forall\psi\in L_2(\cD). 
	\]
	Next, by 
	equivalence of the norms 
	$\norm{\,\cdot\,}{\sigma}$, 
	$\norm{\,\cdot\,}{H^\sigma(\cD)}$ 
	for $\sigma\in\{0,1\}$, 
	see Lemma~\ref{lem:hdot-sobolev}, 
	and by the inverse 
	inequality~\eqref{eq:ass:galerkin:inverse}
	from Assumption~\ref{ass:galerkin}.II, 
	we find, for $\sigma\in\{0,1\}$,  
	\begin{align}
		\text{($\text{B}_\GP$)} 
		&\lesssim_{q} 
		\bigl\|E_Q^\beta \bigr\|_{\cL_2^{0;\sigma}} 
		= 
		\bigl\| L^{\nicefrac{\sigma}{2}} \Pi_h E_Q^\beta 
		\bigr\|_{\cL_2^{0;0}} 
		\lesssim_{(\sigma,A,\kappa,\cD)} 
		h^{-\sigma}
		\bigl\| E_Q^\beta \bigr\|_{\cL_2^{0;0}} 
		\label{eq:EQ-estimate} \\ 
		&\lesssim_{(q,\sigma,\beta,A,\kappa,\cD)} 
		e^{-\nicefrac{\pi^2}{(2k)}} 
		h^{-\sigma} 
		\bigl\| L_h^{-\nbeta} \Pi_h \bigr\|_{\cL_2^{0;0}}  
		\lesssim_{(\beta,A,\kappa,\cD)} 
		e^{-\nicefrac{\pi^2}{(2k)}} 
		h^{-\sigma-\nicefrac{d}{2}\,\mathds{1}_{\{\beta<1\}}}, 
		\notag 
	\end{align}
	where we have applied~\eqref{eq:sum-Pihm} 
	with $m=\nbeta\in\bbN_0$, $m\neq\nicefrac{d}{4}$ 
	for $d\in\{1,2,3\}$  
	in the last step. If $\sigma\in(0,1)$, 
	a respective bound  
	for ($\text{B}_\GP$) follows by interpolation. 
	
	We proceed with the derivation 
	of~\eqref{eq:galerkin:conv:cov-alpha} 
	by estimating ($\text{A}_\varrho$) 
	and ($\text{B}_\varrho$) 
	in~\eqref{eq:A-rho-B-rho}. 
	By~\eqref{eq:cov-fct-err} of 
	Proposition~\ref{prop:regularity:hdot} 
	we obtain  
	\begin{align*} 
		\text{($\text{A}_\varrho$)} 
		&= 
		\bigl\| L^{-2\beta} 
			- L_h^{-2\beta}\Pi_h \bigr\|_{\cL_2^{-\sigma; \sigma}}, 
		\quad 
		\text{($\text{B}_\varrho$)} 
		= 
		\bigl\| L_h^{-2\beta}\Pi_h  
			- Q_{h,k}^{\betafrac} L_h^{-2\nbeta} 
			Q_{h,k}^{\betafrac}
			\Pi_h
		\bigr\|_{\cL_2^{-\sigma; \sigma}}. 
	\end{align*} 
	To bound ($\text{A}_\varrho$), 
	we let $\eps_0>0$ be such that 
	$2\beta-\sigma-\nicefrac{d}{2}>2\eps_0$ and 
	write 
	\begin{align*}
		L^{-2\beta} - L_h^{-2\beta} \Pi_h 
		&=
		\bigl( L^{-(2\beta-\nicefrac{\sigma}{2}-\nicefrac{d}{4}-\eps_0 )} 
		- L_h^{-(2\beta-\nicefrac{\sigma}{2}-\nicefrac{d}{4}-\eps_0 )}\Pi_h 
		\bigr) 
		L_h^{-(\nicefrac{\sigma}{2}+\nicefrac{d}{4}+\eps_0 )} \Pi_h \\
		&\quad + 
		L^{-(2\beta-\nicefrac{\sigma}{2}-\nicefrac{d}{4}-\eps_0)} 
		\bigl(L^{-(\nicefrac{\sigma}{2} + \nicefrac{d}{4}+\eps_0)} 
		- L_h^{-(\nicefrac{\sigma}{2} + \nicefrac{d}{4}+\eps_0)} 
		\Pi_h \bigr), 	
	\end{align*} 
	and find therefore that 
	$\text{($\text{A}_\varrho$)}  
	\leq 
	\text{($\text{A}'_\varrho$)}  
	+ 
	\text{($\text{A}''_\varrho$)}$, 
	where 
	\begin{align*} 
		\text{($\text{A}'_\varrho$)} 
		&:= 
	 	\bigl\| 
	 	\bigl( L^{-(2\beta-\nicefrac{\sigma}{2}-\nicefrac{d}{4}-\eps_0 )} 
	 	- L_h^{-(2\beta-\nicefrac{\sigma}{2}-\nicefrac{d}{4}-\eps_0 )}\Pi_h 
	 	\bigr) 
	 	L_h^{-(\nicefrac{\sigma}{2}+\nicefrac{d}{4}+\eps_0 )} \Pi_h 
	 	\bigr\|_{\cL_2^{-\sigma; \sigma}}, \\
	 	\text{($\text{A}''_\varrho$)} 
	 	&:= 
	 	\bigl\| 
	 	L^{-(2\beta - \nicefrac{\sigma}{2} - \nicefrac{d}{4} - \eps_0)} 
	 	\bigl(L^{-(\nicefrac{\sigma}{2} + \nicefrac{d}{4}+\eps_0)} 
	 	- L_h^{-(\nicefrac{\sigma}{2} + \nicefrac{d}{4}+\eps_0)} 
	 	\Pi_h \bigr)
	 	\bigr\|_{\cL_2^{-\sigma; \sigma}}. 
	\end{align*}
	For term ($\text{A}'_\varrho$), 
	we apply~\eqref{eq:fem-error-fractional}
	of Theorem~\ref{thm:fem-error}, 
	for 
	$\beta' := 2\beta-\nicefrac{\sigma}{2}-\nicefrac{d}{4}-\eps_0$, 
	$\sigma' := \sigma$, and 
	$\delta' := 0$. We thus obtain 
	that, for any $\eps' > 0$ 
	and sufficiently small $h>0$, 
	\begin{align*} 
		\text{($\text{A}'_\varrho$)} 
		&\leq 
		\bigl\| 
		L^{-\beta'} 
		- L_h^{-\beta'}\Pi_h 
		\bigr\|_{\cL\left(\Hdot{0}; \Hdot{\sigma}\right)} 
		\bigl\| 
		L_h^{-(\nicefrac{\sigma}{2}+\nicefrac{d}{4}+\eps_0 )} \Pi_h 
		L^{\nicefrac{\sigma}{2}} 
		\bigr\|_{\cL_2^{0;0}} \\
		&\lesssim_{(\eps_0,\eps',\sigma,\alpha,\beta,A,\kappa,\cD)} 
		h^{\min\{ 
			4\beta-2\sigma-\nicefrac{d}{2} - 2\eps_0 - \eps', \, 
			1+\alpha-\sigma, \, 
			2\alpha 
			\}} 
		\bigl\| 
		L_h^{-(\nicefrac{\sigma}{2}+\nicefrac{d}{4}+\eps_0 )} \Pi_h 
		\bigr\|_{\cL_2^{0;\sigma}}. 
	\end{align*} 
	Here, the arising 
	Hilbert--Schmidt norm is 
	bounded by a constant, since 
	\begin{align*}
		\bigl\| 
		L_h^{-(\nicefrac{\sigma}{2}+\nicefrac{d}{4}+\eps_0 )} \Pi_h 
		\bigr\|_{\cL_2^{0;\sigma}}
		\leq 
		\bigl\| 
		L^{\nicefrac{\sigma}{2}} 
		L_h^{-\nicefrac{\sigma}{2}} \Pi_h 
		\bigr\|_{\cL(L_2(\cD))}
		\bigl\|
		L_h^{-(\nicefrac{d}{4}+\eps_0 )} \Pi_h 
		\bigr\|_{\cL_2(L_2(\cD))},  
	\end{align*} 
	and boundedness 
	follows from 
	\eqref{eq:L+Lh} and~\eqref{eq:sum-Pihm}. 
	
	For term ($\text{A}''_\varrho$), 
	we choose the parameters 
	in~\eqref{eq:fem-error-fractional}
	of Theorem~\ref{thm:fem-error}, 
	as follows: 
	$\beta'' := \nicefrac{\sigma}{2}+\nicefrac{d}{4}+\eps_0$, 
	$\sigma'' := \sigma$, and 
	$\delta'' := \min\{4\beta-2\sigma-d-4\eps_0, 1+\alpha\} > 0$. 
	This gives, for any $\eps'' > 0$ 
	and sufficiently small $h>0$, 
	\begin{align*} 
		\text{($\text{A}''_\varrho$)} 
		&:= 
		\bigl\| 
		L^{\nicefrac{\sigma}{2}} 
		\bigl(L^{-(\nicefrac{\sigma}{2} + \nicefrac{d}{4}+\eps_0)} 
		- L_h^{-(\nicefrac{\sigma}{2} + \nicefrac{d}{4}+\eps_0)} 
		\Pi_h \bigr) 
		L^{-(2\beta - \sigma - \nicefrac{d}{4} - \eps_0) } 
		\bigr\|_{\cL_2^{0;0}} \\
		&\leq
		\bigl \| 
		L^{-\beta''} 
		- L_h^{-\beta''} 
		\Pi_h 
		\bigr\|_{\cL\left(\Hdot{\delta''}; \Hdot{\sigma}\right)}
		\bigl\| 
		L^{-\nicefrac{d}{4} - \eps_0}
		\bigr\|_{\cL_2^{0;0}} \\
		&\lesssim_{(\eps_0,\eps'',\sigma,\alpha,\beta,A,\kappa,\cD)} 
		h^{\min\{ 
			4\beta-2\sigma-\nicefrac{d}{2} - 2\eps_0 - \eps'', \, 
			1+\alpha-\sigma, \, 
			2\alpha 
			\}} , 
	\end{align*}
	since $\bigl\| 
	L^{-\nicefrac{d}{4} - \eps_0}
	\bigr\|_{\cL_2^{0;0}}$ is bounded 
	due to the spectral asymptotics~\eqref{eq:lem:spectral-behav} 
	of Lemma~\ref{lem:spectral-behav}. 
	We conclude that 
	\[
		\text{($\text{A}_\varrho$)}  
		\lesssim_{(\eps,\sigma,\alpha,\beta,A,\kappa,\cD))}
		h^{\min\left\{ 
			4\beta-2\sigma-\nicefrac{d}{2} - \eps, \, 
			1+\alpha-\sigma, \, 
			2\alpha 
			\right\}}, 
	\]
	for every $\eps>0$ and 
	sufficiently small $h>0$. 
	
	Finally, we use the estimate 
	\begin{align}
		\bigl\| T\dual{T} - \widetilde{T}\dual{\widetilde{T}} 
		\bigr\|_{\cL_2^{-\sigma;\sigma}} 
		&= 
		\bigl\| \tfrac{1}{2} \bigl(T + \widetilde{T} \bigr) 
		\dual{\bigl( T - \widetilde{T} \bigr)} 
		+ \tfrac{1}{2}
		\bigl(T - \widetilde{T}\bigr) 
		\dual{\bigl(T + \widetilde{T} \bigr)} \bigr\|_{\cL_2^{-\sigma;\sigma}} 
		\notag \\
		&\leq \bigl\| \bigl( T + \widetilde{T} \bigr) 
		\dual{\bigl(T - \widetilde{T} \bigr)} \bigr\|_{\cL_2^{-\sigma;\sigma}},  
		\label{eq:HS-estimate}
	\end{align} 
	as well as the inverse inequality~\eqref{eq:ass:galerkin:inverse}
	to conclude for term ($\text{B}_\varrho$)
	for $\sigma\in\{0,1\}$ 
	that 
	\begin{align*} 
		\text{($\text{B}_\varrho$)}
		&\lesssim_{(\sigma,A,\kappa,\cD)} 
		h^{-\sigma}
		\bigl\| 
		\bigl(L_h^{-\beta} + Q_{h,k}^{\betafrac} L_h^{-\nbeta}\bigr) \Pi_h 
		\bigr\|_{\cL(L_2(\cD))} 
		\bigl\| 
		\dual{\bigl( E_Q^\beta \bigr)} 
		\bigr\|_{\cL_2^{-\sigma;0}} \\
		&\lesssim_{(\sigma,A,\kappa,\cD)} 
		h^{-\sigma}
		\left( \bigl\| 
		L_h^{-\beta} \Pi_h 
		\bigr\|_{\cL(L_2(\cD))}
		+ 
		\bigl\| 
		Q_{h,k}^{\betafrac} L_h^{-\nbeta} \Pi_h 
		\bigr\|_{\cL(L_2(\cD))} \right) 
		\bigl\| 
		E^\beta_Q \bigr\|_{\cL_2^{0;\sigma}}. 
	\end{align*} 
	Combining the above estimate with \eqref{eq:EQ-estimate} 
	and stability of the operators 
	\begin{equation}\label{eq:uniform-stab} 
		L_h^{-\beta}, \, 
		Q_{h,k}^{\betafrac} 
		\from 
		\left(V_h, \norm{\,\cdot\,}{L_2(\cD)}\right) 
		\to 
		\left(V_h, \norm{\,\cdot\,}{L_2(\cD)}\right)
	\end{equation}
	which is uniform in $h$ and $k$ 
	for sufficiently small $h,k>0$, 
	shows that 
	\[
		\text{($\text{B}_\varrho$)}
		\lesssim_{(\sigma,\beta,A,\kappa,\cD)}
		e^{-\nicefrac{\pi^2}{(2k)}} 
		h^{-2\sigma-\nicefrac{d}{2}\,\mathds{1}_{\{\beta<1\}}}.  
	\]
	Interpolation for $\sigma\in(0,1)$ 
	completes the proof 
	of~\eqref{eq:galerkin:conv:cov-alpha}. 
\end{proof}

Due to the similarity 
in the derivation with 
the proof 
of~\cite[Thm.~2.10]{bkk-strong}, 
we have moved the proof of the 
following proposition to 
Appendix~\ref{app:proofs}. 

\begin{proposition}\label{prop:galerkin:conv:sobolev-p}
	Suppose  
	Assumptions~\ref{ass:coeff}.I--II,~\ref{ass:dom}.I, 
	and~\ref{ass:galerkin}.II--III. 
	Let Assumption~\ref{ass:galerkin}.IV  
	be satisfied with parameters 
	$r,s_0,t>0$ such that 
	$\nicefrac{r}{2}\geq t-1$ 
	and $s_0\geq t$. 
	Let $d\in\bbN$, $\beta > 0$ and   
	$0 \leq \sigma\leq 1$  
	be such that 
	$2\beta-\sigma > \nicefrac{d}{2}$. 
	For $\tau\geq 0$, set   
	\begin{align} 
		\hspace*{-0.9cm} 
		\rho_0(\tau) 
		:= 
		\min\left\{ r, \,  
		s_0, \, 
		2\beta + \tau -\nicefrac{d}{2}   
		\right\}, 
		\    
		\rho_1(\tau) 
		:= 
		\min\left\{ \nicefrac{r}{2}, \,  
		s_0, \, 
		2\beta - 1 + \tau -\nicefrac{d}{2}  
		\right\}.
		\hspace*{-0.9cm} 
		\label{eq:def:rho-01} 
	\end{align}
	Furthermore, define, for $0\leq\sigma\leq 1$,  
	\begin{align} 
		\rho_\GP(\sigma) 
		&:= 
		(1-\sigma) \rho_0(0) 
		+ 
		\sigma \rho_1(0), 
		\quad 
		%\label{eq:def:rho-GP-sigma} \\
		\rho_\varrho(\sigma) 
		:= 
		(1-\sigma) \rho_0(2\beta) 
		+ 
		\sigma \rho_1(2\beta-1). 
		\label{eq:def:rho-sigma}
	\end{align}
	Let~$\GP^\beta$ be the Whittle--Mat\'ern 
	field in~\eqref{eq:gpbeta}
	and, for $h, k>0$, 
	let $\widetilde{\GP}^\beta_{h,k}$ denote the 
	sinc-Galerkin approximation 
	in~\eqref{eq:def:gphk-2}, 
	with covariance functions 
	$\varrho^\beta$ and
	$\widetilde{\varrho}^\beta_{h,k}$, 
	respectively.  
	%cf.~\eqref{eq:def:covfct}.
	Then, for all $q>0$, 
	\begin{align}
		\hspace*{-0.2cm}
		\Bigl(\bbE\Bigl[
		\bigl\| \GP^{\beta} 
		- \widetilde{\GP}^\beta_{h,k} 
		\bigr\|_{H^\sigma(\cD)}^q
		\Bigr]\Bigr)^{\nicefrac{1}{q}} 
		&\lesssim_{(q,\cP)}
		C_{\beta,h}^{\GP}
		\Bigl(h^{ \rho_\GP(\sigma) }  
		+ 
		e^{-\nicefrac{\pi^2}{(2k)}} 
		h^{-\sigma - \nicefrac{d}{2}\,\mathds{1}_{\{\beta<1\}}} \Bigr),
		\hspace*{-0.2cm} 
		\label{eq:galerkin:conv:hdot} 
		\\
		\bigl\| \varrho^\beta 
		- \widetilde{\varrho}^\beta_{h,k} 
		\bigr\|_{H^{\sigma,\sigma}(\cD\times\cD)} 
		&\lesssim_{\cP}    
		C_{\beta,h}^{\varrho}
		\Bigl( 
		h^{ \rho_\varrho(\sigma) } 
		+ 
		e^{-\nicefrac{\pi^2}{(2k)}} 
		h^{-2\sigma - \nicefrac{d}{2}\,\mathds{1}_{\{\beta<1\}}} \Bigr)  
		\label{eq:galerkin:conv:cov} 
	\end{align} 
	hold for 
	sufficiently small $h,k>0$, where  
	\begin{align*}
		C_{\beta,h}^{\GP} 
		&:= 
		\begin{cases}
			\sqrt{\ln(1/h)} 
			& 
			\text{if } 
			2\beta \in 
			\left\{2(t-1)+\gamma+\nicefrac{d}{2}, \, 
				t+\gamma+\nicefrac{d}{2} 
				: \gamma\in\{0,1\} \right\}, \\
			1 
			& 
			\text{otherwise},
		\end{cases} \\
		C_{\beta,h}^{\varrho} 
		&:= 
		\begin{cases}
			\sqrt{\ln(1/h)} 
			& 
			\text{if } 
			4\beta \in 
			\left\{2(t-1)+\gamma+\nicefrac{d}{2}, \, 
				t+\gamma+\nicefrac{d}{2} 
				: \gamma\in\{0,1,2\} \right\}, \\
			1  
			& 
			\text{otherwise}, 
		\end{cases}
	\end{align*} 
	and $\cP:=\{C_0,C_\lambda,\sigma,\beta,A,\kappa,\cD\}$. 
\end{proposition}

\begin{proposition}\label{prop:galerkin:conv:hoelder}   
	Suppose Assumptions~\ref{ass:coeff}.I--II, 
	\ref{ass:galerkin}.II--III, 
	and let Assumption~\ref{ass:galerkin}.I  
	be satisfied with parameters   
	$\theta_0\in(0,1)$ and $\theta_1 \geq 1+\alpha$, 
	where $0 < \alpha \leq 1$ 
	is as in~\eqref{eq:ass:1+alpha}. 
	Assume furthermore that 
	$\Pi_h$ is $H^1(\cD)$-stable, see~\eqref{eq:PihH1stable},  
	and that $d=1$, 
	$\beta>0$ and 
	$0<\gamma\leq \nicefrac{1}{2}$ 
	are such that $2\beta > \gamma + \nicefrac{1}{2}$. 
	Then, the 
	Whittle--Mat\'ern field~$\GP^\beta$
	%defined $(\bbP\text{-a.s.})$ 
	in~\eqref{eq:gpbeta}
	and the sinc-Galerkin 
	approximation $\GP_{h,k}^\beta$ 
	in~\eqref{eq:def:gphk-1} 
	can be taken  
	as continuous random fields. 
	Moreover, 
	for every $\delta\in(0,\gamma)$, 
	all $\eps,q>0$ and     
	sufficiently small $h>0$, we have  
	\begin{align}
		&\Bigl( \bbE 
		\Bigl[ 
		\bigl\| \GP^\beta - \GP^{\beta}_{h,k} \bigr\|_{C^{\delta}(\clos{\cD})}^q 
		\Bigr] 
		\Bigr)^{\nicefrac{1}{q}} \notag \\ 
		&\hspace{1.25cm} 
		\lesssim_{(q,\gamma,\delta,\eps,\alpha,\beta,A,\kappa,\cD)} 
		h^{\min\{ 2\beta - \gamma - \nicefrac{1}{2} -\varepsilon, \, 
			\nicefrac{1}{2} + \alpha - \gamma, \, 
			2\alpha\} }  
		+ 
		e^{-\nicefrac{\pi^2}{(2k)}} h^{-\gamma - \nicefrac{1}{2} }, 
		\label{eq:galerkin:conv:hoelder} 
		\\
		&\sup_{x,y\in\clos{\cD}} 
		\bigl|  
		\varrho^\beta(x,y) - \varrho^\beta_{h,k}(x,y) 
		\bigr| \notag \\
		&\hspace{1.25cm}   
		\lesssim_{(\eps,\alpha,\beta,A,\kappa,\cD)}
		h^{\min\{ 4\beta - 1 - \varepsilon, \, 
			\nicefrac{1}{2} + \alpha - \varepsilon, \, 
			2\alpha\} } 
		+ 
		e^{-\nicefrac{\pi^2}{(2k)}} h^{-1-\eps}.  
		\label{eq:galerkin:conv:cov-Linf}
	\end{align} 
	Here, $\varrho^\beta, \varrho^\beta_{h,k}$ 
	denote the covariance functions 
	of $\GP^\beta$ and $\GP^\beta_{h,k}$, respectively.  
\end{proposition}

\begin{proof} 
	Clearly, $Q^{\betafrac}_{h,k} L_h^{-\nbeta} \Pi_h 
	\in 
	\cL\bigl( L_2(\cD); H^{\gamma + \nicefrac{1}{2}}(\cD) \bigr)$, 
	since $Q^{\betafrac}_{h,k} L_h^{-\nbeta} \Pi_h$ 
	is a finite-rank operator 
	and $V_h \subset H^1_0(\cD) 
	\subset H^{\gamma + \nicefrac{1}{2}}(\cD)$ 
	by assumption. 
	Thus, by Corollary~\ref{cor:regularity:hoelder} 
	$\GP^{\beta}_{h,k}$ can be taken as a continuous GRF; 
	and the same is true for the Whittle--Mat\'ern 
	field~$\GP^\beta$ by Corollary~\ref{lem:matern:hoelder}. 
	Then, $\GP^\beta - \GP^{\beta}_{h,k}$ is 
	a continuous random field, 
	colored by $E^\beta_{h,k} = E^\beta_{V_h} + E_Q^\beta$, 
	see~\eqref{eq:def:err-galerkin}--\eqref{eq:def:err-quad}. 
	Furthermore, by~\eqref{eq:cor:regularity:hoelder}
	and by Lemma~\ref{lem:hdot-sobolev}, 
	since $d=1$ and 
	$\nicefrac{1}{2}<\gamma+\nicefrac{1}{2}\leq 1$,
	we have, for $\delta\in(0,\gamma)$ 
	and $q\in(0,\infty)$, 
	\begin{align*}
		\left( \bbE 
		\Bigl[ 
		\bigl\| \GP^\beta - \GP^{\beta}_{h,k} \bigr\|_{C^{\delta}(\clos{\cD})}^q 
		\Bigr] 
		\right)^{\nicefrac{1}{q}} 
		\lesssim_{(q,\gamma,\delta,A,\kappa,\cD)} 
		\bigl\| E^\beta_{V_h} + E^\beta_{Q} \bigr\|_{
			\cL\left(\Hdot{0}; \Hdot{\gamma+ \nicefrac{1}{2}} \right)}.   
	\end{align*} 
	By~\eqref{eq:fem-error-fractional} 
	of Theorem~\ref{thm:fem-error} 
	we then find, for any $\eps>0$ and 
	sufficiently small $h>0$,  
	\begin{align*} 
		\bigl\| E^\beta_{V_h} \bigr\|_{
			\cL\left(\Hdot{0}; \Hdot{\gamma+\nicefrac{1}{2}} \right)} 
		\lesssim_{(\gamma,\eps,\alpha,\beta,A,\kappa,\cD)} 
		h^{\min\{ 2\beta - \gamma - \nicefrac{1}{2} -\varepsilon, \, 
				\nicefrac{1}{2}+\alpha-\gamma, \, 2\alpha \} }. 
	\end{align*}  
	For $E_Q^\beta\in\cL(V_h)$ 
	we use the inverse inequality~\eqref{eq:ass:galerkin:inverse}
	as well as 
	the quadrature error estimate from 
	\cite[Lem.~3.4, Rem.~3.1 \& Thm.~3.5]{BonitoPasciak:2015}
	and obtain 
	\begin{align}\label{eq:proof-EQ-est} 
		\bigl\| E_Q^\beta \bigr\|_{ 
			\cL\left(\Hdot{0};\Hdot{\gamma+\nicefrac{1}{2}}\right)} 
%		&\lesssim_{(\gamma,A,\kappa,\cD)}
%		h^{-\gamma - \nicefrac{1}{2}} 
%		\bigl\| E_Q^\beta \bigr\|_{\cL\left(\Hdot{0}\right)} 
		\lesssim_{(\gamma,\beta,A,\kappa,\cD)}
		e^{-\nicefrac{\pi^2}{(2k)}} h^{-\gamma - \nicefrac{1}{2}}, 
	\end{align} 
	for sufficiently small $h>0$, 
	which completes the proof 
	of~\eqref{eq:galerkin:conv:hoelder}.  
	
	For the $L_{\infty}(\cD\times\cD)$-estimate    
	\eqref{eq:galerkin:conv:cov-Linf} of the covariance function, 
	fix $\eps\in(0,2)$. 
	First, we recall the Sobolev 
	embedding $H^{\nicefrac{\eps}{4}+\nicefrac{1}{2}}(\cD)
	\hookrightarrow 
	C^{\nicefrac{\eps}{4}}(\clos{\cD})$ 
	as well as the equivalence of 
	the spaces 
	$H^{\nicefrac{\eps}{4}+\nicefrac{1}{2}}(\cD) 
	\cong_{(A,\kappa,\cD)} \Hdot{\nicefrac{\eps}{4}+\nicefrac{1}{2}}$,  
	see Lemma~\ref{lem:hdot-sobolev}. 
	We then conclude with~\eqref{eq:prop:regularity:Linf-ii}  
	of Proposition~\ref{prop:regularity:Linf}\ref{prop:regularity:Linf-ii}
	that, for 
	$\sigma:=\nicefrac{1}{2}+\nicefrac{\eps}{4}\in(\nicefrac{1}{2},1)$,  
	\begin{align*} 
		&\sup_{x,y\in\clos{\cD}} 
		\bigl|  
		\varrho^\beta(x,y) - \varrho^\beta_{h,k}(x,y) 
		\bigr| 
		\leq 
		\bigl\| 
		L^{-2\beta} 
		- 
		Q_{h,k}^{\betafrac} L_h^{-\nbeta} \Pi_h 
		\dual{ \bigl(Q_{h,k}^{\betafrac} L_h^{-\nbeta} \Pi_h \bigr)} 
		\bigr\|_{\cL(\dual{C(\clos{\cD})};C(\clos{\cD}))} \\
		&\qquad\lesssim_{(\eps,A,\kappa,\cD)}  
		\bigl\| 
		\bigl( 
		L^{-2\beta} 
		- 
		L_h^{-2\beta}\Pi_h \bigr) 
		+ 
		\bigl(L_h^{-2\beta} 
		-
		Q_{h,k}^{\betafrac} L_h^{-2\nbeta} Q_{h,k}^{\betafrac} 
		\bigr) \Pi_h 
		\bigr\|_{\cL\left(\Hdot{-\sigma};\Hdot{\sigma}\right)}. 
	\end{align*} 
	By~\eqref{eq:fem-error-fractional}
	of Theorem~\ref{thm:fem-error} 
	we have 
	\begin{align*}
		\bigl\| 
		L^{-2\beta} 
		- 
		L_h^{-2\beta}\Pi_h 
		\bigr\|_{\cL\left(\Hdot{-\sigma};\Hdot{\sigma}\right)} 
		\lesssim_{(\eps,\alpha,\beta,A,\kappa,\cD)} 
		h^{\min\{
			4\beta-1-\eps,\,
			\nicefrac{1}{2}+\alpha-\nicefrac{\eps}{4}, \,
			2\alpha\}}. 
	\end{align*} 
	Furthermore, we find, similarly as in~\eqref{eq:HS-estimate}, 
	that 
	\begin{align*} 
		\bigl\| 
		\bigl( L_h^{-2\beta}  
		&-
		Q_{h,k}^{\betafrac} L_h^{-2\nbeta} Q_{h,k}^{\betafrac} 
		\bigr) \Pi_h 
		\bigr\|_{\cL\left(\Hdot{-\sigma};\Hdot{\sigma}\right)} \\
		&\leq 
		\bigl\| 
		(L_h^{-\beta}  
		+
		Q_{h,k}^{\betafrac} L_h^{-\nbeta} ) 
		\dual{(L_h^{-\beta} - 
		Q_{h,k}^{\betafrac} L_h^{-\nbeta} )} \Pi_h  
		\bigr\|_{\cL\left(\Hdot{-\sigma};\Hdot{\sigma}\right)} \\
		&\lesssim_{(\eps,A,\kappa,\cD)}
		h^{-\nicefrac{1}{2}-\nicefrac{\eps}{4}} 
		\bigl\| 
		\bigl(L_h^{-\beta} 
		+ 
		Q_{h,k}^{\betafrac} L_h^{-\nbeta}\bigr) \Pi_h 
		\bigr\|_{\cL(L_2(\cD))}  
		\bigl\| 
		\dual{\bigl(E^\beta_Q \bigr)}   
		\bigr\|_{\cL\left(\Hdot{-\sigma}; \Hdot{0}\right)}, 
	\end{align*} 
	where 
	we have used the inverse 
	inequality~\eqref{eq:ass:galerkin:inverse} 
	in the last step. 
	Finally, since  
	$\bigl\| 
	\dual{\bigl(E^\beta_Q \bigr)}   
	\bigr\|_{\cL\left(\Hdot{-\sigma}; \Hdot{0}\right)}
	= 
	\bigl\| 
	E^\beta_Q    
	\bigr\|_{\cL\left(\Hdot{0};
			\Hdot{\sigma}\right)}$, 
	the proof is completed by 
	\eqref{eq:proof-EQ-est}
	combined with the uniform 
	stability~\eqref{eq:uniform-stab} of 
	$L_h^{-\beta}$ and $Q_{h,k}^{\betafrac}$. 
\end{proof}

%============================================================
\subsection{Application to finite element approximations}
\label{subsec:matern:fem}
%============================================================

We now   
discuss different scenarios 
of 
\begin{enumerate} 
	\item regularity 
		of the second-order differential $L$ 
		in~\eqref{eq:def:L}, 
	\item  finite element (FE) discretizations
		satisfying Assumptions~\ref{ass:galerkin}.I--IV  
		for specific  
		values of $0 < \theta_0 < \theta_1$ 
		and of $r,s_0,t>0$. 
\end{enumerate}  
We then obtain 
explicit rates of convergence 
for the FE Galerkin approximations 
$\GP_{h,k}^\beta$, $\widetilde{\GP}_{h,k}^\beta$ 
in~\eqref{eq:def:gphk-1}--\eqref{eq:def:gphk-2}
from Propositions~\ref{prop:galerkin:conv:sobolev-alpha}, 
\ref{prop:galerkin:conv:sobolev-p} and \ref{prop:galerkin:conv:hoelder}. 

\begin{assumption}[FE discretization] 
\label{ass:fem} 	
	Throughout this subsection, 
	we suppose the following setting:  
	\begin{enumerate}[leftmargin=2em,label=$\bullet$]
	\item 
		the (minimal)
		Assumptions~\ref{ass:coeff}.I--II 
		on the coefficients $A$, $\kappa$  
		of the operator $L$; 
	\item 
		Assumptions~\ref{ass:dom}.I, i.e., 
		$\cD\subset\bbR^d$ 
		is a bounded Lipschitz domain; 
	\item   
		$(\cT_h)_{h>0}$ 
		is a quasi-uniform family of 
		triangulations on $\clos{\cD}$, 
		indexed by the mesh width $h>0$; 
	\item  
		the basis 
		functions of the 
		finite-dimensional space 
		$V_h\subset H^1_0(\cD)$ 
		are continuous on~$\clos{\cD}$ 
		and piecewise polynomial 
		with respect to $\cT_h$   
		of degree 
		at most $p\in\bbN$.
	\end{enumerate} 
\end{assumption} 
All further assumptions 
on the operator $L$, 
on the domain $\cD$, and 
on the FE spaces 
are explicitly specified for each case. 
Note that quasi-uniformity of $(\cT_h)_{h>0}$ 
already guarantees that 
Assumptions~\ref{ass:galerkin}.II and~\ref{ass:galerkin}.III 
are satisfied (\ref{ass:galerkin}.III is obvious, 
for the inverse inequality~\ref{ass:galerkin}.II 
see, e.g.,~\cite[Cor.~1.141]{ErnGuermond:2004}). 

In Subsection~\ref{subsubsec:fem:smooth} 
we briefly comment on the situation  
of smooth coefficients 
and apply Proposition~\ref{prop:galerkin:conv:sobolev-p} 
to derive optimal convergence 
rates when $p\geq 1$.  
Afterwards,  
in Subsection~\ref{subsubsec:fem:non-smooth}  
we focus on 
less regular problems and $p=1$ 
by using the results from 
Propositions~\ref{prop:galerkin:conv:sobolev-alpha}  
and~\ref{prop:galerkin:conv:hoelder}.

%================================================================
\subsubsection{The smooth case}\label{subsubsec:fem:smooth}
%================================================================

The remaining crucial ingredient 
in order to derive explicit rates 
of convergence from 
Proposition~\ref{prop:galerkin:conv:sobolev-p}  
is to prove validity of 
Assumption~\ref{ass:galerkin}.IV
for the finite element spaces $(V_h)_{h>0}$. 
For the case of a second-order elliptic 
differential operator $L$ 
with smooth coefficients, 
these results are 
well-known and 
we summarize them below.  

\begin{assumption}[smooth case]\label{ass:fem:smooth}
	The domain $\cD$ has a smooth 
	$C^\infty$-boundary $\partial\cD$,  
	and 
	the coefficients of $L$ in~\eqref{eq:def:L}
	are smooth, i.e., 
	$A\in C^\infty(\clos{\cD})^{d\times d}$  
	and 
	$\kappa\in C^\infty(\clos{\cD})$. 
	Furthermore, the Rayleigh--Ritz projection
	$R_h \from H^1_0(\cD) \to V_h$ 
	in~\eqref{eq:def:ritz}
	satisfies the a-priori estimates 
	\begin{align*} 
		\norm{v - R_h v}{H^1(\cD)} 
		&\lesssim_{(p,A,\kappa,\cD)} 
		h^{p} 
		\norm{v}{H^{p+1}(\cD)}, 
		\\
		\norm{v - R_h v}{L_2(\cD)} 
		&\lesssim_{(p,A,\kappa,\cD)} 
		h^{p+1} 
		\norm{v}{H^{p+1}(\cD)}. 
	\end{align*}
\end{assumption}

\begin{lemma}\label{lem:fem:smooth:rst} 
	Suppose Assumptions~\ref{ass:fem} and~\ref{ass:fem:smooth}. 
	Then, 
	Assumption~\ref{ass:galerkin}.IV 
	is satisfied for 
	$r = 2p$ and  
	$s_0 = t = p+1$. 
\end{lemma}

\begin{proof} 
	See, e.g., \cite[Thm.~6.1 \& Thm.~6.2]{StrangFix:2008}. 
\end{proof} 

\begin{theorem}\label{thm:fem:smooth:sobolev} 
	Suppose Assumptions~\ref{ass:fem} 
	and~\ref{ass:fem:smooth}.  
	Let 
	$d\in\bbN$, $\beta > 0$, 
	and $0 \leq \sigma\leq 1$
	be such that 
	$2\beta-\sigma > \nicefrac{d}{2}$,  
	let~$\GP^\beta$ be the Whittle--Mat\'ern 
	field in~\eqref{eq:gpbeta}
	and, for $h, k>0$, 
	let $\widetilde{\GP}^\beta_{h,k}$ be the 
	sinc-Galerkin approximation 
	in~\eqref{eq:def:gphk-2}, 
	and let $\varrho^\beta$, $\widetilde{\varrho}^\beta_{h,k}$ 
	denote their covariance functions. 
	Then we have, 
	for sufficiently small $h>0$, 
	sufficiently small 	
	$k=k(h)>0$, 
	and all $q>0$,  
	\begin{align*}
		\left(\bbE\left[
		\bigl\| \GP^{\beta} - \widetilde{\GP}^\beta_{h,k} 
		\bigr\|_{H^\sigma(\cD)}^q
		\right]\right)^{\nicefrac{1}{q}} 
		&\lesssim_{(q,\cP)} 
		C_{\beta,h}^{\GP} \, 
		h^{ \min \left\{ 
			2\beta - \sigma - \nicefrac{d}{2}, \, 
			p + 1 - \sigma \right\}}, 
		\\
		\bigl\| \varrho^\beta 
		- \widetilde{\varrho}^\beta_{h,k} 
		\bigr\|_{H^{\sigma,\sigma}(\cD\times\cD)}
		&\lesssim_{\cP} 
		C_{\beta,h}^{\varrho} \,
		h^{ (1-\sigma )
			\min\left\{ 4\beta - \nicefrac{d}{2}, \, p+1 \right\}  
			+\sigma 
			\min\left\{ 4\beta - 2 - \nicefrac{d}{2}, \, p\right\} },   
	\end{align*}
	where $C_{\beta,h}^{\varrho}, C_{\beta,h}^{\GP}$ and $\cP$ 
	are as in Proposition~\ref{prop:galerkin:conv:sobolev-p}. 
\end{theorem} 

\begin{proof} 
	By Lemma~\ref{lem:fem:smooth:rst} 
	we have $r=2p$, $s_0 = t = p+1$  
	and, thus, for $\gamma\in\{0,1\}$, 
	$\rho_\gamma(\tau) 
	= \min\left\{p+1-\gamma, \,  
	2\beta + \tau - \gamma - \nicefrac{d}{2}\right\}$ 
	in \eqref{eq:def:rho-01}. 
	Finally, 
	\begin{align*}
		\rho_\GP(\sigma) 
		&= 
		\min\left\{ 2\beta - \nicefrac{d}{2}, \, p+1\right\} 
		- 
		\sigma, \\
		\rho_\varrho(\sigma) 
		&= 
		(1-\sigma )
		\min\left\{ 4\beta - \nicefrac{d}{2}, \, p+1 \right\}  
		+\sigma 
		\min\left\{ 4\beta - 2 - \nicefrac{d}{2},\, p \right\} 
	\end{align*} 
	in~\eqref{eq:def:rho-sigma}, 
	for any $0\leq\sigma\leq 1$, 
	and the assertion holds 
	by Proposition~\ref{prop:galerkin:conv:sobolev-p}. 
\end{proof} 

\begin{remark}
	The convergence rates  
	with respect to the $L_2(\cD)$-norms ($\sigma=0$) 
	\[
		\min\{ 2\beta - \nicefrac{d}{2}, \, p+1\} 
		\quad 
		\text{and} 
		\quad
		\min\{ 4\beta - \nicefrac{d}{2}, \, p+1 \}
	\]
	of the sinc-Galerkin FE approximation 
	$\widetilde{\GP}_{h,k}^\beta$ and its covariance 
	function~$\widetilde{\varrho}_{h,k}^\beta$
	reflect the higher regularity 
	of the Whittle--Mat\'ern field $\GP^\beta$ 
	in~\eqref{eq:gpbeta} for large 
	$\beta > 0$ in~\eqref{eq:def:beta}.  
	In particular, when the integer part 
	does not vanish, $\nbeta\in\bbN$,  
	a polynomial degree 
	$p>1$ is meaningful, 
	since thus higher order 
	convergence rates can be achieved,   
	cf.\ the numerical experiments 
	in Section~\ref{section:numexp}. 
\end{remark}

%===============================================
\subsubsection{Less regularity}
\label{subsubsec:fem:non-smooth} 
%===============================================

We now discuss convergence of FE discretizations 
when the operator $L$ in~\eqref{eq:def:L}
has a coefficient $A$ which is not necessarily 
Lipschitz continuous 
or the domain $\cD$ is not convex, i.e., 
the general case that 
$L$ is only $H^{1+\alpha}(\cD)$-regular. 
In the following definition 
we specify what we mean by this. 

\begin{definition}\label{def:H1alpha-regular} 
	Suppose Assumptions~\ref{ass:coeff}.I--II, 
	\ref{ass:dom}.I, let 
	$0<\alpha\leq 1$ and 
	$L$ be the second-order differential 
	operator in~\eqref{eq:def:L}. 
	We say that the elliptic problem 
	associated with $L$ 
	is $H^{1+\alpha}(\cD)$-regular 
	if the restriction of $L\from H^1_0(\cD) \to \dual{H^1_0(\cD)}$ 
	to $H^1_0(\cD) \cap H^{1+\alpha}(\cD)$ 
	is a continuous 
	map to %$[\dual{H^1_0(\cD)}, L_2(\cD)]_{\alpha}$ 
	$\Hdot{-1+\alpha}=\dual{(\Hdot{1-\alpha})}$,  
	see~\eqref{eq:def:Hdot}, 
	and if additionally 	
	the data-to-solution map 
	$L^{-1} \from f \mapsto L^{-1}f$  
	is a bounded linear operator from 
	%$[\dual{H^1_0(\cD)}, L_2(\cD)]_{\alpha}$ 
	$\Hdot{-1+\alpha}$
	to $H^1_0(\cD) \cap H^{1+\alpha}(\cD)$. 
\end{definition}

We quote the following extension of the equivalence 
in~\eqref{eq:hdot-sobolev1} of Lemma~\ref{lem:hdot-sobolev}
from~\cite[Prop.~4.1]{BonitoPasciak:2015}
to values $1\leq\sigma\leq1+\alpha$ 
which holds if the 
elliptic problem 
associated with $L$ 
be $H^{1+\alpha}(\cD)$-regular.

\begin{lemma}\label{lem:hdot-sobolev-alpha}
	Let the elliptic problem 
	associated with $L$ 
	be $H^{1+\alpha}(\cD)$-regular, 
	see Definition~\ref{def:H1alpha-regular}. 
	Then the equivalence 
	in~\eqref{eq:ass:1+alpha} 
	holds for this parameter $0<\alpha\leq 1$. 
\end{lemma} 

\begin{lemma}\label{lem:scott-zhang} 
	Suppose Assumptions~\ref{ass:fem}, 
	\ref{ass:dom}.III (i.e., $\cD$ is a Lipschitz polytope)
	and let $p=1$. 
	Then, Assumption~\ref{ass:galerkin}.I  
	is satisfied for 
	$\theta_0 = \nicefrac{1}{2}$ 
	and $\theta_1 = 2$. 
\end{lemma} 

\begin{proof}
	The operator $\cI_h \from H^\theta(\cD) \to V_h$ in  
	Assumption~\ref{ass:galerkin}.I  
	can be taken as the Scott--Zhang interpolant, 
	see~\cite[Lem.~1.130]{ErnGuermond:2004}.  
\end{proof}

\begin{theorem}\label{thm:fem:non-smooth} 
	In addition to Assumptions~\ref{ass:fem}, \ref{ass:dom}.III, 
	suppose that the elliptic  
	problem associated with $L$ 
	is $H^{1+\alpha}(\cD)$-regular 
	for some $0<\alpha\leq 1$ 
	(see Definition~\ref{def:H1alpha-regular})
	and let $p=1$. 
	Assume further that $d\in\{1,2,3\}$, 
	$\beta>0$ and $0\leq\sigma\leq 1$ 
	are such that $2\beta-\sigma > \nicefrac{d}{2}$. 
	Let~$\GP^\beta$ be the Whittle--Mat\'ern 
	field in~\eqref{eq:gpbeta}
	and, for $h, k>0$, 
	let $\GP^\beta_{h,k}$ be the 
	sinc-Galerkin approximation 
	in~\eqref{eq:def:gphk-1}, 
	with covariance 
	functions $\varrho^\beta$ and 
	$\varrho_{h,k}^\beta$.
	Then, for 
	every $q,\eps > 0$ 
	and sufficiently small  
	$h>0$, $k=k(h)>0$,  
	\begin{align*} 
		\Bigl(\bbE\Bigl[
		\bigl\| \GP^{\beta} - \GP^\beta_{h,k} 
		\bigr\|_{H^\sigma(\cD)}^q
		\Bigr]\Bigr)^{\nicefrac{1}{q}} 
		&\lesssim_{(q,\eps,\sigma,\alpha,\beta,A,\kappa,\cD)} 
		h^{\min\left\{ 2\beta - \sigma - \nicefrac{d}{2} - \varepsilon, \, 
			1 + \alpha - \sigma, \, 
			2\alpha \right\} },  
		\\
		\bigl\| \varrho^\beta 
		- \varrho^\beta_{h,k} 
		\bigr\|_{H^{\sigma,\sigma}(\cD\times\cD)} 
		&\lesssim_{(\eps,\sigma,\alpha,\beta,A,\kappa,\cD)} 
		h^{\min\left\{ 4\beta - 2\sigma - \nicefrac{d}{2} - \varepsilon, \, 
			1 + \alpha - \sigma, \, 
			2\alpha \right\} },  
	\end{align*} 
	where, if $d=3$, for 
	\eqref{eq:galerkin:conv:field-alpha} to hold, 
	we also suppose 
	that $\beta>1$
	and $\alpha\geq\nicefrac{1}{2}-\sigma$. 
	
	In addition, if $d=1$ and 
	$0< \gamma \leq \nicefrac{1}{2}$ is such that   
	$2\beta> \gamma + \nicefrac{1}{2}$, 
	then 
	\begin{align*} 
		\left( \bbE 
		\Bigl[ 
		\bigl\| \GP^\beta - \GP^{\beta}_{h,k} \bigr\|_{C^{\delta}(\clos{\cD})}^q 
		\Bigr] 
		\right)^{\nicefrac{1}{q}} 
		&\lesssim_{(q,\gamma,\delta,\eps,\alpha,\beta,A,\kappa,\cD)} 
		h^{\min\left\{ 2\beta-\gamma-\nicefrac{1}{2}-\eps, \, 
			\nicefrac{1}{2}+\alpha-\gamma, \, 
			2\alpha \right\}},   \\
		\sup_{x,y\in\clos{\cD}} 
		\bigl| 
		\varrho^\beta(x,y) 
		-
		\varrho^\beta_{h,k} 
		\bigr| 
		&\lesssim_{(\eps,\alpha,\beta,A,\kappa,\cD)}  
		h^{\min\left\{ 4\beta-1-\eps, \, 
			\nicefrac{1}{2}+\alpha-\eps, \, 
			2\alpha \right\}},  
	\end{align*} 
	for sufficiently small $h>0$, $k=k(h)>0$, 
	every $\delta\in(0,\gamma)$ 
	and $\eps,q > 0$.  
\end{theorem} 

\begin{proof}
	By Lemma~\ref{lem:hdot-sobolev-alpha}
	the equivalence 
	in~\eqref{eq:ass:1+alpha} 
	holds. Furthermore,  
	by Lemma~\ref{lem:scott-zhang}  
	Assumption~\ref{ass:galerkin}.I 
	is satisfied   
	for 
	$\theta_0=\nicefrac{1}{2} < 1$ 
	and $\theta_1=2\geq 1+\alpha$. 
	Finally, since we assume that the 
	family of triangulations $(\cT_h)_{h>0}$ 
	of $\clos{\cD}\subset\bbR^d$ 
	is quasi-uniform, 
	the $L_2(\cD)$-orthogonal 
	projection $\Pi_h$ is $H^1(\cD)$-stable, 
	see \cite{CrouzeixThomee:1987} 
	for $d\in\{1,2\}$
	and \cite{BramblePasciakSteinbach:2002} 
	for arbitrary $d\in\bbN$. 
	Thus, Propositions~\ref{prop:galerkin:conv:sobolev-alpha} 
	and~\ref{prop:galerkin:conv:hoelder} 
	are applicable 
	and yield the assertions of this theorem. 
\end{proof}

%========================================================
\section{Numerical experiments}\label{section:numexp}
%========================================================

In the following numerical experiments 
we consider the original  
Whittle--Mat\'ern field from \eqref{eq:intro:matern} 
in Subsection~\ref{subsec:intro:background}, 
i.e., $L:=-\Delta + \kappa^2$, 
on the unit interval $\cD=(0,1)$, 
augmented with homogeneous Dirichlet boundary conditions. 
We choose $\kappa:=0.5$ and 
apply a finite element discretization   
with continuous, 
piecewise polynomial basis functions 
of degree at most $p\in\{1,2\}$
to compute the sinc-Galerkin approximation 
$\GP^{\beta}_{h,k}$ (or $\widetilde{\GP}^{\beta}_{h,k}$)
in \eqref{eq:def:gphk-1}/\eqref{eq:def:gphk-2}. 
More precisely, we investigate   
\begin{enumerate}
	\item the empirical convergence 
		%of the sinc-Galerkin approximation 
		to the Whittle--Mat\'ern field 
		$\GP^\beta$, see \eqref{eq:gpbeta}, 
		with respect to 
		the norms on $L_2(\Omega; L_2(\cD))$, 
		$L_1(\Omega; L_\infty(\cD))$, 
		and $L_2(\Omega; H^1_0(\cD))$  
		for $\beta\in\{0.5, 0.8, 1.1, 1.4, 1.7\}$; 
	\item the empirical convergence 
		of the covariance function 
		with respect to 
		the norms on $L_2(\cD\times\cD)$  
		and $L_\infty(\cD\times\cD)$  
		for $\beta\in\{0.5, 0.6, 0.7, 0.8, 0.9, 1\}$. 
\end{enumerate} 

To this end, we generate an 
equidistant initial 
mesh on $\clos{\cD}=[0,1]$ with 
$N_0 := 9$ nodes 
(resp.\ $N_0 := 17$ for the $L_\infty$-studies), 
of mesh size $h_0 := 2^{-3}$ 
(resp.\ $h_0 := 2^{-4}$). 
This initial mesh is $4$ times 
uniformly refined, so that 
on level $\ell\in\{0,\ldots,4\}$ 
the mesh is of width  
$h_\ell=h_0 2^{-\ell}$. 
For $p\in\{1,2\}$, 
we use the {\sc MATLAB}-based package 
\texttt{ppfem}~\cite{ppfem}
to assemble the matrices 
$\Mmat$ and $\Lmat$ 
in~\eqref{eq:Zbetavec} and~\eqref{eq:Qmat}
with respect to the Babu\v{s}ka--Shen 
nodal basis $\{\phi_{j,h}\}_{j=1}^{N_h}$. 
% on the finite element 
%space~$V_{h_\ell}^{(p)}$.   
On level $\ell$, the step size $k=k_\ell>0$ of the 
sinc quadrature is calibrated with the 
finite element mesh width via 
$k_\ell = -1/(\beta\ln h_\ell)$.  

The reference solutions for the field 
and the covariance function 
are generated based on an 
overkill Karhunen--Lo\`eve expansion 
of $\GP^\beta$ with $N_{\mathrm{KL}}=1000$ terms, 
\[
	\GP_{\mathrm{ref}}^\beta 
	:= 
	\sum_{j=1}^{N_{\mathrm{KL}}} \xi_j \lambda_j^{-\beta} e_j  
	\qquad
	\text{and} 
	\qquad  
	\varrho_{\mathrm{ref}}^\beta(x,y) 
	:= 
	\sum_{j=1}^{N_{\mathrm{KL}}}  
	\lambda_j^{-2\beta} e_j(x)e_j(y), 
\]
where $\lambda_j= j^2 \pi^2 + \kappa^2$ 
and $e_j(x)=\sqrt{2} \sin(j\pi x)$ 
are the eigenvalues and eigenfunctions 
of $L=-\Delta+\kappa^2$ on $\cD=(0,1)$. 
Here, for each of $100$ Monte Carlo runs,  
the same realization 
of the set of random variables  
$\{\xi_1,\ldots,\xi_{N_{\mathrm{KL}}}\}$ 
is used to generate~$\GP_{\mathrm{ref}}^\beta$ 
and the load vector 
$\bvec\sim\cN(\zerovec,\Mmat)$ via 
\[
	\bvec := \mathbf{R}  
		\left( \xi_1, \ldots, \xi_{N_h} \right)^\T , 
	\qquad
	\text{where} 
	\qquad  
	R_{ij} := \scalar{\phi_{i,h}, e_{j,h}}{L_2(\cD)}. 
\] 
For $d=1$, the operator $L$ does not have multiple eigenvalues 
and we can assemble the matrix~$\mathbf{R}$, 
for each $h\in\{h_0,\ldots,h_4\}$,   
by computing the discrete 
eigenfunctions~$\{e_{j,h}\}_{j=1}^{N_h}$ 
and by adjusting their sign so that $e_{j,h}$   
indeed approximates $e_j$ for each $j\in\{1,\ldots,N_h\}$. 
Note that we only have to assemble this matrix $\mathbf{R}$ 
to have comparable 
samples of the sinc-Galerkin approximation 
and the reference solution 
needed for the strong error studies. 
For the simulation practice, one could  
compute the Cholesky factor of the Gramian $\Mmat$ 
or approximate the matrix 
square root~$\sqrt{\Mmat}$, e.g., 
as proposed in~\cite{Hale:2008}, 
in order to sample from $\bvec$. 
Since furthermore the dimension of 
the finite element spaces, even 
at the highest level $\ell=4$,  
is relatively small, we can assemble the covariance 
matrices of the sinc-Galerkin approximation 
directly, without Monte Carlo sampling, as 
\[
	\operatorname{Cov}\bigl(\GPvec^\beta_k \bigr)
	= 
	\begin{cases}
	\Lmat^{-1} \bigl( \Mmat\Lmat^{-1} \bigr)^{\nbeta-1} 
	\Mmat 
	\bigl( \Mmat\Lmat^{-1} \bigr)^{\nbeta-1} \Lmat^{-1} , 
	& 
	\text{if } 
	\betafrac = 0, \\
	\Qmat_{k}^{\beta_\star} 
	\bigl( \Mmat\Lmat^{-1} \bigr)^{\nbeta} 
	\Mmat
	\bigl( \Mmat\Lmat^{-1} \bigr)^{\nbeta} 
	\Qmat_{k}^{\beta_\star}, 
	& 
	\text{if } 
	\betafrac \in (0,1),   
	\end{cases} 
\]
cf.~\eqref{eq:Zbetavec}--\eqref{eq:Qmat}.  

Note that the operator $L:=-\Delta+0.25$ has constant 
(and, thus, smooth) coefficients. 
Therefore, Theorem~\ref{thm:fem:smooth:sobolev} 
provides (essentially) 
optimal convergence rates for 
the error of $\widetilde{\GP}_{h,k}^\beta$ in 
$L_2(\Omega;L_2(\cD))$, $L_2(\Omega;H^1_0(\cD))$ and 
of $\widetilde{\varrho}_{h,k}^\beta$ in 
$L_2(\cD\times\cD)$. 
Furthermore, the convergence results 
of Theorem~\ref{thm:fem:non-smooth} 
on the $L_1(\Omega;L_{\infty}(\cD))$-error  
are (essentially) sharp 
if $\beta\in(\nicefrac{1}{4},1)$ 
(resp.\ if $\beta\in(\nicefrac{1}{4},\nicefrac{5}{8})$ 
for the $L_{\infty}$-error of the covariance). 
For this smooth case, we have $\alpha > p+1$ 
in \eqref{eq:ass:1+alpha}.    
For this reason, we 
expect the convergence rates listed in 
Table~\ref{tab:rates-expected}. 
The expected rates corresponding to the values 
of $\beta>\nicefrac{1}{4}$ used in our experiments 
are shown in parentheses in Table~\ref{tab:rates}.

\begin{table}[t]
	\centering
	\caption{\label{tab:rates-expected}Expected 
		rates of convergence, cf.~Theorems~\ref{thm:fem:smooth:sobolev} 
		and~\ref{thm:fem:non-smooth}}
	\begin{tabular}{l|ccc}
		\toprule[1pt] 
		                      & $L_2$ & $L_\infty$ & $H^1_0$ \\ 
		\cmidrule(r){2-4}
		$\GP_{h,k}^\beta$     & $\min\left\{ 2\beta - \nicefrac{1}{2},\, p+1 \right\}$ & 
								$\min\left\{ 2\beta - \nicefrac{1}{2},\, p+1 \right\}$ &
								$\min\left\{ 2\beta - \nicefrac{3}{2},\, p \right\}$ \\
		$\varrho_{h,k}^\beta$ & $\min\left\{ 4\beta - \nicefrac{1}{2},\, p+1 \right\}$ & 
								$\min\left\{ 4\beta - 1 ,\, p+1 \right\}$ & 
								$\min\left\{ 4\beta - \nicefrac{5}{2},\, p \right\}$ \\
		\bottomrule[1pt] 
	\end{tabular}
\end{table}

\begin{table}[t]
	\centering
	\caption{\label{tab:rates}Observed (resp.~theoretical) 
		rates of convergence for the errors of the 
		field/covariance  
		shown in Figures~\ref{fig:errors:field}--\ref{fig:errors:cov}}
	\begin{tabular}{lc|cccccc}
		\toprule[1pt] 
		&    & \multicolumn{6}{c}{$\beta$ for field error studies}\\
		& $p$ & $0.5$ & $0.8$ & $1.1$ & $1.4$ & $1.7$ &  \\
		\cmidrule(r){2-8}
		\multirow{2}{*}{$L_2$} 
		& $1$ & 0.54 (0.5) & 1.10 (1.1) & 1.67 (1.7) & 1.94 (2)   & 1.96 (2) & \\
		& $2$ & 0.56 (0.5) & 1.10 (1.1) & 1.68 (1.7) & 2.27 (2.3) & 2.85 (2.9) & \\
		\cmidrule(r){1-8}
		\multirow{2}{*}{$L_\infty$}	
		& $1$ & 0.55 (0.5) & 1.05 (1.1) & 1.60 (1.7) & 1.93 (2)   & 1.99 (2) & \\
		& $2$ & 0.68 (0.5) & 1.14 (1.1) & 1.67 (1.7) & 2.25 (2.3) & 2.79 (2.9) &\\
		\cmidrule(r){1-8}
		\multirow{2}{*}{$H^1_0$}	
		& $1$ & -- & 0.22 (0.1) & 0.70 (0.7) & 1.00 (1)   & 1.05 (1)   & \\
		& $2$ & -- & 0.27 (0.1) & 0.73 (0.7) & 1.30 (1.3) & 1.87 (1.9) & \\
		\midrule[0.75pt] 
		&     & \multicolumn{6}{c}{$\beta$ for covariance error studies}\\
		& $p$ & $0.5$ & $0.6$ & $0.7$ & $0.8$ & $0.9$ & $1$ \\
		\cmidrule(r){2-8}
		\multirow{2}{*}{$L_2$} 
		& $1$ & 1.53 (1.5) & 1.85 (1.9) & 1.98 (2)   & 2.00 (2)   & 2.00 (2) & 2.00 (2) \\
		& $2$ & 1.57 (1.5) & 1.94 (1.9) & 2.32 (2.3) & 2.69 (2.7) & 2.94 (3) & 3.00 (3) \\
		\cmidrule(r){1-8}
		\multirow{2}{*}{$L_\infty$}	
		& $1$ & 1.07 (1) & 1.41 (1.4) & 1.72 (1.8) & 1.91 (2)   & 1.98 (2)   & 1.99 (2) \\
		& $2$ & 1.23 (1) & 1.52 (1.4) & 1.86 (1.8) & 2.23 (2.2) & 2.61 (2.6) & 2.99 (3) \\
		\bottomrule[1pt] 
	\end{tabular}
\end{table}

For every of the $100$ Monte Carlo 
samples, we approximate the integrals 
needed for computing the $L_2(\cD)$ and 
$H^1_0(\cD)$-errors 
by using {\sc MATLAB}'s built-in 
function \texttt{integral} 
with tolerance 1e-6. 
For the $L_\infty$-studies 
we consider the largest error with respect to 
an equidistant mesh on on $\clos{\cD}=[0,1]$ 
with $N_{\operatorname{ok}} = 1001$ nodes, i.e., 
\begin{align*} 
	\sup_{x\in\clos{\cD}} 
	\bigl| 
	\GP^\beta_{h,k}(x) - \GP^\beta_{\operatorname{ref}}(x) 
	\bigr| 
	&\approx 
	\sup_{1\leq j\leq N_{\operatorname{ok}}} 
	\bigl| 
	\GP^\beta_{h,k}(x_j) - \GP^\beta_{\operatorname{ref}}(x_j) 
	\bigr|,    
	\\ 
	\sup_{x,y\in\clos{\cD}} 
	\bigl| 
	\varrho_{h,k}^\beta(x,y) - \varrho^\beta_{\operatorname{ref}}(x,y)
	\bigr| 
	&\approx 
	\sup_{1\leq i,j\leq N_{\operatorname{ok}}}  
	\bigl| 
	\varrho_{h,k}^\beta(x_i,x_j) - \varrho^\beta_{\operatorname{ref}}(x_i,x_j) 
	\bigr|,  
\end{align*}
where $x_j := (j-1) 10^{-3}$.  
%for $j\in\{1,\ldots,N_{\operatorname{ok}}\}$.
%
Furthermore, to compute the 
$L_2(\cD\times\cD)$-error, we approximate 
the distance of the covariances 
by a function which is piecewise 
constant on a regular lattice with 
$N_{\operatorname{ok}}^2$ nodes. 
Finally, the empirical convergence rates, also shown 
in Table~\ref{tab:rates}, are obtained 
via a least-squares affine fit with respect to 
the data set 
$\{(\ln h_\ell, \ln \operatorname{err}_\ell) : 2\leq\ell\leq 4\}$. 
Here, $\operatorname{err}_\ell$ denotes the error 
on level $\ell$ with respect to the norm used in the study 
and for the respective value of $\beta$ and $p$. 

\begin{figure}[t]
	\begin{center}
		\begin{minipage}[t]{0.49\linewidth}
			\begin{center}
				\includegraphics[width=\linewidth]{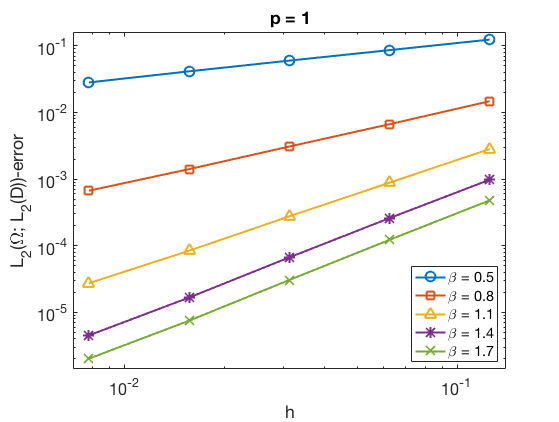}
			\end{center}
		\end{minipage}
		\hfill
		\begin{minipage}[t]{0.49\linewidth}
			\begin{center}
				\includegraphics[width=\linewidth]{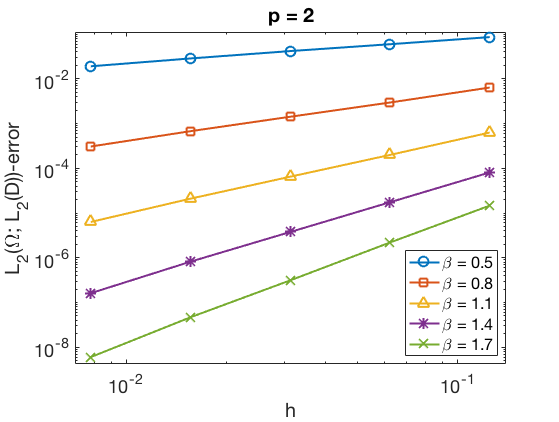}
			\end{center}
		\end{minipage}
		\begin{minipage}[t]{0.49\linewidth}
			\begin{center}
				\includegraphics[width=\linewidth]{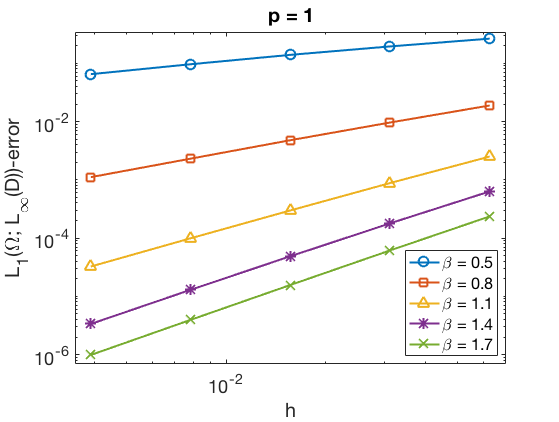}
			\end{center}
		\end{minipage}
		\hfill 
		\begin{minipage}[t]{0.49\linewidth}
			\begin{center}
				\includegraphics[width=\linewidth]{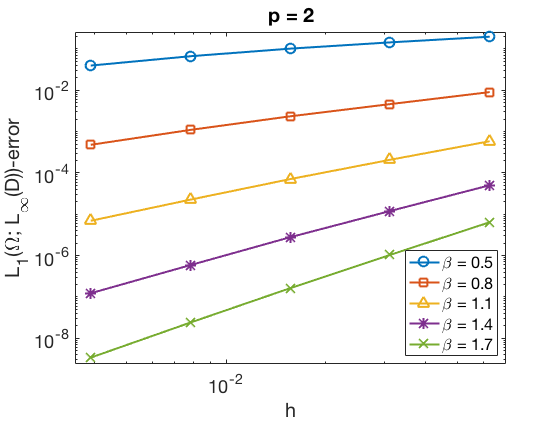}
			\end{center}
		\end{minipage}
		\begin{minipage}[t]{0.49\linewidth}
			\begin{center}
				\includegraphics[width=\linewidth]{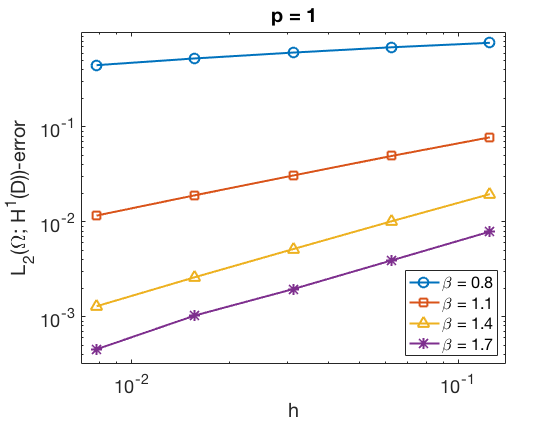}
			\end{center}
		\end{minipage}
		\hfill 
		\begin{minipage}[t]{0.49\linewidth}
			\begin{center}
				\includegraphics[width=\linewidth]{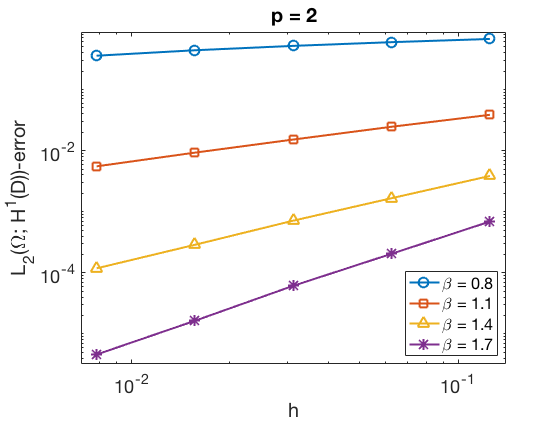}
			\end{center}
		\end{minipage}
	\end{center}
	\caption{Observed errors of the field 
		in $L_2(\Omega; L_2(\cD))$ (top),
		$L_1(\Omega; L_\infty(\cD))$ (middle)
		and $L_2(\Omega; H^1_0(\cD))$ (bottom)
		for polynomial degree $p\in\{1,2\}$ (left, right),
		and different values of $\beta$, 
		shown in a log-log scale as
		a function of the mesh width $h$.
		The corresponding observed convergence rates 
		are shown in Table~\ref{tab:rates}.}
	\label{fig:errors:field}
\end{figure}

\begin{figure}[t]
	\begin{center}
		\begin{minipage}[t]{0.49\linewidth}
			\begin{center}
				\includegraphics[width=\linewidth]{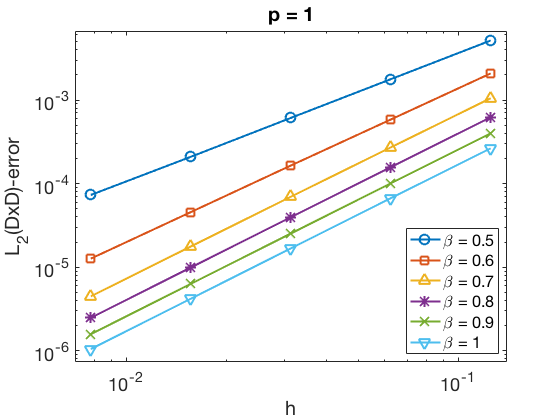}
			\end{center}
		\end{minipage}
		\hfill
		\begin{minipage}[t]{0.49\linewidth}
			\begin{center}
				\includegraphics[width=\linewidth]{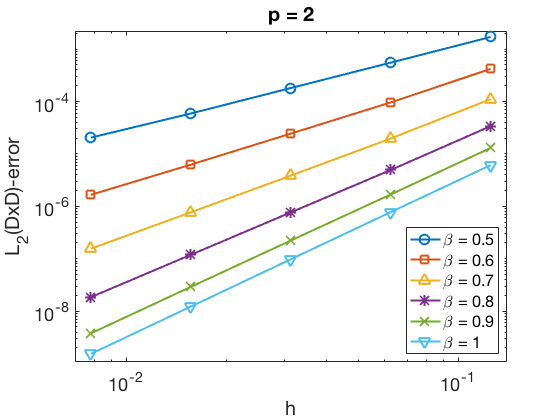}
			\end{center}
		\end{minipage}
		\begin{minipage}[t]{0.49\linewidth}
			\begin{center}
				\includegraphics[width=\linewidth]{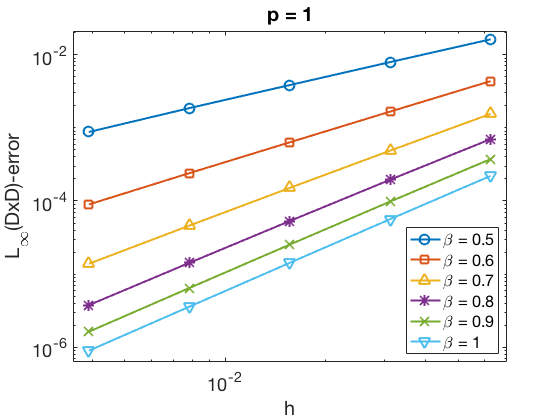}
			\end{center}
		\end{minipage}
		\hfill 
		\begin{minipage}[t]{0.49\linewidth}
			\begin{center}
				\includegraphics[width=\linewidth]{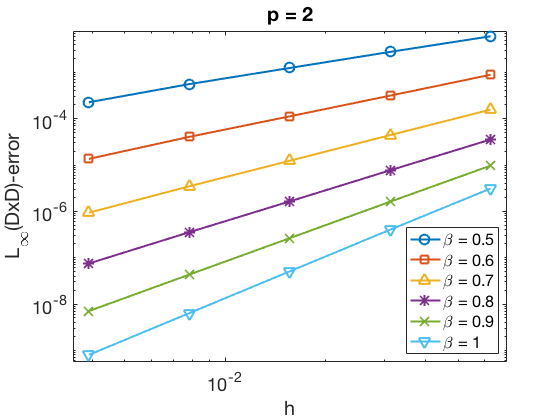}
			\end{center}
		\end{minipage}
	\end{center}
	\caption{Observed $L_q(\cD\times\cD)$-error 
		of the covariance 
		function for $q\in\{2,\infty\}$ (top, bottom),
		polynomial degree $p\in\{1,2\}$ (left, right),
		and different values of $\beta$, 
		shown in a log-log scale as
		a function of the mesh width $h$. 
		The corresponding observed convergence rates 
		are shown in Table~\ref{tab:rates}.}
	\label{fig:errors:cov}
\end{figure}

The resulting observed errors 
are displayed in Figure~\ref{fig:errors:field} 
for the fields and in Figure~\ref{fig:errors:cov}  
for the covariances. 
Overall, the empirical results 
validate our theoretical outcomes 
fairly well, with a slight deviation 
for the $L_{\infty}$-studies 
which may be caused by a larger 
pre-asymptotic range. 

%========================================================
\section{Conclusion and discussion}\label{section:conclusion}
%========================================================

We have identified necessary and sufficient 
conditions for  
square-integrability,  
Sobolev regularity, and H\"older continuity 
(in $L_q(\Omega)$-sense) 
for GRFs in terms of their \emph{color}, 
as well as square-integrability, 
mixed Sobolev regularity, 
and continuity of their  
covariance functions, 
see Propositions~\ref{prop:regularity:hoelder}, 
\ref{prop:regularity:Linf} 
and~\ref{prop:regularity:hdot}. 
Subsequently, we have applied these 
findings to \emph{generalized Whittle--Mat\'ern}  
fields, see~$\GP^\beta$ in \eqref{eq:gpbeta}, 
where these conditions 
become assumptions on the smoothness 
parameter $\beta>0$, corresponding 
to the fractional exponent 
of the color $L^{-\beta}$, 
see Lemmata~\ref{lem:matern:hdot}--\ref{lem:matern:Linf}. 

While these regularity results readily  
implied convergence of \emph{spectral Galerkin}  
approximations, 
see 
Corollaries~\ref{cor:spectral:conv:hdot}--\ref{cor:spectral:conv:hoelder},
significantly more work was needed 
to derive convergence 
for general Galerkin 
(such as finite element) 
approximations, 
for the following reason: 
It was unknown, how
the deterministic fractional   
Galerkin error 
$L^{-\beta}g - L_h^{-\beta}g$ behaves 
in the Sobolev space $H^\sigma(\cD)$,   
for $0 \leq\sigma\leq 1$,  
all possible 
exponents $\beta > 0$, and 
sources $g\in H^{\delta}(\cD)$ 
of possibly negative regularity $\delta<0$. 
We have identified this behavior 
in Theorem~\ref{thm:fem-error} 
for the general situation that the 
second-order elliptic differential operator $L$ 
is $H^{1+\alpha}(\cD)$-regular for some 
$0<\alpha\leq 1$.  
This result could be exploited 
to show convergence of the 
sinc-Galerkin approximations 
and their covariances  
to the Whittle--Mat\'ern field $\GP^\beta$ 
and to its covariance function $\varrho^\beta$, respectively, 
see Theorems~\ref{thm:fem:smooth:sobolev} and~\ref{thm:fem:non-smooth}. 

The fact that 
the Rayleigh--Ritz projection 
and, thus, the deterministic Galerkin 
error $L^{-1} g - L_h^{-1}g$  
converges at the rate  
$\min\{1+\alpha-\sigma, \, 2\alpha\}$ 
in $H^\sigma(\cD)$, $0\leq\sigma\leq 1$, 
if $L$ is $H^{1+\alpha}(\cD)$-regular,  
cf.~Lemma~\ref{lem:fem-error-integer},  
and at the rate 
$p+1-\sigma$ if the problem is ``smooth'' 
and a conforming finite element discretization 
with piecewise polynomial basis 
functions of degree at most $p\in\bbN$ is used, 
combined with the low regularity 
of white noise in $\Hdot{-\nicefrac{d}{2}-\eps}$,  
show that the Sobolev 
convergence rates of
Theorems~\ref{thm:fem:smooth:sobolev} 
and~\ref{thm:fem:non-smooth}
are (essentially, up to $\eps>0$) 
optimal. 
In addition, we believe that our 
results on H\"older convergence of the field 
and on $L_\infty$-convergence 
of the covariance function 
for $d=1$  
in Theorem~\ref{thm:fem:non-smooth}
are optimal 
\begin{enumerate*} 
\item if the problem is only 
	$H^{1+\alpha}(\cD)$-regular for 
	$\alpha\in(0,1)$ maximal, 
	or  
\item if the problem is smooth 
	and 
	$\beta\in(\nicefrac{1}{4}, 1)$ 
	(resp.\ $\beta\in(\nicefrac{1}{4}, \nicefrac{5}{8})$ 
	for the covariance). 
\end{enumerate*} 
However, the 
deterministic 
$p$-FEM 
$L_\infty$-rate 
for $d=1$ 
is known to be $p+1$
if the problem is smooth, 
see~\cite{Douglas:1975}.  
Thus, 
our results will not be sharp 
in this case, 
see also our numerical experiments 
in Section~\ref{section:numexp}. 

Since the approach  
on deriving optimal 
$L_\infty$-rates   
involves non-Hilbertian 
regularity of the solution 
in $W^{p+1,\infty}(\cD)$, 
such a discussion was beyond 
the scope of this article 
and we leave this problem as well 
as the 
$C^{\delta}(\clos{\cD})\,$/$\,L_{\infty}(\cD\times\cD)$ 
error analysis 
of sinc-Galerkin approximations 
in dimension $d\in\{2,3\}$ 
as topics for future research. 

\begin{appendix} 

\section{Proof of Proposition~\ref{prop:galerkin:conv:sobolev-p}} 
\label{app:proofs}

The following lemma will be 
the main tool 
for the derivation 
of Proposition~\ref{prop:galerkin:conv:sobolev-p}. 

\begin{lemma}\label{lem:galerkin:EVh-tilde} 
	Suppose Assumptions~\ref{ass:coeff}.I--II 
	and~\ref{ass:galerkin}.III. 
	Let Assumption~\ref{ass:galerkin}.IV be  
	fulfilled with parameters 
	$r,s_0,t>0$ such that $\nicefrac{r}{2} \geq t-1$ 
	and $s_0 \geq t$. 
	Let $d\in\bbN$, $\beta>0$, and $\rho_0(\,\cdot)$, 
	$\rho_1(\,\cdot\,)$ be as in 
	\eqref{eq:def:rho-01}, i.e., 	 
	\begin{align*}
	\rho_0(\tau) 
	&:= 
	\min\left\{ r, \,  
	s_0, \, 
	2\beta + \tau -\nicefrac{d}{2}  
	\right\}, 
	\quad 
	\rho_1(\tau) 
	:= 
	\min\left\{ \nicefrac{r}{2}, \,  
	s_0, \, 
	2\beta  - 1 + \tau -\nicefrac{d}{2}  
	\right\},  
	\end{align*} 
	and define the exception set 
	\[
	\cE_\tau  
	:= 
	\left\{ 
	2(t-1) - 2\beta + \sigma+\nicefrac{d}{2}, \,
	t - 2\beta + \sigma + \nicefrac{d}{2} 
	: \sigma\in\{0,1\}
	\right\}.
	\]
	Then, for $\sigma\in\{0,1\}$, 
	the Galerkin error 
	$\widetilde{E}_{V_h}^\beta$ 
	in~\eqref{eq:def:err-galerkin} 
	satisfies 
	\begin{align}\label{eq:lem:galerkin:EVh-tilde} 
	\sum_{j\in\bbN} \lambda_j^{-\tau} 
	\bigl\| \widetilde{E}_{V_h}^\beta e_j \bigr\|_{\sigma}^2 
	\lesssim_{(C_0,C_\lambda,\sigma,\tau,\beta,A,\kappa,\cD)}
	C_{\tau, h} 
	h^{2 \rho_\sigma(\tau)} 
	\qquad 
	\forall\tau\geq 0, 
	\end{align}
	for sufficiently small $h>0$.  
	Here, 
	$\{(\lambda_j, e_j)\}_{j\in\bbN}$
	are the $L_2(\cD)$-orthonormal, ordered eigenpairs of 
	$L$ in~\eqref{eq:def:L} and 
	we set $C_{\tau,h} := 1$ if $\tau\not\in\cE_\tau$ and 
	$C_{\tau,h} := \ln(1/h)$ if $\tau\in\cE_\tau$.  
\end{lemma}

\begin{proof}
	Fix $\tau\geq 0$. 
	The definitions  
	of~$\widetilde{E}_{V_h}^\beta$ 
	in~\eqref{eq:def:err-galerkin}
	and of~$\pitilde$ in~\eqref{eq:def:pitilde} 
	yield 
	\begin{align} 
	\sum_{j\in\bbN} \lambda_j^{-\tau} 
	\bigl\| \widetilde{E}_{V_h}^\beta e_j \bigr\|_{\sigma}^2 
	&=
	\sum_{j=1}^{N_h} 
	\lambda_j^{-\tau} 
	\bigl\| \lambda_j^{-\beta} e_j 
	- \lambda_{j,h}^{-\beta} e_{j,h} \bigr\|_{\sigma}^2 
	\notag \\
	&\lesssim 
	\sum_{j=1}^{N_h} 
	\lambda_j^{-\tau + \sigma}
	\bigl| \lambda_j^{-\beta}  
	- \lambda_{j,h}^{-\beta} \bigr|^2	
	+ 
	\sum_{j=1}^{N_h} 
	\lambda_j^{-\tau} \lambda_{j,h}^{-2\beta} 
	\norm{e_j - e_{j,h} }{\sigma}^2. 
	\label{eq:proof:lem:galerkin:EVh-tilde-1} 
	\end{align} 
	By the mean value 
	theorem, 
	$ \lambda_j^{-\beta}  
	- \lambda_{j,h}^{-\beta} 	
	=
	\widetilde{\lambda}_j^{-\beta-1} 
	( \lambda_{j,h} 
	- \lambda_{j} )$  
	for some 
	$\widetilde{\lambda}_j \in (\lambda_j, \lambda_{j,h})$.
	Thus, we can use~\eqref{eq:ass:galerkin:lambda} 
	from Assumption~\ref{ass:galerkin}.IV 
	and the spectral behavior~\eqref{eq:lem:spectral-behav} 
	from Lemma~\ref{lem:spectral-behav}  
	combined with Assumption~\ref{ass:galerkin}.III 
	to bound the first sum 
	in~\eqref{eq:proof:lem:galerkin:EVh-tilde-1}, 
	\begin{align}
	\sum_{j=1}^{N_h} 
	\lambda_j^{-\tau + \sigma}
	\bigl| \lambda_j^{-\beta}  
	- \lambda_{j,h}^{-\beta} \bigr|^2	
	&\leq C_\lambda^2 
	h^{2r} 	\sum_{j=1}^{N_h} 
	\lambda_j^{-2\beta - \tau + \sigma + 2(t-1)} 
	\notag \\
	&\lesssim_{(C_\lambda,\sigma,\tau,\beta,A,\kappa,\cD)} 
	C_{\tau,h} 
	h^{2 \min\left\{r, \, 
		2\beta - \sigma + \tau - \nicefrac{d}{2} 
		\right\} }. 
	\label{eq:proof:lem:galerkin:EVh-tilde-2} 
	\end{align} 
	where we also have used that 
	$r\geq 2(t-1)$ by assumption. 
	For the second sum 
	in \eqref{eq:proof:lem:galerkin:EVh-tilde-1} 
	we distinguish the cases $\sigma=0$ and $\sigma=1$. 
	If $\sigma=0$, we can apply~\eqref{eq:ass:galerkin:e} 
	of Assumption~\ref{ass:galerkin}.IV 
	and obtain 
	\begin{align} 
	\sum_{j=1}^{N_h} 
	\lambda_j^{-\tau} \lambda_{j,h}^{-2\beta} 
	\norm{e_j - e_{j,h} }{0}^2 
	&\leq C_0   
	h^{2s_0} 
	\sum_{j=1}^{N_h} 
	\lambda_j^{-2\beta - \tau + t} 
	\notag \\
	&\lesssim_{(C_0,\sigma,\tau,\beta,A,\kappa,\cD)}  
	C_{\tau,h} 
	h^{2 \min\left\{ s_0, \,  
		2\beta + \tau - \nicefrac{d}{2} 
		\right\} }, 
	\label{eq:proof:lem:galerkin:EVh-tilde-3} 
	\end{align}
	since $s_0\geq t$. 
	For $\sigma=1$, we first note that 
	\eqref{eq:ass:galerkin:lambda}--\eqref{eq:ass:galerkin:e} 
	of Assumption~\ref{ass:galerkin}.IV 
	imply the following estimate 
	with respect to the norm 
	on~$\Hdot{1}$, 
	\[
	\norm{e_j-e_{j,h}}{1}^2 
	= 
	\lambda_j \norm{e_j - e_{j,h}}{0}^2 
	+ 
	\lambda_{j,h} - \lambda_j 
	\leq 
	C_0 h^{2s_0} \lambda_j^{t+1} 
	+ 
	C_\lambda h^r \lambda_j^t.
	\]
	Here, we have used the identity 
	$\scalar{e_j, e_{j,h}}{1} = \lambda_j\scalar{e_j, e_{j,h}}{0}$. 
	Thus, if $\sigma=1$, we can 
	bound the second sum 
	in~\eqref{eq:proof:lem:galerkin:EVh-tilde-1}
	as follows, 
	\begin{align} 
	\sum_{j=1}^{N_h} 
	\lambda_j^{-\tau} \lambda_{j,h}^{-2\beta} 
	&\norm{e_j - e_{j,h} }{1}^2 
	\leq C_0   
	h^{2s_0} 
	\sum_{j=1}^{N_h} 
	\lambda_j^{-2\beta - \tau + t + 1} 
	+ 
	C_\lambda 
	h^{r} 
	\sum_{j=1}^{N_h} 
	\lambda_j^{-2\beta - \tau + t}
	\notag \\
	&\lesssim_{(C_0,C_\lambda,\sigma,\tau,\beta,A,\kappa,\cD)}  
	C_{\tau,h} 
	h^{2 \min\left\{ \nicefrac{r}{2}, \, s_0, \, 
		2\beta - 1 + \tau - \nicefrac{d}{2} 
		\right\} }, 
	\label{eq:proof:lem:galerkin:EVh-tilde-4} 
	\end{align}
	since $s_0\geq t$ and $\nicefrac{r}{2}\geq t-1$ 
	by assumption.  
	Combining~\eqref{eq:proof:lem:galerkin:EVh-tilde-1}, 
	\eqref{eq:proof:lem:galerkin:EVh-tilde-2}, 
	\eqref{eq:proof:lem:galerkin:EVh-tilde-3}
	and~\eqref{eq:proof:lem:galerkin:EVh-tilde-4}
	completes the proof.  
\end{proof}

\begin{proof}[Proof of Proposition~\ref{prop:galerkin:conv:sobolev-p}] 
	In order to derive~\eqref{eq:galerkin:conv:hdot}, 
	we start with splitting the error 
	in the norm $\norm{\,\cdot\,}{\sigma}$ 
	on $\Hdot{\sigma}$, cf.~\eqref{eq:def:Hdot}, 
	which by~\eqref{eq:hdot-sobolev1}
	of Lemma~\ref{lem:hdot-sobolev} 
	implies an upper bound for the 
	Sobolev norm: 
	\begin{align*} 
	\left(\bbE\left[
	\bigl\| \GP^{\beta} - \widetilde{\GP}^\beta_{h,k} 
	\bigr\|_{\sigma}^q 
	\right]\right)^{\nicefrac{1}{q}} 
	&\leq 
	\left(\bbE\left[
	\bigl\| \GP^{\beta} - \GP^\beta_{N_h} 
	\bigr\|_{\sigma}^q 
	\right]\right)^{\nicefrac{1}{q}} 
	+ 
	\left(\bbE\left[
	\bigl\| \GP^{\beta}_{N_h} - \widetilde{\GP}^\beta_{h} 
	\bigr\|_{\sigma}^q 
	\right]\right)^{\nicefrac{1}{q}} \\
	&\quad + 
	\left(\bbE\left[
	\bigl\| \widetilde{\GP}^{\beta}_h 
	- \widetilde{\GP}^\beta_{h,k} 
	\bigr\|_{\sigma}^q 
	\right]\right)^{\nicefrac{1}{q}} 
	=: \text{($\text{A}_{\GP}$)} + \text{($\text{B}_{\GP}$)} + \text{($\text{C}_{\GP}$)}. 
	\end{align*} 
	Here, $\GP^\beta_{N_h}$ is the spectral Galerkin 
	approximation from~\eqref{eq:def:gpbetaN} and 
	$\widetilde{\GP}^\beta_{h}$ denotes a GRF colored by 
	$L_h^{-\beta}\pitilde$. 
	We readily obtain a bound for ($\text{A}_{\GP}$) 
	from~\eqref{eq:spectral:conv:hdot} 
	of Corollary~\ref{cor:spectral:conv:hdot}, 
	combined with Assumption~\ref{ass:galerkin}.III.  
	This gives  
	\[
	\text{($\text{A}_{\GP}$)} 
	\lesssim_{(q,\sigma,\beta,A,\kappa,\cD)} 
	N_h^{-\nicefrac{1}{d} \, \left( 2\beta-\sigma-\nicefrac{d}{2} \right)} 
	\lesssim_{(q,\sigma,\beta,A,\kappa,\cD)} 
	h^{2\beta-\sigma-\nicefrac{d}{2}}. 
	\]
	Note that 
	it suffices to 
	estimate the terms ($\text{B}_{\GP}$) and ($\text{C}_{\GP}$)   
	for $\sigma\in\{0,1\}$. 
	The respective bounds for $\sigma\in(0,1)$ 
	then follow by interpolation. 
	By definition  
	of the Galerkin and the 
	quadrature error, 
	$\widetilde{E}_{V_h}^\beta$, $\widetilde{E}_{Q}^\beta$,  
	in~\eqref{eq:def:err-galerkin}--\eqref{eq:def:err-quad} 
	and by Proposition~\ref{prop:regularity:hdot}, 
	\begin{align*} 
	\text{($\text{B}_{\GP}$)}  
	\lesssim_{q} 
	\bigl\| 
	\widetilde{E}_{V_h}^\beta 
	\bigr\|_{\cL_2^{0;\sigma}}
	\quad 
	\text{and} 
	\quad  
	\text{($\text{C}_{\GP}$)} 
	\lesssim_{q} 
	\bigl\|
	\widetilde{E}_{Q}^\beta 
	\bigr\|_{\cL_2^{0;\sigma}}, 
	\qquad 
	\cL_2^{\theta;\sigma}
	:= 
	\cL_2\bigl(\Hdot{\theta}; \Hdot{\sigma}\bigr). 
	\end{align*}
	Since we have to consider these terms 
	only for $\sigma\in\{0,1\}$, 
	the first 
	term can be bounded 
	by~\eqref{eq:lem:galerkin:EVh-tilde} 
	of Lemma~\ref{lem:galerkin:EVh-tilde}
	(with $\tau := 0$),  
	\begin{align*} 
	\text{($\text{B}_{\GP}$)}^2 
	&\lesssim_q 
	\bigl\| 
	\widetilde{E}_{V_h}^\beta 
	\bigr\|_{\cL_2^{0;\sigma}}^2
	= 
	\sum_{j\in\bbN} 
	\bigl\| \widetilde{E}_{V_h}^\beta e_j \bigr\|_{\sigma}^2 
	\lesssim_{(C_0,C_\lambda,\sigma,\beta,A,\kappa,\cD)} 
	\left( C_{\beta,h}^{\GP} \right)^2
	h^{ 2 \rho_\sigma(0) }. 
	\end{align*} 
	where $C_{\beta,h}^{\GP}>0$ 
	is defined as in the statement 
	of Proposition~\ref{prop:galerkin:conv:sobolev-p}. 
	To estimate ($\text{C}_{\GP}$),  
	we first apply 
	the convergence 
	result of the sinc quadrature 
	from~\cite[Lem.~3.4, Rem.~3.1, Thm.~3.5]{BonitoPasciak:2015}. 
	Thus, for sufficiently small $k>0$ and 
	all $1\leq j\leq N_h$, 
	\[
	\bigl\| \widetilde{E}_Q^\beta e_j \bigr\|_{L_2(\cD)} 
	= 
	\bigl\| \bigl(L_h^{-\betafrac} - Q_{h,k}^{\betafrac} \bigr) 
	L_h^{-\nbeta} e_{j,h} \bigr\|_{L_2(\cD)} 
	\lesssim_{(\beta,A,\kappa,\cD)}
	e^{-\nicefrac{\pi^2}{(2k)}} 
	\lambda_{j,h}^{-\nbeta}.  
	\]
	Again by 
	equivalence of the norms 
	$\norm{\,\cdot\,}{\sigma}$, 
	$\norm{\,\cdot\,}{H^\sigma(\cD)}$ 
	for $\sigma\in\{0,1\}$, 
	see Lemma~\ref{lem:hdot-sobolev}, 
	and by the inverse 
	inequality~\eqref{eq:ass:galerkin:inverse}
	from Assumption~\ref{ass:galerkin}.II, 
	we then find 
	\begin{align*} 
	\text{($\text{C}_{\GP}$)}^2 
	&\lesssim_{q} 
	\bigl\|
	\widetilde{E}_{Q}^\beta 
	\bigr\|_{\cL_2^{0;\sigma}}^2 
	= 
	\sum_{j=1}^{N_h} 
	\bigl\| \widetilde{E}_{Q}^\beta e_j \bigr\|_{\sigma}^2 
	\lesssim_{(\sigma,A,\kappa,\cD)} 
	h^{-2\sigma}
	\sum_{j=1}^{N_h} 
	\bigl\| \widetilde{E}_{Q}^\beta e_j \bigr\|_{L_2(\cD)}^2 \\ 
	&\lesssim_{(q,\sigma,\beta,A,\kappa,\cD)} 
	e^{-\nicefrac{\pi^2}{k}} 
	h^{-2\sigma} 
	\sum_{j=1}^{N_h} 
	\lambda_{j,h}^{-2\nbeta} 
	\lesssim_{(q,\sigma,\beta,A,\kappa,\cD)} 
	e^{-\nicefrac{\pi^2}{k}} 
	h^{-2\sigma - 
		d\,\mathds{1}_{\{\beta<1\}} }, 
	\end{align*}
	where we have used the spectral 
	behavior~\eqref{eq:lem:spectral-behav} 
	from Lemma~\ref{lem:spectral-behav} 
	and Assumptions~\ref{ass:galerkin}.III-IV 
	in the last step. 
	This completes the proof of~\eqref{eq:galerkin:conv:hdot}.  
	
	We now proceed with 
	the derivation of~\eqref{eq:galerkin:conv:cov}. 
	To this end, 
	we consider the error 
	with respect to the norm 
	$\norm{\,\cdot\,}{\sigma,\sigma}$, 
	see~\eqref{eq:def:Hdotmixed}, 
	since the embedding in~\eqref{eq:hdot-sobolev1} 
	implies that 
	$\Hdot{\sigma,\sigma} 
	\hookrightarrow 
	H^{\sigma,\sigma}(\cD\times\cD)$. 
	We again partition the error 	
	in three terms, 
	\begin{align*} 
	\bigl\| \varrho^\beta 
	- \widetilde{\varrho}^\beta_{h,k} 
	\bigr\|_{\sigma,\sigma} 
	&\leq 
	\bigl\| \varrho^\beta 
	- \varrho^\beta_{N_h} 
	\bigr\|_{\sigma,\sigma} 
	+ 
	\bigl\| \varrho^\beta_{N_h} 
	-  \widetilde{\varrho}^\beta_{h} 
	\bigr\|_{\sigma,\sigma} 
	+ 
	\bigl\| \widetilde{\varrho}^\beta_{h} 
	- \widetilde{\varrho}^\beta_{h,k} 
	\bigr\|_{\sigma,\sigma} \\
	&=: 
	\text{($\text{A}_{\varrho}$)} 
	+ \text{($\text{B}_{\varrho}$)} + \text{($\text{C}_{\varrho}$)},  	
	\end{align*}
	where~$\widetilde{\varrho}^\beta_{h}$ 
	denotes the covariance function 
	of the above-introduced 
	GRF~$\widetilde{\GP}^\beta_h$ 
	colored by~$L_h^{-\beta}\pitilde$. 
	A bound for the truncation error 
	is given by~\eqref{eq:spectral:conv:cov} 
	in Proposition~\ref{cor:spectral:conv:hdot}, 
	\[
	\text{($\text{A}_{\varrho}$)}  
	\lesssim_{(\sigma,\beta,A,\kappa,\cD)} 
	N_h^{\nicefrac{1}{d} \, \left(4\beta - 2\sigma - \nicefrac{d}{2} \right)}
	\lesssim_{(\sigma,\beta,A,\kappa,\cD)}
	h^{4\beta - 2\sigma - \nicefrac{d}{2}},  
	\]  
	where we also used 
	Assumption~\ref{ass:galerkin}.\ref{ass:galerkin-iii} 
	We bound the remaining terms 
	($\text{B}_{\varrho}$) and ($\text{C}_{\varrho}$) 
	for $\sigma\in\{0,1\}$. 
	Since  
	$\bigl[\Hdot{0,0}, \Hdot{1,1}\bigr]_{\sigma} = \Hdot{\sigma,\sigma}$, 
	see~\cite[Thm.~16.1]{Yagi:2010}, 
	we may again interpolate 
	these results 
	for $\sigma\in(0,1)$. 
	To this end, we first 
	exploit~\eqref{eq:cov-fct-err} from Proposition~\ref{prop:regularity:hdot} 
	and~\eqref{eq:HS-estimate}
	to derive for ($\text{B}_{\varrho}$) that 
	\begin{align*} 
	\text{($\text{B}_{\varrho}$)}  
	&= 
	\bigl\| L_{N_h}^{-2\beta} - L_h^{-\beta} \pitilde 
	\dual{\bigl(L_h^{-\beta} \pitilde\bigr)} 
	\bigr\|_{\cL_2^{-\sigma;\sigma}} 
	\leq 	
	\bigl\| \widetilde{E}_{V_h}^\beta   
	\dual{\bigl(L_{N_h}^{-\beta} + L_h^{-\beta} \pitilde\bigr)} 
	\bigr\|_{\cL_2^{-\sigma;\sigma}} \\
	&\leq 
	\bigl\| \widetilde{E}_{V_h}^\beta L_{N_h}^{-\beta} 
	\bigr\|_{\cL_2^{-\sigma;\sigma}} 
	+ 
	\bigl\| \widetilde{E}_{V_h}^\beta 
	\dual{\bigl(L_h^{-\beta} \pitilde \bigr)} 
	\bigr\|_{\cL_2^{-\sigma;\sigma}} 
	=: 
	\text{($\text{B}'_{\varrho}$)} + \text{($\text{B}''_{\varrho}$)}. 
	\end{align*}
	By Lemma~\ref{lem:galerkin:EVh-tilde} 
	(for~$\tau := 2\beta - \sigma > \nicefrac{d}{2} > 0$ 
	in~\eqref{eq:lem:galerkin:EVh-tilde}) 
	we have, for $\sigma\in\{0,1\}$ 
	and for $C_{\beta,h}^\varrho>0$ as in 
	the statement of 
	Proposition~\ref{prop:galerkin:conv:sobolev-p},  
	\begin{align*}
	\text{($\text{B}'_{\varrho}$)}^2  
	&= 
	\sum_{j\in\bbN} 
	\lambda_j^{-(2\beta-\sigma)} 
	\bigl\| \widetilde{E}_{V_h}^\beta e_j \bigr\|_{\sigma}^2 
	\lesssim_{(C_0,C_\lambda,\sigma,\beta,A,\kappa,\cD)}
	\bigl( C_{\beta,h}^\varrho \bigr)^2  
	h^{2 \rho_\sigma(2\beta - \sigma)}.     
	\end{align*}
	Next, we use the identity 
	$
	\dual{\bigl( L_h^{-\beta}\pitilde \bigr)} e_j 
	= 
	\sum_{\ell=1}^{N_h} 
	\lambda_{\ell,h}^{-\beta} 
	\scalar{e_j,e_{\ell,h}}{L_2(\cD)} \, e_{\ell} 
	$, the orthogonality 
	$\scalar{e_{k,h},e_{\ell,h}}{\sigma} 
	=\delta_{k\ell} \lambda_{k,h}^\sigma$ 
	(here, $\delta_{k\ell}$ denotes the Kronecker delta), 
	which holds for $\sigma\in\{0,1\}$, 
	and the relation 
	$\lambda_j \leq \lambda_{j,h}$ 
	from Assumption~\ref{ass:galerkin}.IV. 
	With these steps, we obtain 
	again by \eqref{eq:lem:galerkin:EVh-tilde} of 
	Lemma~\ref{lem:galerkin:EVh-tilde} 
	(with~$\tau := 2\beta - \sigma$) 
	a bound for ($\text{B}''_{\varrho}$),   
	\begin{align*}
	\text{($\text{B}''_{\varrho}$)}^2  
	&= 
	\sum_{j\in\bbN}  
	\sum_{i=1}^{N_h}  
	\sum_{\ell=1}^{N_h} 
	\lambda_j^{\sigma} 
	\lambda_{i,h}^{-\beta} 
	\lambda_{\ell,h}^{-\beta}
	\scalar{e_j,e_{i,h}}{L_2(\cD)} 
	\scalar{e_j,e_{\ell,h}}{L_2(\cD)} 
	\bigl( \widetilde{E}_{V_h}^\beta e_{i}, 
	\widetilde{E}_{V_h}^\beta e_{\ell} \bigr)_{\sigma} \\
	&= 
	\sum_{i=1}^{N_h} 
	\sum_{\ell=1}^{N_h} 
	\lambda_{i,h}^{-\beta} 
	\lambda_{\ell,h}^{-\beta}
	\scalar{e_{i,h}, e_{\ell,h}}{\sigma} 
	\bigl( \widetilde{E}_{V_h}^\beta e_{i}
	\widetilde{E}_{V_h}^\beta e_{\ell} \bigr)_{\sigma} 
	= 
	\sum_{\ell=1}^{N_h} 
	\lambda_{\ell,h}^{-(2\beta-\sigma)} 
	\bigl\| \widetilde{E}_{V_h}^\beta e_{\ell} \bigr\|_{\sigma}^2 
	\\ 
	&\leq 
	\sum_{\ell\in\bbN} 
	\lambda_{\ell}^{-(2\beta-\sigma)} 
	\bigl\| \widetilde{E}_{V_h}^\beta e_{\ell} \bigr\|_{\sigma}^2
	\lesssim_{(C_0,C_\lambda,\sigma,\beta,A,\kappa,\cD)}
	\bigl( C_{\beta,h}^\varrho \bigr)^2 
	h^{2 \rho_\sigma(2\beta-\sigma)}. 
	\end{align*}
	In conclusion, 
	$\bigl\| 
	\varrho_{N_h}^\beta - \widetilde{\varrho}_h^\beta 
	\bigr\|_{\sigma,\sigma} 
	\leq \text{($\text{B}'_{\varrho}$)} + \text{($\text{B}''_{\varrho}$)}  
	\lesssim_{(C_0,C_\lambda,\sigma,\beta,A,\kappa,\cD)} 
	\bigl( C_{\beta,h}^\varrho \bigr)^2 
	h^{ \rho_\sigma(2\beta-\sigma)}$ 
	for $\sigma\in\{0,1\}$. 
	For~($\text{C}_{\varrho}$), we derive with the 
	equivalence of the norms 
	$\norm{\,\cdot\,}{\sigma}$, 
	$\norm{\,\cdot\,}{H^\sigma(\cD)}$, 
	the inverse inequality~\eqref{eq:ass:galerkin:inverse} 
	from Assumption~\ref{ass:galerkin}.II, 
	and 
	the convergence 
	result for the sinc quadrature 
	\cite[Lem.~3.4, Rem.~3.1, Thm.~3.5]{BonitoPasciak:2015} 
	the following, if $\sigma\in\{0,1\}$, 
	\begin{align*}
	\text{($\text{C}_{\varrho}$)}^2  
	%			&= 
	%			\bigl\| \widetilde{E}_{Q}  
	%			\dual{\bigl(L_h^{-\beta} \pitilde + 
	%				Q_{h,k}^{\betafrac} 
	%				L_h^{-\nbeta} \pitilde\bigr)} 
	%			\bigr\|_{\cL_2^{-\sigma;\sigma}}^2 \\ 
	&\lesssim_{(\sigma,A,\kappa,\cD)} 
	h^{-2\sigma}
	\sum_{j\in\bbN} 
	\lambda_j^\sigma 
	\bigl\| 
	\widetilde{E}_{Q}^\beta   
	\dual{\bigl(L_h^{-\beta} \pitilde + 
		Q_{h,k}^{\betafrac} 
		L_h^{-\nbeta} \pitilde\bigr)} e_j
	\bigr\|_{L_2(\cD)}^2 \\
	&\lesssim_{(\sigma,\beta,A,\kappa,\cD)}  
	e^{-\nicefrac{\pi^2}{k}}
	h^{-2\sigma} 
	\sum_{j\in\bbN} 
	\lambda_j^\sigma 
	\bigl\| 
	L_h^{-\nbeta} \pitilde  
	\dual{\pitilde} 
	\dual{(L_h^{-\nbeta})}  
	\dual{\bigl(L_h^{-\betafrac}  + 
		Q_{h,k}^{\betafrac} 
		\bigr)} e_j 
	\bigr\|_{L_2(\cD)}^2 . 
	\end{align*} 
	Since $L_h^{-\nbeta} \pitilde  
	\dual{\pitilde} 
	\dual{(L_h^{-\nbeta})} e_{\ell,h} 
	= \lambda_{\ell,h}^{-2\nbeta} e_{\ell,h}$ 
	for all $\ell\in\{1,\ldots,N_h\}$,  
	this shows that 
	\begin{align*} 
	\text{($\text{C}_{\varrho}$)}^2  
	&\lesssim_{(\sigma,\beta,A,\kappa,\cD)}  
	e^{-\nicefrac{\pi^2}{k}}
	h^{-2\sigma}  
	\sum_{\ell=1}^{N_h} 
	\sum_{j\in\bbN}
	\lambda_j^\sigma  
	\lambda_{\ell,h}^{-4\nbeta}  
	\bigl( e_j, \bigl(L_h^{-\betafrac}  + 
	Q_{h,k}^{\betafrac} 
	\bigr) e_{\ell,h} \bigr)_{L_2(\cD)}^2 .  
	\end{align*} 
	Next, again by the inverse 
	inequality~\eqref{eq:ass:galerkin:inverse} we find   
	\begin{align*}
	\text{($\text{C}_{\varrho}$)}^2
	&\lesssim_{(\sigma,\beta,A,\kappa,\cD)}   
	e^{-\nicefrac{\pi^2}{k}}
	h^{-4\sigma} 
	\sum_{\ell=1}^{N_h} 
	\lambda_{\ell,h}^{-4\nbeta}  
	\bigl\| \bigl(L_h^{-\betafrac}  + 
	Q_{h,k}^{\betafrac} 
	\bigr) e_{\ell,h} \bigr\|_{L_2(\cD)}^2 \\
	&\lesssim_{(\sigma,\beta,A,\kappa,\cD)}      
	e^{-\nicefrac{\pi^2}{k}} 
	h^{-4\sigma} 
	\sum_{\ell=1}^{N_h} 
	\lambda_{\ell}^{-4\nbeta} 
	\lesssim_{(\sigma,\beta,A,\kappa,\cD)}  
	e^{-\pi^2/k}
	h^{-4\sigma 
		-d\,\mathds{1}_{\{\beta<1\}}}. 
	\end{align*} 
	Here, we have used the 
	uniform stability of $L^{-\betafrac}$, 
	$Q_{h,k}^{\betafrac}$ 
	with respect to $h$ and $k$, 
	see \eqref{eq:uniform-stab}, 
	as well as~\eqref{eq:lem:spectral-behav} 
	from Lemma~\ref{lem:spectral-behav} 
	and Assumption~\ref{ass:galerkin}.\ref{ass:galerkin-i}  
	Combining the bounds for 
	($\text{A}_{\varrho}$), ($\text{B}_{\varrho}$), ($\text{C}_{\varrho}$) 
	completes the proof.  
\end{proof} 

\end{appendix}  

%%=======================================================================================
\bibliographystyle{siam}
\bibliography{ck-bib}
%%=======================================================================================

\end{document}